\documentclass[a4paper,11 pt,twoside]{amsart}

\usepackage[T1]{fontenc}
\usepackage[english]{babel}
\usepackage{times}
\usepackage{amsmath}
\usepackage{amsfonts}
\usepackage{amssymb}
\usepackage{amsmath}
\usepackage{amsthm}
\usepackage{graphicx}
\usepackage{array}
\usepackage{color}
\usepackage{mathrsfs}
\usepackage{hyperref}
\usepackage{esint} 
\usepackage{tikz}
\usepackage{ upgreek }
\usepackage{enumitem}

\definecolor{darkgreen}{rgb}{0.1,0.7,0.1}
\definecolor{darkred}{rgb}{0.7,0.1,0.1}

\newtheorem{theorem}{Theorem}
\newtheorem{lemma}{Lemma}[section]
\newtheorem{proposition}[lemma]{Proposition}
\newtheorem{corollary}[lemma]{Corollary}

\newtheorem{remark}[lemma]{Remark}

\newcommand{\cc}{\complement}

\newcommand{\eps}{\varepsilon}

\newcommand\symb[2][\bf]{{\mathchoice{\hbox{#1#2}}{\hbox{#1#2}}%
        {\hbox{\scriptsize#1#2}}{\hbox{\tiny#1#2}}}}

\def\R{{\symb R}}

\def\Z{{\symb Z}}

\def\Q{{\symb Q}}
\def\P{{\symb P}}

\def\hZ{\hat{Z}}

\renewcommand{\P}{\mathbb{P}}
\newcommand{\E}{\mathbb{E}}

\newcommand{\bP}{\mathbf{P}}
\newcommand{\bQ}{\mathbf{Q}}

\newcommand{\cA}{\mathcal{A}}
\newcommand{\cB}{\mathcal{B}}
\newcommand{\cC}{\mathcal{C}}
\newcommand{\cD}{\mathcal{D}}
\newcommand{\cE}{\mathcal{E}}
\newcommand{\cF}{\mathcal{F}}
\newcommand{\cG}{\mathcal{G}}
\newcommand{\cH}{\mathcal{H}}
\newcommand{\cI}{\mathcal{I}}

\newcommand{\cL}{\mathcal{L}}
\newcommand{\cM}{\mathcal{M}}
\newcommand{\cN}{\mathcal{N}}
\newcommand{\cO}{\mathcal{O}}
\newcommand{\cP}{\mathcal{P}}
\newcommand{\cQ}{\mathcal{Q}}

\newcommand{\cU}{\mathcal{U}}

\hypersetup{
citecolor=blue,
colorlinks=true, 
breaklinks=true, 
urlcolor= blue, 
linkcolor= black, 
bookmarksopen=true, 
pdftitle={SAO}, 
}

\begin{document}

\title[The stochastic Airy operator at large temperature]{The stochastic Airy operator\\at large temperature}

\author{Laure Dumaz}
\address{
Universit\'e Paris-Dauphine, PSL University, UMR 7534, CNRS, CEREMADE, 75016 Paris, France}
\email{dumaz@ceremade.dauphine.fr}

\author{Cyril Labb\'e}
\address{
Universit\'e Paris-Dauphine, PSL University, UMR 7534, CNRS, CEREMADE, 75016 Paris, France}
\email{labbe@ceremade.dauphine.fr}

\vspace{2mm}

\date{\today}

\maketitle

\begin{abstract}

It was shown in [J. A. Ram\'irez, B. Rider and B. Vir\'ag. J. Amer. Math. Soc. {\bf 24} 919-944 (2011)] that the edge of the spectrum of $\beta$ ensembles converges in the large $N$ limit to the bottom of the spectrum of the stochastic Airy operator. In the present paper, we obtain a complete description of the bottom of this spectrum when the temperature $1/\beta$ goes to $\infty$: we show that the point process of appropriately rescaled eigenvalues converges to a Poisson point process on $\R$ of intensity $e^x dx$ and that the eigenfunctions converge to Dirac masses centered at IID points with exponential laws. Furthermore, we obtain a precise description of the microscopic behavior of the eigenfunctions near their localization centers.

\medskip

\end{abstract}

\setcounter{tocdepth}{1}
\tableofcontents

\section{Introduction}

Consider the law of $N$ interacting particles $\mu_1>\ldots>\mu_N$ given by the density:
\begin{equation}\label{Eq:GbE} \frac1{Z_N^\beta} \prod_{i < j} |\mu_i - \mu_j|^\beta e^{-\frac{\beta}{4} \sum_{i=1}^N \mu_i^2}\;,\end{equation}
where $\beta > 0$ is an inverse temperature and $Z_N^\beta$ is a partition function. This law is usually referred to as the (Gaussian) $\beta$-ensemble. In the special cases $\beta=1$, $2$ and $4$, this measure coincides with the law of the eigenvalues of the Gaussian Orthogonal, Unitary and Symplectic ensembles, which are laws of random matrices invariant under conjugation with respectively orthogonal, unitary and symplectic matrices. However, the connection with random matrices is not restricted to these three particular values of $\beta$: Dumitriu and Edelman~\cite{DumEde} showed that for any $\beta >0$, one can build a symmetric, tridiagonal random matrix whose eigenvalues distribution is given by \eqref{Eq:GbE}.\\

The repulsion between particles increases with the parameter $\beta$: in particular, for fixed $N$ and $\beta$ goes to $0$, the particles, multiplied by $\sqrt{\beta}$, converge in law to $N$ IID Gaussian random variables. The behavior of these ensembles when $N$ goes to infinity and the inverse temperature $\beta$ is sent to zero has been the subject of recent works. In~\cite{BGP} the regime where $N$ goes to infinity and $\beta$ goes to $0$ but $N\beta$ remains constant is considered: the local statistics in the bulk of the spectrum are shown to converge to a Poisson point process. In~\cite{DuyNak} an alternative proof of this convergence is presented and the intensity measure of the Poisson point process is given explicitly. Let us also cite the work~\cite{Pakzad} where it is shown that for $N\beta \to 0$ the bottom of the spectrum, properly rescaled, converges to a Poisson point process. 

In the present work, we consider the case where $N$ goes to infinity first, and then $\beta$ is sent to $0$: loosely speaking, we are in the case where $N\beta$ goes to infinity. We prove the convergence of the bottom of the spectrum, properly rescaled, to a Poisson point process and also a localization phenomenon for the corresponding eigenfunctions. We believe that our strategy of proof could be adapted to treat the case where $\beta$ is sent to $0$ slowly enough with $N$.

\medskip

Let us comment on the underlying physical motivations of the model. The invariant ensembles of random matrices were originally introduced to model energy levels of heavy nuclei. For general $\beta >0$, the $\beta$-ensembles can be seen as a Coulomb-gas with logarithmic interaction: the parameter $\beta$ then plays the role of an inverse temperature. As mentioned above, there has been some research activity on the behavior of this gas of particles when the temperature is sent to infinity: in the present article, we focus on the extremal particles and aim at understanding their statistical behavior.

\medskip

The scaling limit of the edge of the $\beta$-ensemble, in the regime where $N$ goes to infinity and $\beta > 0$ is fixed, was obtained by Ram\'{\i}rez, Rider and Vir\'ag~\cite{RamRidVir}. They showed that for any $k\ge 1$, the $k$-dimensional vector $\big(N^{1/6}(2\sqrt N - \mu_i); i=1\ldots k\big)$ converges in distribution to the $k$ lowest eigenvalues of the following random operator called Stochastic Airy Operator (SAO)
\begin{equation}\label{Eq:SAO}
\cA_\beta = -\partial^2_x + x + \frac{2}{\sqrt \beta} \xi\;,\quad x\in(0,\infty)\;,
\end{equation}
endowed with homogeneous Dirichlet boundary condition at $x=0$. The potential $\xi$ appearing in this operator is a white noise on $(0,\infty)$, that is, the derivative in the sense of distributions of a Brownian motion. This operator is self-adjoint in $L^2(0,\infty)$ with pure point spectrum $\mu_1 < \mu_2 < \ldots$ of multiplicity one and the corresponding eigenfunctions $(\psi_k)_{k\ge 1}$, normalized in $L^2(0,\infty)$, are H\"older functions of regularity index $3/2^-$, see~\cite{RamRidVir,Gaudreau}.

Up to rescaling the eigenvalues / eigenfunctions appropriately (see Remark \ref{Rk:Scaling} below), it is equivalent to consider the operator
$$\cL_\beta = -\partial_x^2 + \frac{\beta}{4}x + \xi\;,\quad x\in(0,\infty)\;,$$
endowed with homogeneous Dirichlet boundary condition at $x=0$. For simplicity, we will also call $\cL_\beta$ the Stochastic Airy Operator: this will not cause any confusion in the sequel. We denote by $\lambda_1 < \lambda_2 < \ldots$ its eigenvalues and $(\varphi_k)_{k\ge 1}$ the associated normalized eigenfunctions. The asymptotic behavior of $\cL_\beta$ as $\beta\downarrow 0$ will rely on the deterministic quantity $L=L_\beta$ defined by
\begin{align}\label{def:L}
L_\beta := \frac{1}{\beta \big(\frac38 \ln 1/\beta\big)^{1/3}}\;.
\end{align}
Notice that $L\to \infty$ when $\beta\to 0$, and that $\beta\mapsto L$ is injective on $(0,\beta_0)$ for some $\beta_0>0$. We will also rely on a deterministic function $a_L$, whose precise definition will be given later on (see \eqref{Eq:DefaL}) and whose asymptotic behavior is given by $a_L \sim (3/8 \ln L)^{2/3}$ as $L\to\infty$.\\

In~\cite{AllezDumazTW}, the asymptotic behavior as $\beta\downarrow 0$ of the first eigenvalue $\lambda_1$ of $\cL_\beta$ was studied: using a representation (originally introduced in~\cite{RamRidVir}) of the eigenvalues / eigenfunctions in terms of a family of time-inhomogeneous diffusions, it was shown that $\lambda_1 \sim -a_L$ and that $-4\sqrt{a_L} (\lambda_1 + a_L)$ converges to a Gumbel law. The convergence of the joint law of the smallest eigenvalues towards a Poisson point process was left as a conjecture.


\bigskip

In the present paper, we obtain a complete description of the bottom of the spectrum of $\cL_\beta$ when $\beta\downarrow 0$. We show that the properly rescaled eigenvalues converge to a Poisson point process with explicit intensity, and that the eigenfunctions converge to Dirac masses localized at IID points with exponential distribution. Furthermore, we obtain a precise description of the microscopic behavior of the eigenfunctions near their localization centers.\\

To state precisely our results, we let $U_{k}$ be the first point in $(0,\infty)$ where $|\varphi_k|$ reaches its maximum. We also build probability measures on $(0,\infty)$ from the rescaled eigenfunctions:
$$ m_k(dx) := L \varphi_k^2\big(x L\big) dx\;,\quad x\in (0,\infty)\;.$$
Our first main result is the following.
\begin{theorem}\label{Th:Main}
As $\beta \downarrow 0$, we have the following convergence in law
$$ \Big(4\sqrt{a_L} (\lambda_k + a_L),U_{k} / L, m_k\Big)_{k\ge 1} \Longrightarrow \Big(\Lambda_k,I_k,\delta_{I_k}\Big)_{k\ge 1}\;,$$
where $(\Lambda_k,I_k)_{k\ge 1}$ are the atoms of a Poisson point process on $\R\times\R_+$ with intensity $e^x e^{-t} dx\otimes dt$.
\end{theorem}



\noindent Here convergence takes place in the set of sequences of elements in $\R\times\R_+\times\cP(\R_+)$ endowed with the product topology, where $\cP(\R_+)$ is the space of probability measures on $\R_+$ endowed with the topology of weak convergence.

A natural question is then to determine the length scale of localization, together with the behavior of the eigenfunctions near their localization centers. This is the content of our next result, which relies on the following notations. We set for $x\in\R$
\begin{align*}
h_{k,\beta}(x) := \frac{\sqrt 2}{a_L^{1/4}} \varphi_k\Big(U_k + \frac{x}{\sqrt{a_L}}\Big)\;,\quad b_{k,\beta}(x) := \frac1{\sqrt{a_L}} \Big(B(U_k + \frac{x}{\sqrt{a_L}}) - B(U_k)\Big)\;,
\end{align*}
where $B(x) := \int_0^x \xi(dy)$. We also define $h(x) = 1/{\cosh x}$ and $b(x) = -2\tanh(x)$ for all $x\in\R$.
\begin{theorem}\label{Th:Shape}
For every $k\ge 1$, the random processes $h_{k,\beta},b_{k,\beta}$ converge to $h,b$ in probability locally uniformly on $\R$.
\end{theorem}

More can be said on the eigenfunctions. First, they decay at the exponential rate $\sqrt{a_L}$ from their localization centers. Second, if we let $0=z_0 < z_1 < \ldots < z_{k-1} < z_k=\infty$ be the zeros of $\varphi_k$ and if we let $\ell_*$ be such that the localization center $U_k$ lies in $[z_{\ell_*-1},z_{\ell*}]$, then on every $[z_{i-1},z_i]$ with $i < \ell_*$ (resp.~$i>\ell_*$) the function $\varphi_k$ admits a local maximum which is very close to the localization center of some eigenfunction $\varphi_j$ with $j< k$ and which is also very close to $z_i$ (resp.~to $z_{i-1}$). These estimates can be established using the material presented in this article but with some additional effort: we chose not to include their proofs in the present paper, but we refer the interested reader to~\cite{DL17} where similar results were established for the continuous Anderson Hamiltonian.

\begin{remark}\label{Rk:Scaling}
One can couple the two operators $\cA_\beta$ and $\cL_\beta$ and get the almost sure identities:
$$ \lambda_k = (\beta/4)^{2/3} \mu_k\;,\quad \varphi_k(x) = (\beta/4)^{1/6} \psi_k(x (\beta/4)^{1/3})\;,\quad x\in (0,\infty)\;.$$
Setting $c_\beta := (\frac{3}{2\beta} \ln \frac1{\pi\beta})^{2/3}$ and letting $E_k$ be the point where $|\psi_k|$ reaches its maximum, Theorem \ref{Th:Main} then reads
$$ \Big(\beta \sqrt{c_\beta} (\mu_k + c_\beta),E_{k} \beta \sqrt{c_\beta}, m_k\Big)_{k\ge 1} \Longrightarrow \Big(\Lambda_k,I_k,\delta_{I_k}\Big)_{k\ge 1}\;,$$
and the limit is the same as in the statement of the theorem. Furthermore, if one takes
$$ h_{k,\beta}(x) := \frac{\sqrt 2}{c_\beta^{1/4}} \Big|\psi_k\Big(E_{k} + \frac{x}{\sqrt{c_\beta}}\Big)\Big|\;,\quad b_{k,\beta}(x) := \frac{(\beta/4)^{1/6}}{\sqrt{c_\beta}} \Big(W\Big(E_{k} + \frac{x}{\sqrt{c_\beta}}\Big)-W\Big(E_{k}\Big)\Big)\;,$$
where $W$ is the Brownian motion associated to the white noise that drives $\cA_\beta$, then the statement of Theorem \ref{Th:Shape} still holds.
\end{remark}

\section{The Riccati transform and the strategy of proof}\label{Section:Riccati}

It was shown in~\cite[Section 3]{RamRidVir} that the study of the eigenvalues / eigenfunctions of $\cA_\beta$ could be carried out at the level of a family of diffusions obtained through the so-called Riccati transform. The same transform can be applied to $\cL_\beta$ and this yields the following family of diffusions
\begin{equation}\label{Eq:Za}
dZ_a(t) = (a + \frac{\beta}{4} t - Z_a(t)^2) dt + dB(t)\;,\quad Z_a(0)=+\infty\;,\quad a \in \R\;,
\end{equation}
with the Brownian motion $B$ introduced above. This is a time-inhomogeneous diffusion that evolves in the potential
$$V(t,x) = \frac{x^3}{3} - \Big(a + \frac{\beta}{4}t\Big) x\;.$$
At any time $t\ge 0$ and for $a>0$, the function $V(t,\cdot)$ has a local minimum at $x=\sqrt{a+\frac{\beta}{4} t}$ and a local maximum at $x=-\sqrt{a+\frac{\beta}{4} t}$: the region in between these two points will be referred to as \emph{the barrier of potential} since, there, the diffusion feels a very strong drift towards the local minimum.\\

The diffusion $Z_a$ may explode to $-\infty$ in finite time: we then restart it immediately from $+\infty$. It is shown in~\cite[Section 3]{RamRidVir} that almost surely for every $k\ge 1$, the event $\{\lambda_k \le -a\}$ coincides with the event $\{Z_a$ explodes to $-\infty$ at least $k$ times$\}$, and that we have
$$ \frac{\varphi_k'}{\varphi_k}(t) = Z_{-\lambda_k}(t)\;,\quad \forall t\ge 0\;.$$
The map $\varphi_k \mapsto Z_{-\lambda_k}$ is usually referred to as the Riccati transform.

%

\begin{figure}[!h]

\begin{minipage}{7.4cm}
\centering
\includegraphics[width = 3.7cm]{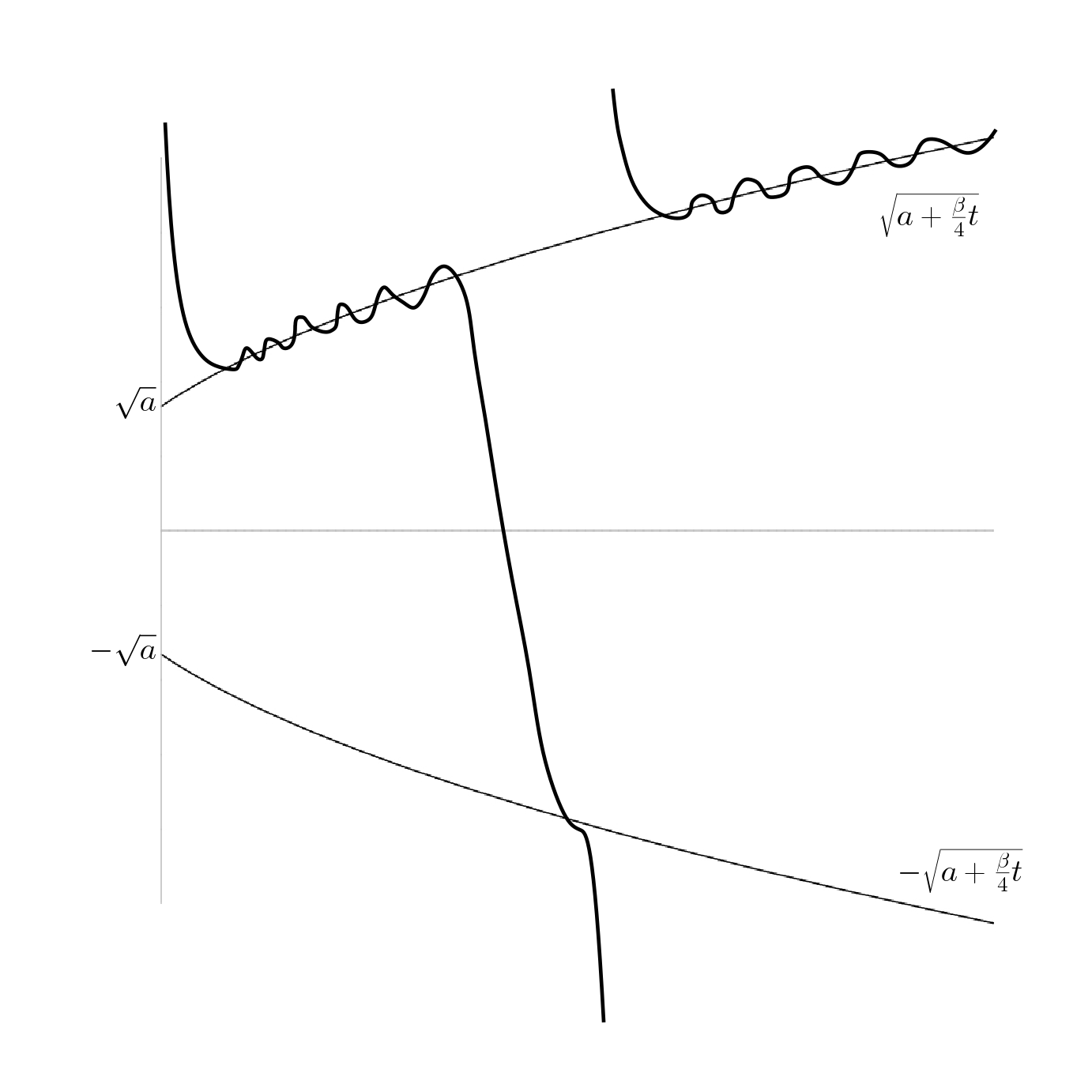}
\caption{\small A typical realization of the diffusion $Z_a$. Note that it takes a very short time to come down from infinity, spends most of its time near the curve $\sqrt{a+\frac{\beta}{4}t}$ and does not spend much time near the curve $-\sqrt{a+\frac{\beta}{4}t}$.}\label{Fig:Za}
\end{minipage}\hfill
\begin{minipage}{8cm}
\centering
\includegraphics[width = 3.7cm]{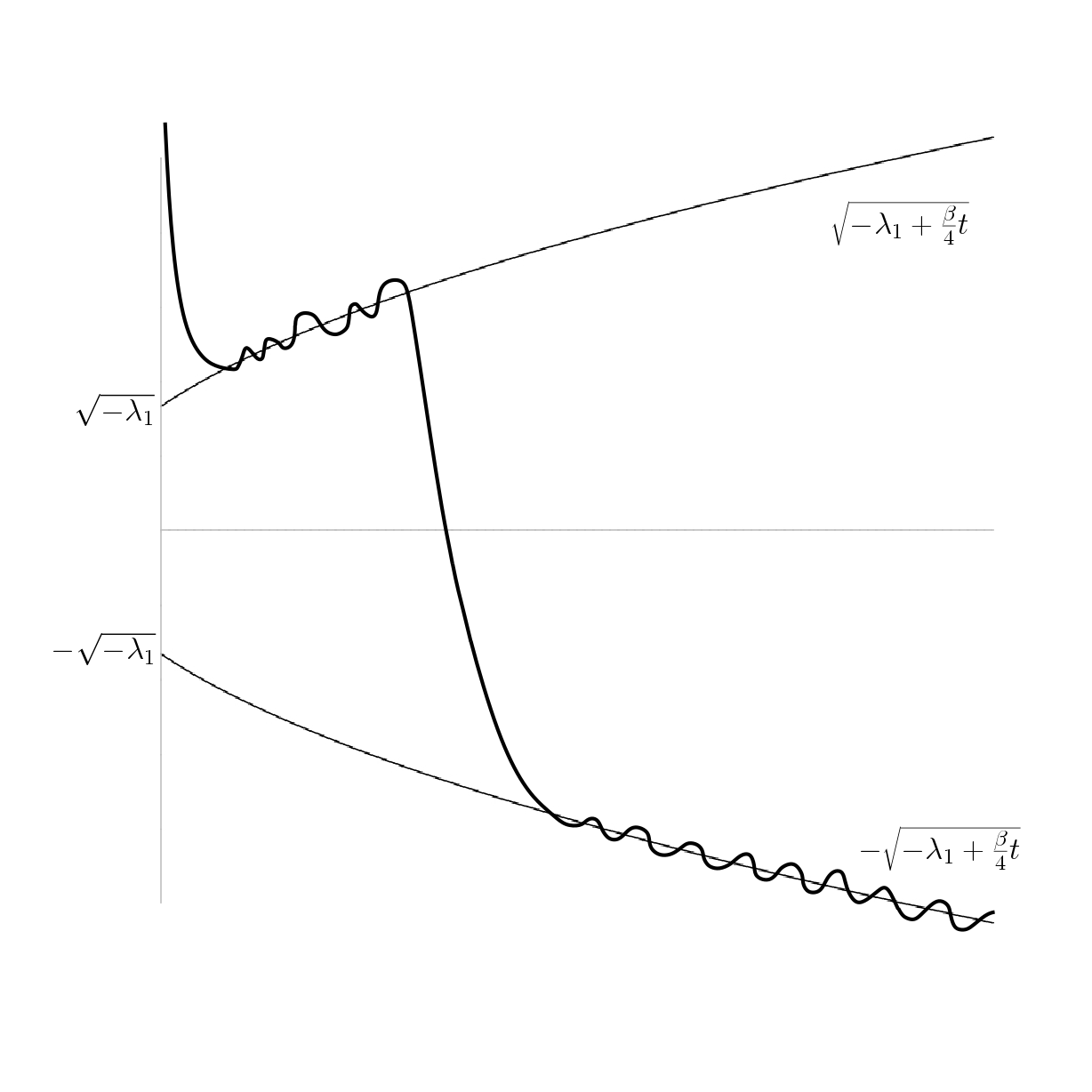}
\caption{\small A typical realization of the Riccati transform $\chi_1$ of the first eigenfunction. After having crossed the barrier of potential, the process oscillates forever around $-\sqrt{-\lambda_1+\frac{\beta}{4}t}$. Note that this behavior is unlikely for the diffusion $Z_a$.}\label{Fig:Y1}
\end{minipage}
\end{figure}

\subsection{Strategy of proof}

By rescaled eigenvalues, we mean the values $4\sqrt{a_L}(\lambda_k + a_L)$, $k\ge 1$. To prove the convergence of the eigenvalues, the main step consists in showing that, for any $p\ge 1$ and any disjoint intervals $[a_i,b_i]$, $i=1,\ldots,p$, the numbers of rescaled eigenvalues that fall into $[a_i,b_i]$ converge to independent Poisson r.v.~with intensities $\int_{a_i}^{b_i} e^{x} dx$.\\
To that end, we subdivide the time-interval $[0,\infty)$ of the diffusions into $2^n$ intervals $[t^n_j L,t^n_{j+1} L)$ with $0=t^n_0 < \ldots < t^n_{2^n}=\infty$. We consider the stochastic Airy operator restricted to every such interval and endowed with Dirichlet b.c. We then show that with large probability in the large $L$ and $n$ limit:
\begin{enumerate}
\item each restricted SAO has at most one (rescaled) eigenvalue in $\cup_{i=1}^n [a_i,b_i]$,
\item the number of (rescaled) eigenvalues in $[a_i,b_i]$ for the SAO equals the sum of the number of (rescaled) eigenvalues in $[a_i,b_i]$ of the restricted SAO's.
\end{enumerate}
Since the restricted SAO's are independent from each other, and since we are able to estimate the probability that they have one eigenvalue in $[a_i,b_i]$, a standard argument (see Lemma \ref{Lemma:CVQn}) yields convergence towards a vector of independent Poisson r.v. The proof of the convergence of the eigenvalues is presented in Subsection \ref{Subsec:PPP} and relies on a technical result established in Section \ref{Section:Explo}: these two parts can be read independently of the rest of the paper.

\bigskip

To prove the statements about the eigenfunctions, we observe that it suffices to prove their equivalent versions at the level of the Riccati transforms of the eigenfunctions: therefore, we only deal with the random processes $\chi_k := Z_{-\lambda_k}$. For simplicity, let us explain only the case $k=1$ in this introduction (the behavior of the next ones is illustrated on Figure \ref{Fig:Yk}).
We will show that $\chi_1$ comes down from infinity very quickly, then oscillates for a time of order $L$ around the curve $\sqrt{-\lambda_1+\frac{\beta}{4}t}$ and, at some point, crosses the ``barrier of potential'' to reach the curve $-\sqrt{-\lambda_1+\frac{\beta}{4}t}$ and then oscillates forever around this latter curve. This is illustrated on Figure \ref{Fig:Y1}. Moreover, the process crosses the barrier of potential by staying very close to a deterministic curve given by a hyperbolic tangent.\\
Inverting the Riccati transform, one deduces that $\varphi_1$ has exponential growth (resp. decay) when $\chi_1$ oscillates around $\sqrt{-\lambda_1+\frac{\beta}{4}t}$ (resp. around $-\sqrt{-\lambda_1+\frac{\beta}{4}t}$), and that the crossing of the barrier corresponds to the inverse of a hyperbolic cosine.
It is striking to compare this behavior with that of a ``typical'' realization of the diffusion $Z_a$ for a fixed parameter $a$, see Figures \ref{Fig:Za} and \ref{Fig:Y1}: the diffusion $Z_a$ would \emph{not} spend time around the curve $-\sqrt{-\lambda_1+\frac{\beta}{4}t}$ as it corresponds to an unstable line of its (time-inhomogeneous) potential.

\medskip

To prove the above assertions, we need two preliminary results. First of all, we establish that $L$ defined in \eqref{def:L} is indeed the relevant length scale of the localization centers of the eigenfunctions and that the associated value $a_L$ (see \eqref{Eq:DefaL}) is the order of magnitude of the eigenvalues. This is carried out by showing that a diffusion $Z_a$ with $a$ close enough to $a_L$ explodes \textit{finitely} many times and that all its explosion times are of order $L$ with large probability, uniformly over all $\beta$ small enough. This is a delicate result that relies on approximations of the time-inhomogeneous diffusion $Z_a$ by some time-homogeneous ones. In particular, an important part of the paper is devoted to prove that the diffusions $Z_a$ with $a$ close enough to $a_L$ typically do not explode after a time $C L$ for some large constant $C$, see Section \ref{Section:Explo}.

Second, to obtain a precise description of the eigenfunctions, we rely on the monotonicity of the diffusions: if for $a<a'$, the diffusion $Z_a$ explodes once and $Z_{a'}$ never explodes, then $\chi_1$ is squeezed in between these two diffusions until the explosion time of the former. To carry on the analysis after this explosion time, we apply a similar strategy but backward in time.\\
We start by showing that there exists a unique process $\hZ_a$ that solves
\begin{equation*}
d\hZ_a(t) = (a+\frac{\beta}{4}t - \hZ_a(t)^2)dt + dB(t)\;,\quad \hZ_a(+\infty) = -\infty\;.
\end{equation*}
We also show that the diffusion $Z_a$ converges to either $+\infty$ or $-\infty$ when $t\to\infty$, and that in the latter case it necessarily coincides with $\hZ_a$. This provides an alternative characterization of the eigenvalues: $-a$ is an eigenvalue if and only if $\hZ_a(0) = +\infty$. We refer to Theorem \ref{Th:TimeReversal} and Corollary \ref{Cor:bc}.\\
Building on these backward diffusions, we then track $\chi_1$ \emph{backward in time} by squeezing it in between two diffusions $\hat{Z}_a$ and $\hat{Z}_{a'}$. Then, an important part of our proof is devoted to patching together the forward and backward controls that we have on $\chi_1$.

\begin{figure}[!h]
\centering
\includegraphics[width = 6cm]{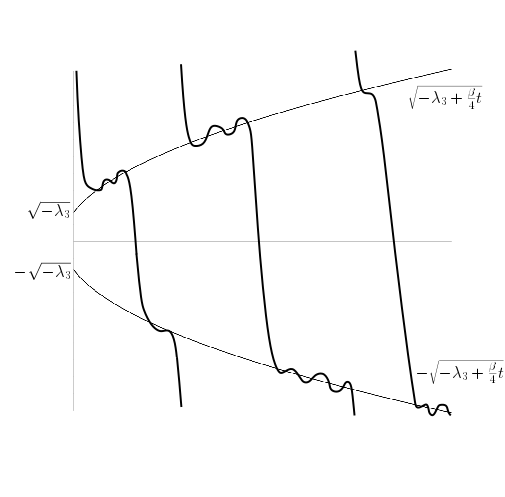}
\caption{\small A typical realization of the Riccati transform $Y_3$ of the third eigenfunction. Until the second crossing of the barrier of potential, the process is similar to $Z_a$. Then, it is similar to the backward diffusion $\hZ_a$.}\label{Fig:Yk}
\end{figure}

\subsection{Connection with the Anderson Hamiltonian}

As mentioned above, an important tool in our approach is a discretization scheme which boils down to comparing the original SAO with independent, restricted SAO's. It turns out that the interval lengths on which we consider the restricted SAO's will be of order $2^{-n} L$: at such a scale, the term $(\beta/4) x$ in the expression of the operator is essentially constant. Therefore, it is tempting to (and we will) approximate any such restricted SAO by the so-called Anderson Hamiltonian (shifted by a constant $c$ that approximates $(\beta/4) x$ on the corresponding interval $\cI$):
$$ \cH := -\partial_x^2 + c + \xi\;,\quad x\in \cI\;,$$
endowed with Dirichlet b.c. Actually, this approximation will be made at the level of the Riccati transforms, see Sections \ref{Sec:Techos} and \ref{Sec:Fine}.

In a recent work~\cite{DL17} we obtained a complete description of the bottom of the spectrum of the Anderson Hamiltonian when the size of the underlying interval goes to $\infty$. In particular, we showed that the smallest eigenvalues converge to a Poisson point process of intensity $e^{x}dx$ and the corresponding eigenfunctions are localized at some IID uniform points and are close to the inverse of a hyperbolic cosine near their localization centers. The present results can therefore be seen as a time-inhomogeneous extension of those in~\cite{DL17}.

\subsection{Organization of the paper} In Section \ref{Section:Reversal}, we construct the backward diffusions needed for the 
study of the eigenfunctions. In Section \ref{Section:Explo}, we prove the convergence of the point process of explosion times of $Z_a$ towards a Poisson point process. In Section \ref{Section:Proofs} we present the proofs of the main theorems. The reader interested in the sole convergence of the eigenvalues can skip Section \ref{Section:Reversal}, and will find all the arguments in Section \ref{Section:Explo} and Subsection \ref{Subsec:PPP}. In Section \ref{Sec:Techos} we present estimates on the diffusion $Z_a$ when it comes down from infinity, explodes and oscillates near the bottom of its time-inhomogeneous well, and we prove some intermediate results stated in the previous sections. Section \ref{Sec:Fine} is dedicated to delicate estimates on the behavior of $Z_a$ when it crosses its barrier of potential: these estimates are very similar to estimates established in~\cite{DL17} on a time-homogeneous diffusion and the proofs in that section therefore rely extensively on~\cite{DL17}.

\subsection*{Acknowledgements} The work of LD is supported by the project MALIN  ANR-16-CE93-0003. The work of CL is supported by the project SINGULAR ANR-16-CE40-0020-01.

\section{Construction of the backward diffusions}\label{Section:Reversal}

As mentioned in the previous section, the diffusions defined in \eqref{Eq:Za} play an important role in the study of the eigenfunctions. The present section is devoted to introducing the associated backward diffusions, as they will be instrumental in proving the localization of the eigenfunctions. In the whole section, the parameter $\beta > 0$ (or equivalently, the parameter $L$) is fixed.

For any $a\in \R$, and for any space-time point $(t_0,x_0) \in \R_+ \times \R$ one can consider the \emph{forward diffusion} that starts from $x_0$ at time $t_0$
\begin{equation}\label{Eq:Riccati}
\begin{cases}
dZ_a^{(t_0,x_0)}(t) &= (a+\frac{\beta}{4}t - Z_a^{(t_0,x_0)}(t)^2)dt + dB(t)\;,\quad t >t_0\;,\\
Z_a^{(t_0,x_0)}(t_0) &= x_0\;,
\end{cases}
\end{equation}
but one can also consider the \emph{backward diffusion} that ends at $x_0$ at time $t_0$
\begin{equation}\label{Eq:RiccatiReversed}
\begin{cases}
d\hZ_a^{(t_0,x_0)}(t) &= (a+\frac{\beta}{4}t - \hZ_a^{(t_0,x_0)}(t)^2)dt + dB(t)\;,\quad t \in [0,t_0)\;,\\
\hZ_a^{(t_0,x_0)}(t_0) &= x_0\;.
\end{cases}
\end{equation}
Concatenating $Z_a^{(t_0,x_0)}$ and $\hZ_a^{(t_0,x_0)}$, one obtains a path\footnote{By path, we mean a function from some interval of $\R$ into $\R\cup\{+\infty\}\cup\{-\infty\}$.} that coincides with $Z_a^{(0,x)}$ for $x= \hZ_a^{(t_0,x_0)}(0)$.\\
Note that it is natural to consider the backward diffusion with time run backward. Setting $Y(t) := \hZ_a^{(t_0,x_0)}(t_0-t)$ leads to the following:
$$\begin{cases}
dY(t) &= (-a-\frac{\beta}{4}(t_0-t) + Y(t)^2)dt - dB(t_0-t)\;,\quad t \in (0,t_0]\;,\\
Y(0) &= x_0\;.\end{cases}
$$
\begin{remark}
The diffusion $Y$ evolves in the time-inhomogeneous potential $(a+\beta(t_0-t)/4)x - x^3/3$: for $a>0$, the bottom of the well at time $t$ is located at $-\sqrt{a+ \beta(t_0-t)/4}$. This means that the backward diffusion $\hZ_a$ tends to be close to $-s_a(t)$ while the forward diffusion typically lies in a neighborhood of $s_a(t)$, with $s_a(t) =\sqrt{a+\beta t/4}$.
\end{remark}

\medskip

At this point, let us make a few technical comments. First of all, the construction of these diffusions is totally deterministic: once we are given a standard Brownian motion $B$, we can work deterministically and construct all the above processes as solutions to ODEs driven by the continuous trajectory $t\mapsto B_t(\omega)$. Second, simple arguments applied to the ODE show that the forward diffusion is well-defined when starting from $x_0 = +\infty$ since the associated ODE comes down from infinity; similarly, the backward diffusion is well-defined when starting from $x_0 = -\infty$. Furthermore, the forward diffusion may hit $-\infty$ in finite time: then, it restarts immediately from $+\infty$. Similarly, the backward diffusion - when run backward in time - may hit $+\infty$ in finite time and then restarts from $-\infty$. Third, the diffusion inherits a monotonicity property from the ODE. Namely, for all $a\le a'$, all $(t_0,x_0), (t_0',x_0')$ and all $s\in [t_0\vee t_0',\infty)$, if we have $Z_{a}^{(t_0,x_0)} (s) \le Z_{a'}^{(t_0',x_0')} (s)$ then
$$ Z_{a}^{(t_0,x_0)} (s+\cdot) \le Z_{a'}^{(t_0',x_0')} (s+\cdot)\;,$$
up to the next explosion time of $Z_{a}^{(t_0,x_0)}$. A similar statement holds for the backward diffusion.

\medskip

We aim at understanding the possible behaviors of the forward diffusions as $t\to\infty$. This is intimately linked to the construction of the backward diffusion starting from some point $x_0$ at time $t_0 = +\infty$. The main result of this section is the following.

\begin{theorem}\label{Th:TimeReversal}
There exists an event of probability one on which the following holds. For all $a\in\R$ and all $(t_0,x_0) \in \R_+ \times (\R\cup\{+\infty\})$, the forward diffusion $Z_a^{(t_0,x_0)}(t)$ goes to either $+\infty$ or $-\infty$ as $t\to\infty$. Additionally, for all $a\in\R$ there exists a unique path $\hZ_a^{(+\infty,-\infty)}$ that solves
\begin{equation}\label{Eq:RiccatiReversed2}
\begin{cases}
d\hZ_a(t) &= (a+\frac{\beta}{4}t - \hZ_a(t)^2)dt + dB(t)\;,\quad t \in [0,\infty)\;,\\
\hZ_a(+\infty) &= -\infty\;.
\end{cases}
\end{equation}
\end{theorem}

From this result, we deduce that for any given $a\in\R$, there exists a unique starting point $x_0\in\R\cup \{+\infty\}$ such that $Z_a^{(0,x_0)}(t)$ goes to $-\infty$ as $t\to\infty$: this starting point coincides with $\hZ_a^{(+\infty,-\infty)}(0)$. Any other starting point makes the forward diffusion go to $+\infty$ (this prevents uniqueness of a backward diffusion starting from $(+\infty,+\infty)$).

\begin{remark}
The discussion at the end of~\cite[Sec 3]{RamRidVir} shows that either $Z_a$ goes to $+\infty$ or $\int^t Z_a(s) ds$ is asymptotically smaller than $- C t^{3/2}$ for some positive constant $C$. While this result almost covers the statement of our theorem, it does not imply that $Z_a$ goes to $-\infty$ in the second case. 
\end{remark}

In Subsection \ref{Subsec:Symmetry}, we collect important consequences of the above theorem for the study of the eigenfunctions of $\cL_\beta$. The subsequent subsections are devoted to the proof of Theorem \ref{Th:TimeReversal}.

\medskip

From now on, we will implicitly view the backward diffusions as evolving backward in time (even though their evolution equations are stated forward in time). For the sake of clarity, we will put under quotation marks the words after or until when time is run backward. For instance, the sentence

\begin{center}
``until'' its first explosion time, the diffusion $\hat{Z}^{(t_0,x_0)}$ does [...]\end{center}

\noindent means that on the interval $[\tau,t_0]$ the diffusion does [...], where $\tau := \sup\{t < t_0: \hat{Z}^{(t_0,x_0)}(t) = +\infty\}$.

\subsection{Backward diffusions and eigenfunctions}\label{Subsec:Symmetry}

In the sequel, we abbreviate $Z_a^{(0,+\infty)}$ and $\hZ_a^{(+\infty,-\infty)}$ into $Z_a$ and $\hZ_a$.

\begin{corollary}\label{Cor:bc}
Almost surely, the set of eigenvalues $\{\lambda_k, k\ge 1\}$ coincides with the set
$$\{-a\in \R: \lim_{t\to\infty} Z_a(t) = -\infty\} = \{-a\in \R: \hZ_a(0) = +\infty\}\;.$$
Furthermore, the event $\{\lambda_k \le -a\}$ coincides with the event $\{Z_a$ explodes to $-\infty$ at least $k$ times$\}$, and we have
\begin{equation}\label{Eq:vep}
\frac{\varphi_k'}{\varphi_k}(t) = Z_{-\lambda_k}(t) = \hZ_{-\lambda_k}(t)\;,\quad \forall t\ge 0\;.
\end{equation}
\end{corollary}
\begin{proof}
The discussion at the beginning of~\cite[Section 3]{RamRidVir} shows that the Riccati transform applied to $\varphi_k$ yields a process that starts from $+\infty$ at time $0$ (due to the Dirichlet b.c.~imposed on $\cL_\beta$) and satisfies the same differential equation as $Z_a$ with $a=-\lambda_k$. Let us show that it necessarily goes to $-\infty$ at $+\infty$. Since $\varphi_k$ is in $L^2((0,\infty))$, the associated process $Z_a$, with $a=-\lambda_k$, cannot go to $+\infty$ at $+\infty$: by Theorem \ref{Th:TimeReversal} we deduce that $Z_a$ necessarily goes to $-\infty$.\\
Conversely, if $\hZ_a(0) = +\infty$ (or equivalently $Z_a(t)\to -\infty$ as $t\to\infty$) then the reverse Riccati transform provides an $L^2((0,\infty))$ function that solves the eigenproblem associated to $\cL_\beta$, thus concluding the proof.\\
Finally the monotonicity property of the diffusions implies that $\{\lambda_k \le -a\}$ coincides with the event $\{Z_a$ explodes to $-\infty$ at least $k$ times$\}$.
\end{proof}

Here is a simple consequence of identity \eqref{Eq:vep}. Let us denote by $0 < \zeta_a(1) < \zeta_a(2) < \ldots$ the successive explosion times (to $-\infty$) of $Z_a$, and by $0 < \hat{\zeta}_a(1) < \hat{\zeta}_a(2) < \ldots$ the successive explosion times (to $+\infty$) of $\hZ_a$. For convenience we set $\zeta_a(0) := 0$.
\begin{lemma}[Ordering of the explosions]\label{Lemma:Symmetry}
Almost surely for every $k\ge 1$, if $Z_a$ explodes $k$ times then $\hZ_a$ explodes $k$ times as well and we have for every $i\in\{1,\ldots,k\}$
\begin{align*}
 \zeta_a(i-1) \le \hat{\zeta}_a(i) \le \zeta_a(i) \;.
\end{align*}
\end{lemma}
\begin{proof}
The events ``$Z_a$ explodes $k$ times'' and ``$\hZ_a$ explodes $k$ times'' both coincide with the event ``$\lambda_k \le -a$'' so that they are almost surely equal.

Assume that we have $\hat{\zeta}_a(i) < \zeta_a(i-1)$ and take some rational number $t_0$ in between these two values. The operator $-\partial^2_x + \frac{\beta}{4}x + \xi$ restricted to $[0,t_0]$ has strictly less than $i-1$ eigenvalues below $-a$ due to $\zeta_a(i-1) > t_0$. On the other hand by monotonicity, the diffusion $\hZ^{(t_0,-\infty)}_a$ explodes at least $i-1$ times since $\hZ_a$ explodes at least $i$ times on $[0,t_0]$: consequently, the aforementioned operator must have at least $i-1$ eigenvalues below $-a$ thus raising a contradiction. Similar arguments yield the other inequality.
\end{proof}

\subsection{Construction of the backward diffusions}\label{Subsec:Construction}

We will construct a solution of  \eqref{Eq:RiccatiReversed2} by approximations. More precisely, for every $a\in\R$, we consider the two sequences of processes (indexed by $N\ge 1$)
$$\hZ^{(\frac{N}{\beta},-\infty)}_a(t) \quad \mbox{ and }\quad \hZ^{(\frac{N}{\beta},0)}_a(t)\;,\quad t\in [0,\frac{N}{\beta}]\;.$$
Note that these two diffusions, when run backward in time, start at time $N/\beta$ one above the other and, by the monotonicity property presented previously, remain ordered ``until'' the first explosion time to $+\infty$ of $\hZ^{(\frac{N}{\beta},0)}_a$.\\
One expects these two processes to be close to the parabola $-s_a(t) := -\sqrt{a+\beta t/4}$, at least for large enough $t$. Indeed, for the diffusion run backward in time, this parabola corresponds to the bottom of the well of its time-inhomogeneous potential, see Figures \ref{Fig:Za}, \ref{Fig:Y1} and \ref{Fig:Yk}.

Very informally, we will construct a solution of \eqref{Eq:RiccatiReversed2} by taking the limit of the sequence $\hZ^{(\frac{N}{\beta},-\infty)}_a$ on some (random) neighborhood of $+\infty$ where this sequence is non-decreasing. Regarding uniqueness, since any solution $Y$ of \eqref{Eq:RiccatiReversed2} tends to $-\infty$, there exists some $N_0$ such that for all $N\ge N_0$ we have
$$ -\infty = \hZ^{(\frac{N}{\beta},-\infty)}_a(\frac{N}{\beta}) \le Y(\frac{N}{\beta}) \le \hZ^{(\frac{N}{\beta},0)}_a(\frac{N}{\beta}) = 0\;,$$
and, consequently, $Y$ is squeezed in between the two sequences for large enough times: we will thus prove that the difference between these two processes tends to $0$ to conclude.\\

\begin{remark}
We consider the (seemingly complicated) sequence of times $(\frac{N}{\beta})_{N\ge 1}$ in order for our later estimates to be uniform over all $\beta > 0$. Indeed, these estimates will be applied in the next section for different purposes. However, for the sole proof of Theorem \ref{Th:TimeReversal}, we could have ``started'' our processes at time $N$ instead of $\frac{N}{\beta}$.
\end{remark}

The key technical result for the proof is the following proposition.
 
\begin{proposition}\label{Prop:ZmZp}
Fix $\ell \in \Z$ and $\beta >0$. As $k_0\to\infty$, the probability of the following event goes to $1$. For all $a\in[\ell-1,\ell]$ and for all $N \geq k_0+1$,
\begin{equation}\label{Eq:BoundZmZp}
 \forall t\in[\frac{k_0}{\beta},\frac{N-1}{\beta}]\;,\quad -\frac32 s_a(t) \le \hZ^{(\frac{N}{\beta},-\infty)}_a(t) \le \hZ^{(\frac{N}{\beta},0)}_a(t) \le -\frac12 s_a(t)\;,
 \end{equation}
and
\begin{align*}
\forall t\in[\frac{N-1}{\beta},\frac{N}{\beta}],\quad \hZ^{(\frac{N}{\beta},-\infty)}_a(t) \le \hZ^{(\frac{N}{\beta},0)}_a(t) \le 1\;.
\end{align*}
\end{proposition}

To control the behavior of the forward diffusions, we will need a companion result to the previous proposition. We consider the diffusion $Z_a^{(\frac{N}{\beta},0)}$ that starts from $0$ at time $\frac{N}{\beta}$ and goes \emph{forward} in time.
\begin{proposition}\label{Prop:BoundFwd}
Fix $\ell \in \Z$ and $\beta >0$. As $k_0\to\infty$, the probability of the following event goes to $1$. For all $a\in[\ell-1,\ell]$ and for all $N \geq k_0$,
\begin{equation}\label{Eq:BoundFwd}
\frac12 s_a(t) \le Z_a^{(\frac{N}{\beta},0)}(t) \le \frac32 s_a(t)\;,\quad \forall t \ge \frac{N+1}{\beta}\;.
 \end{equation}
\end{proposition}

We defer the proof of these two results until Subsection \ref{Subsec:Lemma} and now prove Theorem  \ref{Th:TimeReversal}.
\begin{proof}[Proof of Theorem \ref{Th:TimeReversal}]
If we prove that for any given $\ell\in\Z$, the statement of the theorem holds almost surely for all $a\in[\ell-1,\ell]$, then it obviously holds almost surely for all $a\in \R$. Therefore, $\ell \in\Z$ is fixed until the end of the proof.

\smallskip

Let us first prove the existence of solutions of \eqref{Eq:RiccatiReversed2} for $a \in [\ell-1,\ell]$. From Proposition \ref{Prop:ZmZp}, there exists a random integer $k_0 >0$ such that for all $a\in[\ell-1,\ell]$ and for every $N\ge k_0+1$, we have
$$ \hZ^{(\frac{N}{\beta},-\infty)}_a(t) \le -\frac12 s_a(t)\;,\quad t\in [\frac{k_0}{\beta}, \frac{N-1}{\beta}]\quad \mbox{ and }\quad \hZ^{(\frac{N}{\beta},-\infty)}_a(t) \le 1\;,\quad t\in [\frac{N-1}{\beta},\frac{N}{\beta}]\;.$$
By monotonicity, we thus deduce that for every $t \in [k_0/\beta,\infty)$, the sequence $\hZ^{(\frac{N}{\beta},-\infty)}_a(t)$, $N \ge t \beta$ is bounded non-decreasing and therefore converges pointwise: we call $\hZ_a(t)$ its limit. This limit satisfies $\hZ_a(t) \le -(1/2) s_a(t)$ and therefore goes to $-\infty$ as $t\to+\infty$. Furthermore, by passing to the limit on the equation solved by $\hZ^{(\frac{N}{\beta},-\infty)}_a$, we deduce that almost surely
$$ d\hZ_a(t) = \Big(a+\frac{\beta}{4}t - \hZ_a(t)^2\Big)dt + dB(t)\;,\quad t \in [\frac{k_0}{\beta},\infty)\;.$$
In addition, we set $\hZ_a(t) := \hZ_a^{(k_0/\beta,x_0)}(t)$ for all $t\in [0,k_0/\beta]$, where $x_0 = \hZ_a(k_0/\beta)$.

\smallskip

For uniqueness, let us first observe that on the event where \eqref{Eq:BoundZmZp} holds, for every given $t\ge k_0/\beta$ we have $(\hZ^{(\frac{N}{\beta},0)}_a-\hZ^{(\frac{N}{\beta},-\infty)}_a)(t) \to 0$ as $N\to\infty$. Indeed, solving the differential equation satisfied by the difference of these two processes we obtain that for all $t \in [\frac{k_0}{\beta}, \frac{N-1}{\beta}]$,
\begin{align*}
(\hZ^{(\frac{N}{\beta},0)}_a-\hZ^{(\frac{N}{\beta},-\infty)}_a)(t) &= (\hZ^{(\frac{N}{\beta},0)}_a-\hZ^{(\frac{N}{\beta},-\infty)}_a)(\frac{N-1}{\beta}) e^{\int_t^{\frac{N-1}{\beta}} (\hZ^{(\frac{N}{\beta},0)}_a + \hZ^{(\frac{N}{\beta},-\infty)}_a)(u) du}\\
&\le s_a(\frac{N-1}{\beta}) e^{-\int_t^{\frac{N-1}{\beta}}s_a(u) du}\;,
\end{align*}
which goes to $0$ as $N\to\infty$.\\
Let $Y_a$ be another solution of \eqref{Eq:RiccatiReversed2}. Since it goes to $-\infty$ as $t\to\infty$, there exists a random time $s_0$ after which $Y_a$ remains negative. As a consequence, almost surely for every $N \ge s_0 \beta$, $\hZ^{(\frac{N}{\beta},-\infty)}_a(\frac{N}{\beta}) < Y_a(\frac{N}{\beta}) \le \hZ^{(\frac{N}{\beta},0)}_a$ so that monotonicity ensures that $Y_a$ remains in between the two curves $\hZ^{(\frac{N}{\beta},-\infty)}_a$ and $\hZ^{(\frac{N}{\beta},0)}_a$ on $[(k_0\vee s_0)/\beta, N/\beta]$. Passing to the limit on $N$, we thus deduce that $Y_a$ must coincide with $\hZ_a$.

\medskip

We turn to the statement regarding the limit of $Z=Z_a^{(t_0,x_0)}$. We distinguish two cases. If $Z(\frac{N}{\beta}) \le 0$ occurs for infinitely many $N\ge 1$, then the argument presented right above to prove uniqueness shows that $Z$ actually coincides with $\hZ_a$: it therefore goes to $-\infty$ as $t\to\infty$.\\
Otherwise, there exists a random $N_0$ such that for all $N\ge N_0$, we have $Z(\frac{N}{\beta}) > 0$. Using Proposition \ref{Prop:BoundFwd} and monotonicity, we deduce that $Z(t)$ remains above $(1/2)s_a(t)$ for all $t\in [t_0,\infty)$ for some random $t_0$ and therefore $Z(t)$ goes to $+\infty$ as $t\to\infty$.
\end{proof}

\subsection{Exit time of Ornstein-Uhlenbeck process}

For $\theta > 0$ consider the Ornstein-Uhlenbeck process
$$ dU(t) = -\theta U(t) dt + dB(t)\;,\quad U(0) = 0\;.$$
For $b>0$ introduce the exit time
$$ H := \inf\Big\{t\ge 0: U(t) \notin \Big(-\frac{b}{\sqrt{2\theta}},\frac{b}{\sqrt{2\theta}}\Big)\Big\}\;.$$
In the next subsection, we will use the following estimate.
\begin{proposition}\label{Prop:OU}
There exists $C>0$ and $b_0 >0$ such that for all $\theta > 0$, all $\nu \in (0,1)$ and all $b > b_0$
$$ \E[e^{-\theta \nu H}] \le \frac{C}{1 + \frac{\nu}{b^2} e^{\frac{b^2}{2}}}\;.$$
\end{proposition}
\begin{remark}
This estimate indicates that, for $b$ large, the typical value of $H$ is of order $e^{\frac{b^2}{2}}$ (up to polynomial corrections). This is in line with the large deviation theory that asserts that the diffusion $U$, which evolves within the potential $V(x) = \theta x^2/2$, hits level $\pm b/\sqrt{2\theta}$ at a time of order $e^{2V(b/\sqrt{2\theta})}$.
\end{remark}
\begin{proof}
By~\cite[II.7.3.0.1 p.548 and Appendix 2.9 p.639]{Handbook} we have
$$ \E[e^{-\theta \nu H}] = e^{-\frac{b^2}{4}} \frac{2 D_{-\nu}(0)}{D_{-\nu}(-b)+D_{-\nu}(b)}\;,$$
where $D_{-\nu}$ is the so-called parabolic cylinder function. From its expression~\cite[Appendix 2.9 p.639]{Handbook}, we deduce that
\begin{align*}
\frac1{\E[e^{-\theta \nu H}]} &= 1+\sum_{k\ge 1} \frac{\nu(\nu+2)\times\ldots \times (\nu+2k-2)}{(2k)!}b^{2k}\\
&\ge 1+ \nu \sum_{k\ge 1} \frac{2 \times 4 \times\ldots \times(2k-2)}{(2k)!}b^{2k}\;.
\end{align*}
Since
$$\frac{2 \times 4 \times \ldots \times (2k-2)}{(2k)!}\,b^{2k} = \frac1{1 \times 3 \times \ldots \times (2k-1) \times (2k)}\, b^{2k} \ge \frac1{b^2} (\frac{b^2}{2})^{k+1} \frac1{(k+1)!}\;,$$
we deduce that
$$ \frac1{\E[e^{-\theta \nu H}]} \ge 1 + \frac{\nu}{b^2} (e^{\frac{b^2}{2}} - 1 - \frac{b^2}{2})\;.$$
The asserted bound follows.
\end{proof}

\subsection{Proofs of intermediate results}\label{Subsec:Lemma}

We will present in details the proof of Proposition \ref{Prop:BoundFwd}, and we will then present the main steps of the proof of Proposition \ref{Prop:ZmZp} since it is quite similar. The main argument is the following. On a small interval of time, the increment of the Brownian motion is small with large probability: the diffusion $Z^{(\frac{N}{\beta},0)}_{a}$ is then very close to the solution of the deterministic ODE obtained by removing the Brownian motion from its evolution equation. This ODE has an explicit solution that goes very quickly to a neighborhood of the curve $t\mapsto s_a(t)$.\\
The proof is split into two lemmas. The first one controls the behavior of $Z^{(\frac{N}{\beta},0)}_a$ on $[\frac{N}{\beta},\frac{N+1}{\beta}]$.

\begin{lemma}\label{Lemma:ZfwdInit}
There exist $A>0$ and $c>0$ such that, for all $\ell\in\Z$ and all $N\ge 1$ such that $\ell+N \ge A$, the following holds with probability at least $1-e^{-(\ell+N)^{1/3}}-\beta^{-1} e^{-c(\ell+N)^{3/2}}$. For all $a\in [\ell-1,\ell]$, we have
$$ \inf_{t\in [\frac{N}{\beta},\frac{N+1}{\beta}]} Z_a^{(\frac{N}{\beta},0)}(t) \ge -1\;,\quad\mbox{ and }\qquad\frac23 s_a(\frac{N+1}{\beta}) \le Z_a^{(\frac{N}{\beta},0)}(\frac{N+1}{\beta}) \le \frac43 s_a(\frac{N+1}{\beta})\;.$$
\end{lemma}

\noindent The second lemma will allow to bound the process $Z_a^{(\frac{N}{\beta},0)}$ between two deterministic curves through a recursion in time.
\begin{lemma}\label{Lemma:ZfwdRecur}
There exist $A>0$ and $c>0$ such that for all $\ell\in\Z$ and all $k\ge 1$ such that $\ell+k \ge A$, the following holds with probability at least $1-e^{-(\ell+k)^{1/3}} - \beta^{-1} e^{-c(\ell+k)^{3/2}}$. For all $N \le k-1$ and all $a\in [\ell-1,\ell]$, if
\begin{equation}\label{Eq:BoundZ0}
\frac23 s_a(\frac{k}{\beta}) \le Z^{(\frac{N}{\beta},0)}_a(\frac{k}{\beta}) \le \frac43 s_a(\frac{k}{\beta})\;,
\end{equation}
then we have
\begin{equation}\label{Eq:BoundZ}
\frac12 s_a(t) \le Z^{(\frac{N}{\beta},0)}_a(t) \le \frac32 s_a(t)\;,\quad \forall t\in [\frac{k}{\beta},\frac{k+1}{\beta}]\;,
\end{equation}
and
\begin{equation}\label{Eq:BoundZ2}
\frac23 s_a(\frac{k+1}{\beta}) \le Z^{(\frac{N}{\beta},0)}_a(\frac{k+1}{\beta}) \le \frac43 s_a(\frac{k+1}{\beta})\;,
\end{equation}
\end{lemma}

\begin{proof}[Proof of Proposition \ref{Prop:BoundFwd}]
Fix $\ell\in\Z$. Take $k_0 \ge 1$ such that $\ell+k_0\ge A$ where $A$ is the maximum of the $A$'s appearing in the above two lemmas. Applying the first lemma and iterating the second, we see that the probability of the event of the statement of the proposition is at least
$$1 - \sum_{k\ge k_0} (e^{-(\ell+k)^{1/3}}+\beta^{-1} e^{-c(\ell+k)^{3/2}}) - \sum_{N\ge k_0} (e^{- (\ell+N)^{1/3}}+\beta^{-1} e^{-c(\ell+N)^{3/2}})\;.$$
This probability goes to $1$ as $k_0\to\infty$, thus concluding the proof.
\end{proof}

\begin{proof}[Proof of Lemma \ref{Lemma:ZfwdInit}]
To alleviate notations, we will simply write $Z_a$ for $Z_a^{(\frac{N}{\beta},0)}$. Set
$$\kappa_N := \frac{\ln s_\ell(\frac{N}{\beta})}{s_\ell(\frac{N}{\beta})}\;.$$
Note that $s_\ell(N/\beta) = \sqrt{\ell+N}$ is well-defined as soon as $\ell+N \ge 0$. Note also that as $\ell+N \to\infty$ we have
$$ \frac{s_\ell(N/\beta)}{s_{\ell-1}(N/\beta)} \to 1\;,$$
so that, in the sequel, we will implicitly assume that this ratio is as close as desired to $1$.

Consider the event
$$\cA:=\Big\{\forall t\in [\frac{N}{\beta},\frac{N}{\beta}+\kappa_N],\; |B(t)-B(\frac{N}{\beta})| < 1\Big\}\;,$$
and note that $\P(\cA^\cc) \le 2e^{-\frac12\kappa_N^{-1}}$. Since $\kappa_N^{-1} \gg (\ell+N)^{1/3}$ as $\ell+N$ goes to $\infty$, we deduce that $\P(\cA^\cc) \le e^{-(\ell+N)^{1/3}}$ for $\ell+N$ large enough.\\

We first prove that on the event $\cA$ and as soon as $N$ is large enough, we can squeeze all the processes $Z_a$, $a\in [\ell-1,\ell]$, in between simple deterministic curves on the time interval $[\frac{N}{\beta},\frac{N}{\beta}+\kappa_N]$.

\medskip

By monotonicity, we have for all $a\in [\ell-1,\ell]$
$$ Z_{\ell-1} \le Z_a \le Z_{\ell}\;,$$
until the first explosion time of $Z_{\ell-1}$ to $-\infty$. Consequently, it suffices to bound from below $Z_{\ell-1}$ and from above $Z_{\ell}$. We start with the bound of the former, and set $R(t) := Z_{\ell-1}(\frac{N}{\beta}+t) - B(\frac{N}{\beta}+t) + B(\frac{N}{\beta})$ for all $t\ge 0$. We have
$$ dR(t) = s_{\ell-1}(\frac{N}{\beta}+t)^2 dt - Z_{\ell-1}(\frac{N}{\beta}+t)^2 dt\;.$$
We now work on the event $\cA$ and on the time-interval $[0,\kappa_N]$. If $|Z_{\ell-1}(\frac{N}{\beta}+t)| \in [0,(1/2) s_{\ell-1}(\frac{N}{\beta})]$ then for all $N$ large enough
$$ dR(t) \ge s_{\ell-1}(\frac{N}{\beta})^2 \Big(1 - \frac3{s_{\ell-1}(\frac{N}{\beta})}\Big)dt - R(t)^2dt\;,$$
and if $|Z_{\ell-1}(\frac{N}{\beta}+t)| \ge (1/2) s_{\ell-1}(\frac{N}{\beta})$ then
$$ dR(t) \ge s_{\ell-1}(\frac{N}{\beta})^2 dt - R(t)^2 \Big(1 + \frac5{s_{\ell-1}(\frac{N}{\beta})}\Big)dt\;.$$
Therefore, if we take $G$ as the solution of $G(0) = 0$ and
$$ dG(t) = s_{\ell-1}(\frac{N}{\beta})^2 \Big(1 - \frac3{s_{\ell-1}(\frac{N}{\beta})}\Big)dt - G(t)^2 \Big(1 + \frac5{s_{\ell-1}(\frac{N}{\beta})}\Big)dt\;,$$
then on the event $\cA$ we have $R(t) \ge G(t)$ for $t\in [0,\kappa_N]$. The function $G$ is explicit:
$$ G(t) = s_{\ell-1}(\frac{N}{\beta}) \sqrt{\frac{1 - \frac3{s_{\ell-1}(\frac{N}{\beta})}}{1 + \frac5{s_{\ell-1}(\frac{N}{\beta})}}} \tanh\Big(s_{\ell-1}(\frac{N}{\beta})\, t\,\sqrt{(1 - \frac3{s_{\ell-1}(\frac{N}{\beta})})(1 + \frac5{s_{\ell-1}(\frac{N}{\beta})})} \Big)\;.$$
It is non-negative and for all $\ell+N$ large enough, we have $G(\kappa_N) \ge \frac56 s_{\ell}(\frac{N+1}{\beta}) + 1$. Since
$$Z_{\ell-1}(\frac{N}{\beta}+t) \ge R(t) - |B(\frac{N}{\beta}+t)-B(\frac{N}{\beta})| \ge G(t) - |B(\frac{N}{\beta}+t)-B(\frac{N}{\beta})|\;,$$
we deduce that on the event $\cA$
$$ Z_{\ell-1}(t) \ge -1\;,\quad t\in [\frac{N}{\beta},\frac{N}{\beta}+\kappa_N]\;,\qquad Z_{\ell-1}(\frac{N}{\beta}+\kappa_N) \ge \frac56 s_\ell(\frac{N+1}{\beta})\;.$$

To bound from above $Z_{\ell}$, we proceed similarly. We set $R(t) = Z_{\ell}(\frac{N}{\beta}+t) - B(\frac{N}{\beta}+t) + B(\frac{N}{\beta})$, and one can check that for all $N$ large enough, on the event $\cA$ and for $t\in [0,\kappa_N]$ we have $R(t) \le F(t)$ where $F(0) = 0$ and
$$ dF(t) = s_{\ell}(\frac{N+1}{\beta})^2\Big( 1 + \frac{3}{s_{\ell}(\frac{N+1}{\beta})} \Big) dt - F^2(t) \Big( 1 - \frac{5}{s_{\ell}(\frac{N+1}{\beta})} \Big) dt\;.$$
Here again, it can be checked that $F(\kappa_N) \le \frac76 s_{\ell-1}(\frac{N+1}{\beta}) -1$ for all $N$ large enough. Consequently, on the event $\cA$ we have
\begin{equation}\label{Eq:IneqZZ} \frac56 s_\ell(\frac{N+1}{\beta}) \le Z_{\ell-1}(\frac{N}{\beta}+\kappa_N) \le Z_{\ell}(\frac{N}{\beta}+\kappa_N) \le \frac76 s_{\ell-1}(\frac{N+1}{\beta})\;.\end{equation}

To conclude, it suffices to prove that, conditionally given the filtration of the Brownian motion up to time $\frac{N}{\beta}+\kappa_N$, $Z_{\ell-1}$ remains above $(2/3) s_{\ell}(\frac{N+1}{\beta})$ and $Z_{\ell}$ remains below $(4/3) s_{\ell-1}(\frac{N+1}{\beta})$ on the time-interval $[\frac{N}{\beta}+\kappa_N,\frac{N+1}{\beta}]$. This is achieved by a comparison with a reflected Ornstein-Uhlenbeck process: let us give the details for the upper bound.\\
For all $\ell+N$ large enough, $s_\ell(t) \le (7/6) s_{\ell-1}((N+1)/\beta)$ for all $t\in [\frac{N}{\beta}+\kappa_N,\frac{N+1}{\beta}]$. An elementary computation yields
$$ d Z_{\ell}(t) \le -2s_\ell(t)(Z_\ell - s_\ell(t)) dt + dB(t) \le -2s_\ell(t)(Z_\ell - (7/6) s_{\ell-1}((N+1)/\beta)) dt + dB(t)\;,$$
and therefore $Z_\ell(t) \le R(t) + (7/6) s_{\ell-1}((N+1)/\beta)$ for all $t\in [\frac{N}{\beta}+\kappa_N,\frac{N+1}{\beta}]$ where $R$ is a non-negative process satisfying
$$ dR(t) = -2s_\ell(N/\beta) R(t) + dB(t) + dL(t)\;,\quad R(\frac{N}{\beta}+\kappa_N) = 0\;,$$
and $L$ is a reflection measure supported by the zeros of $R$. The process $R$ is equal in law to $|U|$ where $U$ is the Ornstein-Uhlenbeck process
$$ dU(t) =  -2s_\ell(N/\beta) U(t) + dB(t)\;,\quad U(\frac{N}{\beta}+\kappa_N) = 0\;.$$
As a consequence, there exists a constant $d>0$ such that, conditionally given \eqref{Eq:IneqZZ}, the probability that $Z_\ell$ hits $(4/3) s_{\ell-1}(\frac{N+1}{\beta})$ on the time interval $[\frac{N}{\beta}+\kappa_N,\frac{N+1}{\beta}]$ is bounded from above by the probability that $U$ exits $[-d s_{\ell}(N/\beta),d s_{\ell}(N/\beta)]$ on the same time interval. Applying Proposition \ref{Prop:OU} with $\nu = \beta/(2s_\ell(N/\beta))$, the latter is bounded by $\beta^{-1} \exp(-c(\ell+N)^{3/2})$, for some constant $c>0$ as soon as $\ell+N$ is large enough.
\end{proof}
\begin{proof}[Proof of Lemma \ref{Lemma:ZfwdRecur}]
Assume that \eqref{Eq:BoundZ0} holds for some $k$. Similarly as in the previous proof, by a comparison with reflected Ornstein-Uhlenbeck processes, one can deduce that \eqref{Eq:BoundZ} holds with a probability at least $1 - \beta^{-1} e^{-c(\ell + k)^{3/2}}$ for some constant $c>0$. 

To prove \eqref{Eq:BoundZ2}, we set $\kappa_k := \ln(s_\ell(k/\beta)) / s_\ell(k/\beta)$ and we work on the event
$$ \cA := \Big\{\forall t\in [\frac{k+1}{\beta}-\kappa_k,\frac{k+1}{\beta}],\; |B(t)-B(\frac{k+1}{\beta}-\kappa_k)| \le 1\Big\}\;.$$
On this event, one can squeeze the trajectory of $Z_a^{(\frac{N}{\beta},0)}$ in between two deterministic curves that get close to $s_a(\frac{k+1}{\beta})$ in a short time so that \eqref{Eq:BoundZ2} is satisfied: the proof is very similar to that of the last lemma so we do not provide the details. The probability of $\cA$ is larger than $1-e^{-(\ell+k)^{1/3}}$ provided that $\ell+k$ is large enough. This concludes the proof.
\end{proof}

\bigskip

The proof of Proposition \ref{Prop:ZmZp} is very similar. It relies on two intermediate lemmas and a recursion \emph{backward} in time. We only state the two lemmas since the arguments are essentially the same as above.
\begin{lemma}\label{Lemma:Z-Z+}
There exist $A>0$ and $c>0$ such that, for all $\ell\in\Z$ and all $N\ge 1$ such that $\ell+N \ge A$, the following holds with probability at least $1-e^{-(\ell+N)^{1/3}}-\beta^{-1} e^{-c(\ell+N)^{3/2}}$. For all $a\in [\ell-1,\ell]$, we have
$$ \sup_{t\in [\frac{N-1}{\beta},\frac{N}{\beta}]} \hZ^{(\frac{N}{\beta},-\infty)}_a(t) \le \sup_{t\in [\frac{N-1}{\beta},\frac{N}{\beta}]} \hZ^{(\frac{N}{\beta},0)}_a(t) \le 1\;,$$
and
$$ -\frac43 s_a(\frac{N-1}{\beta}) \le \hZ^{(\frac{N}{\beta},-\infty)}_a(\frac{N-1}{\beta}) \le \hZ^{(\frac{N}{\beta},0)}_a(\frac{N-1}{\beta}) \le -\frac23 s_a(\frac{N-1}{\beta})\;.$$
\end{lemma}

\begin{lemma}\label{Lemma:ZmZpRecur}
There exist $A>0$ and $c>0$ such that for all $\ell\in\Z$ and all $k\ge 1$ such that $\ell+k \ge A$, the following holds with probability at least $1-e^{-(\ell+k)^{1/3}}-\beta^{-1} e^{-c(\ell+k)^{3/2}}$. For all $N$ such that $k_0 \le N \le k-1$ and all $a\in [\ell-1,\ell]$, if
\begin{equation}\label{Eq:BoundhZ0}
-\frac43 s_a(\frac{k}{\beta}) \le \hZ^{(\frac{N}{\beta},-\infty)}_a(\frac{k}{\beta}) \le \hZ^{(\frac{N}{\beta},0)}_a(\frac{k}{\beta})\le -\frac23 s_a(\frac{k}{\beta})\;,
\end{equation}
then we have
\begin{equation}\label{Eq:BoundhZ}
-\frac32 s_a(t) \le \hZ^{(\frac{N}{\beta},-\infty)}_a(t) \le \hZ^{(\frac{N}{\beta},0)}_a(t) \le -\frac12 s_a(t)\;,\quad \forall t\in [\frac{k-1}{\beta},\frac{k}{\beta}]\;,
\end{equation}
and
\begin{equation}\label{Eq:BoundhZ2}
-\frac43 s_a(\frac{k-1}{\beta}) \le \hZ^{(\frac{N}{\beta},-\infty)}_a(\frac{k-1}{\beta}) \le \hZ^{(\frac{N}{\beta},0)}_a(\frac{k-1}{\beta})\le -\frac23 s_a(\frac{k-1}{\beta})\;.
\end{equation}
\end{lemma}

\medskip

\section{Convergence of the point process of explosion times}\label{Section:Explo}

Let $0 < \zeta_a(1) < \zeta_a(2) < \ldots$ be the successive explosion times of $Z_a := Z_a^{(0,+\infty)}$. For a function $a_L  \sim (\frac38 \ln L)^{2/3}$ whose precise definition will be given in the next subsection, we have the following result.

\begin{theorem}\label{Th:Explo}
Fix $r\in\R$ and set $a=a_L - \frac{r}{4\sqrt{a_L}}$. As $L\to\infty$, the random measure
$$ \sum_{k\ge 1} \delta_{\zeta_{a}(k)/L}\;,$$
converges in law for the topology of weak convergence of finite measures to a Poisson point process on $\R_+$ with intensity $e^r e^{-t} dt$.
\end{theorem}
This result strengthens~\cite[Th 4.1]{AllezDumazTW}: therein, the aforementioned convergence is established in the topology of vague convergence of Radon measures. This topology does not allow to control the mass at infinity while this is required in order to study the eigenvalues of the operator $\cL_\beta$. Actually, even to compute the limiting fluctuations of the first eigenvalue, one needs to evaluate the probability of non-explosion of $Z_{a}$ and this requires to control the mass at infinity of the above random point process.

To prove the theorem, we subdivide $[0,\infty)$ into three regions. First, in $[0,\eps^{-1}L]$ the process makes a finite number of explosions and the point process of explosion times restricted to this interval converges to a Poisson point process of the asserted intensity thanks to~\cite[Th 4.1]{AllezDumazTW}. Second, for any given $C_0>0$, in $[\eps^{-1}L, C_0 L {\ln L}]$ we will prove that the process does not explode with a probability that goes to $1$ as $\eps\to 0$, uniformly over all $L$ large enough. Third, in $[C_0L{\ln L},\infty)$ the process remains in between two deterministic curves with a probability going to $1$ as $L\to \infty$, provided $C_0>0$ is chosen large enough: this relies on exactly the same arguments as those presented in the proof of Theorem \ref{Th:TimeReversal}.

\subsection{The time-homogeneous diffusion}\label{Subsec:TimeHomo}

In this subsection, we introduce time-homogeneous versions of the diffusions $Z_a$: at many occasions in this article we will rely on comparison arguments involving this diffusion. For every $a\in\R$ and every $(t_0,x_0)\in \R_+\times (-\infty,+\infty]$ we define $X_a^{(t_0,x_0)}$ as the solution of the following SDE
\begin{equation}
\begin{cases}
dX_a(t) &= (a - X_a(t)^2)dt + dB(t)\;,\quad t >t_0\;,\\
X_a(t_0) &= x_0\;.
\end{cases}
\end{equation}
Each time $X_a$ hits $-\infty$, it restarts immediately from $+\infty$. This family of diffusions satisfies the following monotonicity property. Almost surely for all $a\le a'$, all $(t_0,x_0), (t_0',x_0')$ and all $s\in [t_0\vee t_0',\infty)$, if we have $X_{a}^{(t_0,x_0)} (s) \le X_{a'}^{(t_0',x_0')} (s)$ then we have $X_{a}^{(t_0,x_0)} (s+\cdot) \le X_{a'}^{(t_0',x_0')} (s+\cdot)$ up to the next explosion time of $X_{a}^{(t_0,x_0)}$.

Notice that this is a diffusion in the potential $V_a(x) = x^3/3 - a x$. When $a>0$, this potential admits a well centered at $x=\sqrt a$ and an unstable equilibrium point at $x=-\sqrt a$. A typical sample path of the diffusion spends most of its time near the bottom of the well, and from time to time manages to reach the unstable equilibrium point from where it either explodes to $-\infty$ or comes back to the bottom of the well within a short time.

Let us recall the following convergence result due to McKean~\cite{McKean}. If we let $\gamma_a$ be the first time at which $X_a$ explodes, and if we let $m(a) = \E[\gamma_a]$, then $\gamma_a / m(a)$ converges in law to an exponential r.v.~of parameter $1$ as $a\to\infty$.

Observe that from the stochastic monotonicity of $a\mapsto \gamma_a$, the map $a\mapsto m(a)$ is non-decreasing. McKean~\cite{McKean} showed that it satisfies:
$$ m(a) = \frac{\pi}{\sqrt a} \exp(\frac83 a^{3/2})(1+o(1))\;,\quad a\to\infty\;.$$
Let us collect two estimates on $m(a)$ for the sequel. Simple computations show that for all $x\in \R$
$$ \frac{m(a + \frac{x}{4\sqrt a})}{m(a)} \to e^x\;,\quad a \to\infty\;,$$
and that, for any $x_0 \in\R$, there exists a constant $c>0$ such that for all $a$ large enough and for all $x > x_0$
\begin{equation}\label{Eq:mac} \frac{m(a + \frac{x}{4\sqrt a})}{m(a)} > ce^{x}\;.\end{equation}

We define the function $L\mapsto a_L$ as the inverse of $a\mapsto m(a)$. We have as $L\to\infty$
\begin{equation}\label{Eq:DefaL} a_L = \Big(\frac38 \ln L\Big)^{2/3}\Big(1+ \frac29 \frac{\ln\ln L}{\ln L} + (-\frac23 \ln \pi + \frac29 \ln \frac38) \frac1{\ln L} + o(\frac1{\ln L})\Big)\;.\end{equation}
Recall that $L=L(\beta)$ and note that as $\beta \to 0$ (which is equivalent to $L\to\infty$) we have
$$ a_L = \Big(\frac38 \ln \frac1{\beta}\Big)^{2/3}\Big(1- \frac23 \frac{ \ln \pi}{\ln(1/\beta)} + o(\frac1{\ln (1/\beta)})\Big)\;.$$

\subsection{An estimate on McKean's convergence result}

In~\cite{McKean}, McKean showed that the first explosion time $\gamma_a$ of the time-homogeneous diffusion $X_a$, rescaled by $m(a)$, converges in distribution to an exponential r.v.~with parameter $1$. The following proposition gives more precise information about the probability that the diffusion explodes at a time much smaller than $m(a)$. 

Let $\cE(1)$ denote an exponential r.v.~of parameter $1$.
\begin{proposition}\label{Prop:CVrate}We have
$$ \varlimsup_{a\to\infty} \sup_{x\in [(\ln a)^{-3}, 1]} \frac1{x} \Big| \ln\Big( \frac{\P(\cE(1) > x)}{\P(\gamma_a/m(a) > x)}\Big)\Big| = 0\;.$$
\end{proposition}
\noindent This convergence takes the following equivalent form:
$$ \varlimsup_{a\to\infty} \sup_{x\in [(\ln a)^{-3}, 1]} \frac1{x} \Big|\P(\cE(1) \le x) - \P(\gamma_a/m(a) \le x)\Big| = 0\;.$$
\begin{proof}
By monotonicity and from the explicit expression of the exponential density, it suffices to prove
$$ \varlimsup_{a\to\infty} \sup_{n=1,\ldots, (\ln a)^3} n \Big| \ln\Big( \frac{\P(\cE(1) > 1/n)}{\P(\gamma_a/m(a) > 1/n)}\Big)\Big| = 0\;.$$
For every $n\ge 1$, we let $X^j_a := X_a^{(t_j^n,+\infty)}$ be the diffusion that starts from $+\infty$ at time $t_j^n := (j/n)m(a)$ and solves the same SDE as $X_a$. We then let $A_n$ be the event on which $X_a$ explodes on $[0,m(a)]$ if and only if there exists $j\in \{0,\ldots,n-1\}$ such that $X_a^j$ explodes on $[t_j^n,t_{j+1}^n]$. Let us denote by $\gamma_a^j:= \inf\{t\ge 0: X_a^j(t_j^n+t) = -\infty\}$. We write
\begin{align*}
\P(\gamma_a/m(a) > 1) = \P(\gamma_a/m(a) > 1; A_n^\cc) + \P(\gamma_a/m(a) > 1; A_n)\;.
\end{align*}
The first term can be bounded by $\P(A_n^\cc)$ while the second term satisfies
\begin{align*}
\P(\gamma_a/m(a) > 1; A_n) &= \P(\cap_{j=0}^{n-1} \{\gamma_a^j / m(a) > 1/n\} \cap A_n)\\
&= \P(\cap_{j=0}^{n-1} \{\gamma_a^j/m(a) > 1/n\}) - \P(\cap_{j=0}^{n-1} \{\gamma_a^j/m(a) > 1/n\} \cap A_n^\cc)\\
&= \P(\gamma_a/m(a) > 1/n)^n - \P(\cap_{i=1}^n \{\gamma_a^j/m(a) > 1/n\} \cap A_n^\cc)\;,
\end{align*}
since the $\gamma_a^j$'s are IID with the same law as $\gamma_a$. Hence
$$ \big|\P(\gamma_a/m(a) > 1) - \P(\gamma_a/m(a) > 1/n)^n\big| \le 2\,\P(A_n^\cc)\;,$$
for all $n\ge 1$. Since $\P(\cE(1) > 1/n)^n = \P(\cE(1)> 1)$, we get
\begin{align*}
n\Big|\ln\Big( \frac{\P(\cE(1) > 1/n)}{\P(\gamma_a/m(a) > 1/n)}\Big)\Big| &= \Big|\ln\Big( \frac{\P(\gamma_a/m(a) > 1/n)^n}{\P(\cE(1) > 1)}\Big)\Big|\\
&\le \Big|\ln\Big( \frac{\P(\gamma_a/m(a) > 1)}{\P(\cE(1) > 1)}\Big)\Big| +  \Big|\ln\Big( \frac{\P(\gamma_a/m(a) > 1/n)^n}{\P(\gamma_a/m(a) > 1)}\Big)\Big|\;.
\end{align*}
The first term on the r.h.s.~converges to $0$ as $a\to\infty$. The second term is bounded by $C \P(A_n^\cc)$ for some constant $C>0$ uniformly over all $n\ge 1$ and all $a$ large enough. Therefore, we are left with proving that $\sup_{n\le (\ln a)^3} \P(A_n^\cc) \to 0$ as $a\to\infty$.\\

%
Using the proof of \cite[Prop. 2.6]{DL17} at times $t_j^n$, we easily deduce that with a probability greater than $1- n \exp(-b(\ln\ln a)^2)$ for some $b>0$, $X_a$ explodes on $[0,m(a)]$ if and only if there exists $j\in\{0,\ldots,n-1\}$ such that $X_a^j$ explodes on $[t^n_j,t^n_{j+1}]$, as long as the explosion times of $X_a$ are at a distance at least $C/\sqrt{a}$ from the times $t_j^n$. The latter holds true with large probability thanks to \cite[Cor. 4.8]{DL17}: indeed, it is shown therein that with a probability greater than $1- n \exp(-b(\ln\ln a)^2)$ the diffusion $X_a$ remains close to a stationary diffusion up to its $n$-th explosion time and it is easy to control the probability that a stationary diffusion does not explode in small neighborhoods of the $t^n_j$ using the estimates in \cite[Lemma 4.1]{DL17}. Since $(\ln a)^3 \ll e^{b(\ln \ln a)^2}$, this completes the proof.
\end{proof}

\subsection{The delicate region}\label{Subsec:Delicate}

The goal of this subsection is to prove the following result.

\begin{proposition}\label{Prop:ExploDelicate}
Fix $C_0>0$. For all $\eps$ small enough and all $L$ large enough, the probability that $Z_{a_L}$ does not explode on $(\eps^{-1}L,C_0 L {\ln L}]$ is larger than $1-\eps$.
\end{proposition}

To prove the proposition, we cover $(\eps^{-1}L,C_0 L {\ln L}]$ by the disjoint intervals $(s_{i-1},s_{i}]$, for $i=1,\ldots,i_1$ with $s_i := e^i \eps^{-1} L$ and where $i_1$ is the smallest integer such that $s_i \ge C_0 L \ln L$. Note that $i_1 \le 2 \ln\ln L$ as soon as $L$ is large enough.
For any $i\ge 0$, set $a(s_i) := a_L + \frac{\beta}{4} s_i$. We then let $X^i$ be the time-homogeneous diffusion starting from $+\infty$ at time $s_i$ and with parameter $a_-(s_i) := a(s_i) - \frac1{4\sqrt{a(s_i)}}$. We also set $\gamma^i := \inf\{t\ge 0: X^i(s_i+t) = -\infty\}$. We define $A_L$ as the event on which an explosion of $Z_{a_L}$ on $(\eps^{-1} L, C_0L {\ln L}]$ implies the existence of some $i \in \{0,\ldots, i_1-1\}$ such that $X^i$ explodes on $(s_i,s_{i+1}]$.

\begin{lemma}\label{Lemma:AL}
The probability of $A_L$ goes to $1$ as $L\to\infty$.
\end{lemma}
\begin{proof}
Set $\kappa_i = \frac{\ln a(s_i)}{\sqrt{a(s_i)}}$. Let $\cD$ be the event where for all $i\in\{0,\ldots,i_1-1\}$ we have
\begin{equation*}
\sqrt{a(s_i)} - \frac14 \le Z_{a_L}(t) \le \sqrt{a(s_i)} + \frac14\;, \quad\forall t\in [s_i,s_i + 9 \kappa_i]\;,
\end{equation*}
together with
\begin{equation*}
\inf_{t\in [s_i,s_i + 9 \kappa_i]} X^i(t) \ge \sqrt{a(s_i)} - \frac14 \;,\quad  \sup_{t\in [s_i + 2\kappa_i ,s_i + 9 \kappa_i]} X^i(t)\le \sqrt{a(s_i)} + \frac14\;.
\end{equation*}
Recall that $i_1 \le 2 \ln\ln L$. By the forthcoming Lemma \ref{Lemma:XZSqueeze} (note that the diffusion $Z_{a_L}$ on the time interval $[s_i,\infty)$ can be obtained from the diffusion $Z_{a(s_i)}$ on $[0,\infty)$) we have
$$\P(\cD^\cc) \le i_1 C a_L^{-2} \to 0\;.$$
We now work on the event $\cD$. The processes $X^i$ and $Z_{a_L}$ lie in the strip $[(1/2) \sqrt{a(s_i)},(3/2)\sqrt{a(s_i)}]$ on the time-interval $[s_i + 2\kappa_i ,s_i + 9 \kappa_i]$. The difference $D(t) := Z_{a_L}(t) - X^i(t)$ solves
$$ dD(t) = \Big(\frac{\beta}{4}(t-s_i) + \frac1{4\sqrt{a(s_i)}} - (Z_{a_L}(t)+X^i(t)) D(t) \Big)dt\;.$$
If $D(s_i + 2\kappa_i) \ge 0$ then $D$ remains non-negative until the next explosion time of $X^i$. Otherwise, observe that as long as $D$ is negative we have on the time-interval $[s_i + 2\kappa_i ,s_i + 9 \kappa_i]$
$$ dD(t) \ge \Big(\frac1{4\sqrt{a(s_i)}} - 3\sqrt{a(s_i)} D(t)\Big)dt\;,$$
so that a simple computation shows that $D$ becomes positive on the time-interval $[s_i + 2\kappa_i ,s_i + 9\kappa_i]$.\\
Then, monotonicity ensures that $D$ remains non-negative until the next explosion time of $X^i$. Henceforth, if $Z_{a_L}$ explodes on $(s_i,s_{i+1}]$ then necessarily $X^i$ explodes as well.
\end{proof}
With this result at hand, we can prove our proposition.
\begin{proof}[Proof of Proposition \ref{Prop:ExploDelicate}]
By independence we have
\begin{align*}
\P(Z_{a_L}\mbox{ does not explode on }(\eps^{-1} L, C_0 L {\ln L}]) &\ge \P(Z_{a_L}\mbox{ does not explode on }(\eps^{-1} L, C_0 L {\ln L}] ; A_L)\\ 
&\ge\prod_{i=0}^{i_1-1} \P(X^i \mbox{ does not explode on }(s_i,s_{i+1}]) - \P(A_L^\cc)\\
&= \prod_{i=0}^{i_1-1} \P(\gamma^i > (e-1)s_i) - \P(A_L^\cc)\;.
\end{align*}
Note that $\sqrt{a(s_i)} \ge \sqrt{a_L}$. For $\epsilon$ small enough, by \eqref{Eq:mac} we have
$$ \E[\gamma^i] = m(a_-(s_i)) \ge cL e^{\frac12 e^i \eps^{-1}}\;,$$
and therefore
$$ \frac{(e-1)s_i}{m(a_-(s_i))} \le c^{-1}(e-1)e^i \eps^{-1} e^{-\frac12 e^i \eps^{-1}} =: \kappa_i \;.$$
By Proposition \ref{Prop:CVrate} we deduce that for all $L$ large enough we have
$$ \sup_{i=0,\ldots,i_1} \frac1{\kappa_i \vee \ln(a_L)^{-3}} \Big|\ln\Big( \frac{\P(\cE(1) > \kappa_i \vee (\ln a_L)^{-3})}{\P(\gamma_i/m(a_-(s_i)) > \kappa_i \vee (\ln a_L)^{-3})}\Big)\Big| \le \eps\;.$$
We thus get
$$ \P(\gamma^i > (e-1)s_i) \ge e^{-(1+\eps)(\kappa_i \vee (\ln a_L)^{-3})}\;.$$
Recall that $i_1 \le 2 \ln\ln L$. A simple computation then shows that the product over $i\in\{0,\ldots,i_1-1\}$ of the last expression is larger than $1-\eps$ for all $L$ large enough and all $\eps$ small enough.
\end{proof}

\subsection{End of proof}

\begin{lemma}\label{Lemma:NoExploInfinity}
With a probability going to $1$ as $L\to\infty$, provided $C_0>0$ is chosen large enough, the process $Z_{a_L}$ remains in between the deterministic curves
$$ \frac12 \sqrt{a_L + \frac{\beta}{4} t}\quad \mbox{ and }\quad \frac32 \sqrt{a_L + \frac{\beta}{4} t}\;,$$
on the time interval $[C_0 L{\ln L},\infty)$, and therefore does not explode on this time interval.
\end{lemma}
\begin{proof}
Set $N := \lfloor C_0 L \ln L \beta \rfloor$. By the forthcoming Lemma \ref{Lemma:Stabil}, the probability that $Z_{a_L}(N/\beta) \ge 0$ goes to $1$ as $L\to\infty$. On the event where this happens, we know that $Z_{a_L}$ remains above $Z^{(N/\beta,0)}_{a_L}$ until the next explosion time of the latter.\\
Note that $N \sim C_0(8/3)^{1/3} (\ln L)^{2/3}$ as $L\to\infty$. Applying Lemma \ref{Lemma:ZfwdInit} and Lemma \ref{Lemma:ZfwdRecur}, we deduce that the probability that $Z^{(N/\beta,0)}_{a_L}$ does not remain above $\frac12 \sqrt{a_L + \frac{\beta}{4} t}$ on the time interval $[C_0L{\ln L}, \infty)$ is bounded from above by a term of order $\sum_{k\ge N} (e^{-k^{1/3}} + \beta^{-1} e^{-k^{3/2}}) \lesssim e^{-\frac12 N^{1/3}} + \beta^{-1} e^{-\frac12 N^{3/2}}$. The first term goes to $0$ as $L\to\infty$. Since $\ln(1/\beta) \sim \ln L$ as $L\to\infty$, it suffice to take $C_0>0$ large enough for the second term to go to $0$ as $L\to\infty$.\\
The proof of the upper bound follows from exactly the same type of arguments.
\end{proof}

\begin{proof}[Proof of Theorem \ref{Th:Explo}]
To simplify the notations, we consider the case $r=0$: since $m(a_L) = L$ and $m(a_L-r/4\sqrt{a_L})/m(a_L)$ goes to $e^{-r}$ as $L\to\infty$, it is immediate to deduce the general case by a simple time-change. We already know that the convergence of the theorem holds for the vague topology by~\cite[Theorem 4.1]{AllezDumazTW}. To complete the proof, we argue as follows. Fix $\delta > 0$. By Proposition \ref{Prop:ExploDelicate} and Lemma \ref{Lemma:NoExploInfinity}, there exists $\eps >0$ such that for all $L$ large enough, the probability that $Z_{a_L}$ never explodes after time $\eps^{-1} L$ is larger than $1-\delta$ uniformly over all $L$ large enough. This estimate suffices to strengthen the topology in which the aforementioned convergence holds.
\end{proof}

\section{Proofs of the main theorems}\label{Section:Proofs}

To prove our theorems, we introduce a discretization scheme and define approximations of the eigenfunctions: these approximations possess more independence so that they are easier objects to deal with.

\medskip

First of all, we discretize the interval $[0,\infty)$. Let $0=:t^n_0 < t^n_1 < \ldots < t^n_{2^n} := +\infty$ be the points that satisfy
$$ \int_{t^n_j}^{t^n_{j+1}} e^{-t} dt = \frac{1}{2^n}\;,\quad \forall j \in \{0,\ldots,2^n-1\}\;.$$
In other words, $t^n_j$ is the point where the cumulative distribution function of the exponential law reaches $j2^{-n}$: this discretization is adapted to the limiting intensity of the point process of explosion times from Theorem \ref{Th:Explo}. Indeed, this theorem shows that as $L\to\infty$, the number of explosions of the diffusion $Z_{a_L-r/(4\sqrt a_L)}$ in the time interval $[t^n_j L,t^n_{j+1}L]$ converges to a Poisson r.v.~of intensity $2^{-n} e^r$.

\medskip

Second, by Theorem \ref{Th:Explo}, the first eigenvalues of $\cL_\beta$ typically deviate from $a_L$ like $1/\sqrt{a_L}$. Therefore we discretize the axis of eigenvalues by introducing for $\eps > 0$ the grid
$$ \cM_{L,\eps} := \Big\{ a_L + p \frac{\eps}{4\sqrt a_L}: p\in \Z\cap [-1/\eps^2,1/\eps^2]\Big\}\;.$$

\subsection{Convergence of the point process of eigenvalues}\label{Subsec:PPP}

For every $j\in\{0,\ldots,2^n-1\}$ and every $a\in \cM_{L,\eps}$, we introduce the diffusion $Z_{a}^j := Z_{a}^{(t^n_j L, +\infty)}$ and use it to approximate the diffusion $Z_{a}$ on the time interval $[t^n_j L, t^n_{j+1} L]$. The justification behind this approximation is provided by the following lemma, whose proof is postponed to Subsection \ref{Subsec:ApproxZZ}.

\begin{lemma}\label{Lemma:ApproxZZ}
With a probability going to $1$ as $L$ goes to $\infty$ and then $n$ goes to $\infty$, the following holds. For all $a\in\cM_{L,\eps}$ and all $j\in\{0,\ldots,2^n-1\}$:\begin{itemize}
\item the diffusion $Z_{a}$ explodes at most one time on $(t^n_j L,t^n_{j+1} L]$,
\item the diffusion $Z_{a}$ explodes on $(t^n_j L,t^n_{j+1} L]$ if and only if the diffusion $Z_{a}^j$ explodes on $(t^n_j L,t^n_{j+1} L]$.
\end{itemize}
\end{lemma}

Denote by $(q_i)_{i=1\ldots m}$ the elements of $\cM_{L,\eps}$ listed in \textit{decreasing} order $q_1 > q_2 > \ldots > q_m$ and let $r_i$ be such that
$$ q_i = a_L - \frac{r_i}{4\sqrt a_L}\;,\quad i=1,\ldots,m\;.$$

Set $V_j(i) = 1$ if the diffusion $Z_{q_i}^j$ explodes on $(t_j^n L, t_{j+1}^n L]$, and $V_j(i) = 0$ otherwise. We also set $V_j(0) = 0$, $q_0=+\infty$ and $r_0 := -\infty$. We define 
$$ {Q}^{(n)}_L(i) := \sum_{j=0}^{2^n - 1} \Big(V_j(i) - V_j(i-1) \Big)\;,\quad i=1,\ldots,m\;.$$
For every $i$, the r.v.~${Q}^{(n)}_L(i)$ counts the number of intervals $(t^n_j L,t^n_{j+1} L]$ where the diffusion $Z^j_{q_i}$ explodes but the diffusion $Z^j_{q_{i-1}}$ does not. By Lemma \ref{Lemma:ApproxZZ}, $Q_L^{(n)}(i)$ is a good approximation of the total number of explosions of $Z_{q_i}$ minus the total number of explosions of $Z_{q_{i-1}}$ in the large $L$ and $n$ limit.

\begin{lemma}\label{Lemma:CVQn}
The vector $\big({Q}^{(n)}_L(i)\big)_{i=1,\ldots,m}$ converges in distribution as $L\to\infty$ and $n\to\infty$ to a vector of independent Poisson r.v.~with parameters $p_i = \int_{r_{i-1}}^{r_{i}} e^{x} dx$.
\end{lemma}
\begin{proof}
Recall that, for any given $j\in \{0,\ldots,2^n-1\}$, the diffusions $(Z_{q_i}^j, i=1,\ldots,m)$ on the time interval $[t_j^n L, \infty)$ are ordered up to their first explosion times. This implies that for all $j\in\{0,\ldots,2^n-1\}$, the r.v.~$(V_j(i),i =1,\ldots, m)$ satisfy the following monotonicity property:
$$ V_j(1) \le V_j(2) \le \ldots \le V_j(m)\;.$$
Since in addition these r.v.~are $\{0,1\}$-valued, we get the very simple identities:
\begin{align*}
\P\big(V_j(1)=0,\ldots, V_j(i-1)=0, V_j(i)=1,\ldots,V_j(m) = 1\big) &= \P\big(V_j(i)=1\big)-\P\big(V_j(i-1)=1\big)\;,\\
\P\big(V_j(1)=0,\ldots,V_j(m) = 0\big) &= \P\big(V_j(m)=0\big)\;,\\
\P\big(V_j(1)=1,\ldots,V_j(m)=1\big) &= \P\big(V_j(1)=1\big)\;,
\end{align*}
so that the only knowledge of the one-dimensional marginals suffices to determine the law of the vector. By Theorem \ref{Th:Explo}
\begin{equation}\label{Eq:ExploU}
\P\big(V_j(i)=1\big) \rightarrow 1-\exp\big(-2^{-n} e^{r_i}\big)\quad\mbox{ as }L\to\infty\;.
\end{equation}

Furthermore, the $m$-dimensional vectors $(V_j(1), V_j(2), \ldots ,V_j(m))$, $j=0,\ldots,2^n-1$, are independent since they depend on the evolution of the Brownian motion $B$ on disjoint intervals. We then perform the computation of the law of $\big({Q}^{(n)}_L(i)\big)_{i=1,\ldots,m}$. For any given integers $\ell_1,\ldots,\ell_m$, set $\ell = \sum_i \ell_i$. Then
\begin{align*}
\P\big({Q}^{(n)}_L = (\ell_1,\ldots,\ell_m)\big) &= \sum_{\substack{S_1,\ldots,S_m \subset \{0,\ldots,2^n-1\}\\S_i \cap S_{i'} = \emptyset\\ \#S_i = \ell_i}} \prod_{i=1}^m \prod_{j\in S_i} \Big(\P\big(V_j(i)=1\big)-\P\big(V_j(i-1)=1\big)\Big)\\
&\qquad\qquad\quad\qquad\times \prod_{j\notin S_1\cup\ldots\cup S_m} \P(V_j(m)=0)\;.
\end{align*}
Using \eqref{Eq:ExploU}, we deduce that the $L\to \infty$ limit of the last expression equals
\begin{align*}
{2^n \choose \ell_1, \ldots, \ell_m, 2^n-\ell} \prod_{i=1}^m \Big(\exp(-2^{-n} e^{r_{i-1}})-\exp(-2^{-n}e^{r_i})\Big)^{\ell_i}\Big(\exp(-2^{-n} e^{r_m})\Big)^{2^n -\ell}\;.
\end{align*}
Taking the limit as $n\to\infty$, a computation shows that this last quantity converges to
$$ \prod_{i=1}^m \frac{p_i^{\ell_i}}{\ell_i !} e^{-p_i}\;,$$
as required.
\end{proof}

\begin{proof}[First part of the proof of Theorem \ref{Th:Main}]
We define $\cQ_L := \sum_{k\ge 1} \delta_{4\sqrt a_L(\lambda_k+a_L)}$ and we view this object as a r.v.~in the space of measures on $(-\infty,\infty)$ which are finite on all intervals bounded to the right (but possibly unbounded to the left). Note that for every $L$, since there is a smallest eigenvalue, the random measure $\cQ_L$ indeed belongs to this space.\\
We endow this space with the topology that makes continuous the maps $m \mapsto \langle f, m\rangle$ for any continuous and bounded function $f$ with support bounded to the right: in other words, this is the weak topology towards $-\infty$ and the vague topology towards $+\infty$. The reason for this topology is simple: it permits to control the increasing sequence of atom locations of $\cQ_L$.

\smallskip

If we prove that $\cQ_L$ converges in law (for the sigma field associated with this topology) to a Poisson point process of intensity $e^x dx$, then standard arguments ensure that the increasing sequence of its atom locations converges in law for the product topology to the increasing sequence of atom locations of this Poisson point process.

\smallskip

Let us show that for any $\eps > 0$ (recall that $\eps$ controls the mesh of $\cM_{L,\eps}$), the random vector
$$\cQ_L((r_{i-1},r_i])\;,\quad i=1,\ldots,m\;,$$
converges in distribution as $L\to\infty$ to a vector of independent Poisson random variables of intensity $e^{r_i}-e^{r_{i-1}}$. On the event on which the assertions of Lemma \ref{Lemma:ApproxZZ} hold true, we have for every $i\in\{1,\ldots,m\}$:
$$ \cQ_L((r_{i-1},r_i]) = {Q}^{(n)}_L(i)\;,$$
so that Lemma \ref{Lemma:CVQn} yields the desired result.

We deduce from this convergence the tightness of $(\cQ_L)$: indeed the above convergence provides the required control on the mass given by $\cQ_L$ to $(-\infty,r]$ for any given $r$. Furthermore, the marginals of any limiting point are uniquely identified thanks to this convergence: for instance by considering the marginals coming from dyadic points and by choosing $\eps$ appropriately.
\end{proof}

\subsection{Typical diffusions}

In this subsection, we collect several estimates on the diffusions $Z_a$ for $a\in\cM_{L,\eps}$, the proofs of which are postponed to Sections \ref{Sec:Techos} and \ref{Sec:Fine} in order not to interrupt the line of argument. The statements of these estimates are rather long, however, a look at the form of the time-inhomogeneous potential in which $Z_a$ evolves (see Figure \ref{Fig:Yk}) allows to see that these estimates are natural.\\

We rely on the following notations: $\tau^{(i)}_{-\infty}$ denotes the $i$-th explosion time of $Z_a$ and $\tau^{(i)}_{-2\sqrt{a_L}}$ denotes its first hitting time of $-2\sqrt{a_L}$ after the $(i-1)$-th explosion time. Moreover, we adopt the convention $\tau^{(0)}_{-\infty} = 0$ and the notation $ \fint_s^t f := (t-s)^{-1} \int_s^t f$. We also set $t_L := \frac{\ln a_L}{\sqrt{a_L}}$.\\

A typical realization of $Z_a$ for $a\in\cM_{L,\eps}$ behaves as follows:\begin{enumerate}
\item \emph{Entrance.} For any $i\ge 0$, after its $i$-th explosion time, the diffusion comes down from $+\infty$ in an almost deterministic way and quickly reaches a small neighborhood of $\sqrt{a_L}$:
$$ \sup_{t\in (\tau^{(i)}_{-\infty},\tau^{(i)}_{-\infty}+(3/8)t_L]} |Z_a(t) - \sqrt{a_L}\coth(\sqrt{a_L} (t-\tau^{(i)}_{-\infty}))| \le 1 \;.$$
\item \emph{Explosion.} For any $i\ge 1$, after time $\tau^{(i)}_{-2\sqrt{a_L}}$, the diffusion behaves almost deterministically and reaches $-\infty$ within a very short time:
$$ \sup_{t\in (\tau_{-2\sqrt{a_L}}^{(i)},\tau_{-\infty}^{(i)}]} |Z_a(t) - \sqrt{a_L}\coth(\sqrt{a_L}(t-\tau_{-\infty}^{(i)}))| \le 1 \;.$$
\item \emph{Oscillations.} For any $i\ge 0$, in between the explosion times $\tau^{(i)}_{-\infty}$ and $\tau^{(i+1)}_{-\infty}$, the diffusion spends most of its time near $\sqrt{a_L}$:
$$ \fint_{\tau^{(i)}_{-\infty}+(3/8)t_L}^t Z_a(s) ds \in [\sqrt{a_L} -10, \sqrt{a_L}+10]\;,\quad \forall t\in [\tau^{(i)}_{-\infty}+(3/8)t_L,\tau_{-2\sqrt{a_L}}^{(i+1)}\wedge ({\eps^{-2}}L)]\;.$$
\item \emph{Long-time behavior.} The diffusion does not explode after time ${\eps^{-2}}L$.
\end{enumerate}

\medskip

Note that the choice ${\eps^{-2}}$ is relatively arbitrary here: it is taken such that (4) holds true for all $a\in \cM_{L,\eps}$ with a probability $1-\cO(\eps)$.
On the other hand, in estimate (3), the time parameter is taken smaller than $\eps^{-2}L$ for a simple reason: the typical location of the diffusion is given by the bottom of the well of its time-inhomogeneous potential, and the latter remains around $\sqrt a_L$ as long as time is not too large (actually, much smaller than $L\ln L$).

\medskip

Similar estimates hold for the backward diffusion, however the situation is slightly different in that case for the obvious reason that time is run backward and the process explodes to $+\infty$. We then let $\hat{\tau}^{(1)}_{+\infty}$ be the largest time $t\ge 0$ at which $\hat{Z}_a$ hits $+\infty$, and recursively, $\hat{\tau}^{(i)}_{+\infty}$ the largest time $t\in [0,\hat{\tau}^{(i-1)}_{+\infty})$ at which $\hat{Z}_a$ hits $+\infty$. Furthermore, we let $\hat{\tau}^{(i)}_{2\sqrt{a_L}}$ be the largest time $t \in [0, \hat{\tau}^{(i-1)}_{+\infty})$ at which $\hat{Z}_a$ hits $2\sqrt{a_L}$. For convenience we set $\hat{\tau}^{(0)}_{+\infty} = +\infty$.\\
Take $C_0 > 0$ large enough (in view of Lemma \ref{Lemma:NoExploInfinity}). A typical realization of $\hat{Z}_a$ behaves as follows (recall that the quotation marks are used when we view the diffusion evolving backward in time):
\begin{enumerate}
\item \emph{Oscillations at infinity.} On the time-interval $[C_0 L {\ln L}, \infty)$, $\hat{Z}_a \le -(1/2) \sqrt{a_L}$. Then ``after'' $C_0L{\ln L}$ and ``until'' time $\eps^{-2} L$, the diffusion remains most of the time below $-(1/2)\sqrt{a_L}$:
$$ \fint_{\eps^{-2}L}^{t} \hat{Z}_a(s) ds \le -\frac12 \sqrt{a_L}\;,\quad t\in [2\eps^{-2} L, C_0L\ln L]\;.$$
Furthermore, the diffusion does not explode ``until'' time $\eps^{-2} L$.
\item \emph{Entrance.} For any $i\ge 1$, ``after'' its $i$-th explosion time, the diffusion exits from $-\infty$ almost deterministically:
$$\sup_{t\in (\hat\tau^{(i)}_{+\infty}-(3/8)t_L,\hat\tau^{(i)}_{+\infty}]} |\hat{Z}_a(t) - \sqrt{a_L}\coth(\sqrt{a_L} (t-\hat\tau^{(i)}_{+\infty}))| \le 1\;.$$
\item \emph{Explosion.} For any $i\ge 1$, ``after'' time $\hat{\tau}^{(i)}_{2\sqrt{a_L}}$, the diffusion behaves almost deterministically and reaches $+\infty$ within a very short time:
$$\sup_{t\in (\hat{\tau}^{(i)}_{+\infty},\hat{\tau}^{(i)}_{2\sqrt{a_L}}]} |\hat{Z}_a(t) - \sqrt{a_L}\coth(\sqrt{a_L}(t-\hat{\tau}^{(i)}_{+\infty}))| \le 1\;.$$
\item \emph{Oscillations.} The diffusion spends most of its time near $-\sqrt{a_L}$. More precisely for every $i\ge 1$
$$\fint_{t}^{\hat\tau^{(i)}_{+\infty}-(3/8)t_L} \hat{Z}_a(s) ds \in [- \sqrt{a_L}-10,-\sqrt{a_L}+10]\;,\quad \forall t\in [\hat{\tau}^{(i+1)}_{2\sqrt{a_L}},\hat\tau^{(i)}_{+\infty}-(3/8)t_L]\;,$$
and
$$\fint_{t}^{2\eps^{-2}L} \hat{Z}_a(s) ds \in [- \sqrt{a_L}-10,-\sqrt{a_L}+10]\;,\quad \forall t\in [\hat{\tau}^{(1)}_{2\sqrt{a_L}},2\eps^{-2}L]\;.$$
\end{enumerate}

\begin{proposition}\label{Prop:TypicalZ}
There exists $c>0$ such that for all $L$ large enough, the following holds with a probability larger than $1-c \, \eps$: For all $a\in \cM_{L,\eps}$, the diffusions $Z_a$ and $\hat{Z}_a$ satisfy the above estimates.
\end{proposition}
\noindent We refer to Subsection \ref{Subsec:ProofTypicalZ} for the proof of this result.

\medskip

We also need some precise information on the behavior of $Z_a$ when it crosses the barrier of potential of its time-inhomogeneous potential: namely, when it goes from the curve $\sqrt{a+\beta t /4}$ to the curve $-\sqrt{a+\beta t/4}$. Here again, the statement is long and technical, however the underlying observation is relatively simple: the theory of large deviations shows that the behavior of the diffusion $Z_a$, when it crosses the barrier of potential, is essentially deterministic and is given by a hyperbolic tangent.\\
To state precisely the estimates, we need to introduce some notations. For every $a\in\cM_{L,\eps}$ and every $j\in \{0,\ldots,2^n-1\}$, we define $\theta_a^j$ as the first hitting time by $Z_a$ of $-\sqrt{a_L}$ after time $t^n_j L$. We also let $\iota_a^j$ and $\upsilon_a^j$ be the last hitting times of $\sqrt{a_L}$ and $0$ respectively before time $\theta_a^j$. We finally let $\zeta_a^j$ be the first hitting time of $-\infty$ by the diffusion $Z_a$ after time $\theta_a^j$. We call \emph{excursion to $-\sqrt{a_L}$} a portion of the trajectory that starts from $+\sqrt{a_L}$, hits $-\sqrt{a_L}$ and comes back to $+\sqrt{a_L}$ (possibly after an explosion). We refer to Figure \ref{Fig:ZaCross} for an illustration.\\
We take similar definitions for the backward diffusions. We let $\hat{\theta}_a^j$ be the first hitting time of $\sqrt{a_L}$ ``after'' time $t^n_{j+1}L$, that is,
$$ \hat{\theta}_a^j := \sup\{t\in (0,t^n_{j+1}L]: \hat{Z}_a(t) = \sqrt{a_L}\}\;.$$
We then let $\hat{\iota}_a^j$ and $\hat{\upsilon}_a^j$ be the last hitting times of $-\sqrt{a_L}$ and $0$ ``before'' time $\hat{\theta}_a^j$, and we let $\hat{\zeta}_a^j$ be the first hitting time of $+\infty$ ``after'' time $\hat{\theta}_a^j$.\\
To alleviate the notations, we will often omit writing the superscript $j$.

\begin{figure}[!h]
\centering
\includegraphics[width=10cm,height=8cm]{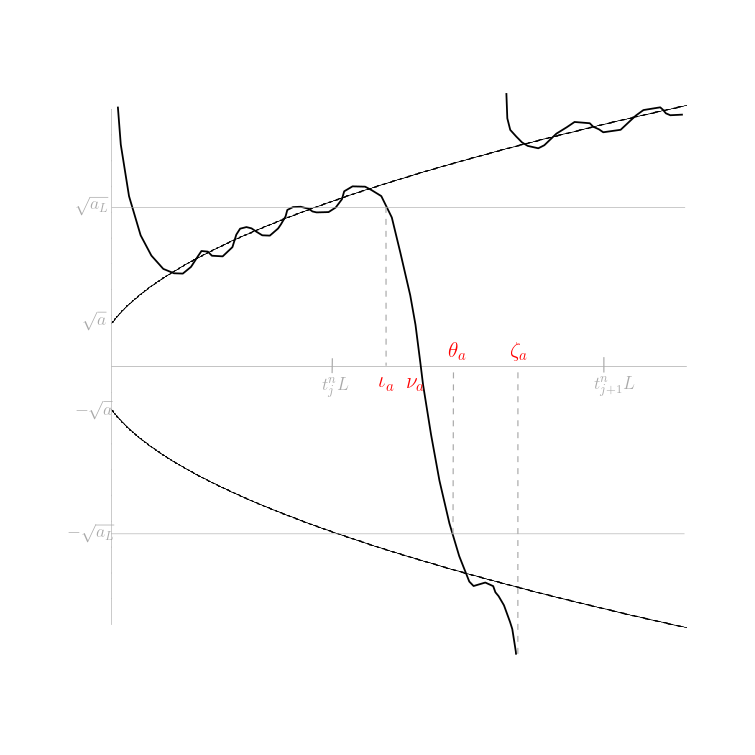}
\caption{\small A very schematic plot of the diffusion $Z_a$ when it crosses its barrier of potential: the trajectory between $\iota_a$ and $\theta_a$ is very close to a hyperbolic tangent.}\label{Fig:ZaCross}
\end{figure}

The statement of the following proposition is long and technical: at first reading, one can go directly to Subsections \ref{Subsec:first} and \ref{Subsec:second} where these estimates are used whenever needed.

\begin{proposition}\label{Prop:TypicalPairZ}
There exist two constants $C,c>0$ such that for all $L$ and $n$ large enough, with a probability larger than $1-c\eps$ the following holds for all $a \le a' \in \cM_{L,\eps}$ and all $j\in \{0,\ldots,2^n-1\}$ such that $t^n_{j+1} < {\eps^{-2}}$.\\
We have $Z_a(t^n_j L) \in [(1/2)\sqrt{a_L},(3/2)\sqrt{a_L}]$ and the diffusions are ordered $Z_a(t^n_j L) \le Z_{a'}(t^n_j L)$. The process $Z_a$ makes at most one excursion to $-\sqrt{a_L}$ on the time interval $[t^n_j L, t^n_{j+1} L]$ and if it does then:\begin{enumerate}
\item \emph{Behavior of $Z_a$.} We have
\begin{align*}
\upsilon_a - \iota_a \ge (3/8)t_L - C\frac{\ln\ln a_L}{\sqrt{a_L}}\;,\\
|\theta_a - \upsilon_a - \frac38 t_L| \le C \frac{\ln\ln a_L}{\sqrt{a_L}}\;.
\end{align*}
Moreover, the diffusion $Z_a$ is close to a hyperbolic tangent near $\upsilon_a$
\begin{align*}
&\sup_{t\in [\iota_a,\theta_a]}|Z_a(t) - \sqrt{a_L} \tanh(-\sqrt{a_L}(t-\upsilon_a))| \le C \frac{\sqrt{a_L}}{\ln a_L} \;.\\
\end{align*}
In addition, if $Z_a$ explodes after $\theta_a$ before coming back to $\sqrt{a_L}$ then $|\zeta_a - \theta_a -(3/8)t_L| \le C (\ln\ln a_L)^2 / \sqrt a_L$.
\item \emph{Coupling with $Z_{a'}$.} We have
\begin{align*}
|Z_{a'}(t)-Z_a(t)| \le 1\;,\quad &t\in [\upsilon_a - (1/16)t_L,\upsilon_a + (1/16)t_L]\;,\\
Z_{a'}(t) \le -\sqrt{a_L} + Ca_L^{3/7}\;,\quad &t\in [\upsilon_a + (1/16)t_L,\theta_a - (1/16)t_L]\;,\\
Z_{a'}(t) \le \sqrt{a_L} - 1\;,\quad &t\in [\theta_a - (1/16)t_L,\theta_a]\;.
\end{align*}
\item \emph{Explosion of $Z_{a'}$.} If in addition $Z_{a'}$ explodes on $[t^n_j L, t^n_{j+1}L]$, then so does $Z_a$ and we have the estimates
\begin{align*}
|\upsilon_a-\upsilon_{a'}| &< \frac{C}{\sqrt{a_L} \ln a_L}\;,\\
|\theta_a-\theta_{a'}| &< C \frac{\ln\ln a_L}{\sqrt{a_L}}\;,\\
|\zeta_a-\zeta_{a'}| &< C \frac{(\ln\ln a_L)^2}{\sqrt{a_L}}\;,
\end{align*}
and $Z_{a'}$ remains below $-\sqrt{a_L}+1$ on $[\theta_{a'},\zeta_{a'}]$. Moreover the explosion times of $Z_a$ and $Z_{a'}$ remain at distance at least $2^{-2n}L$ from $t^n_j L$ and $t^n_{j+1} L$.
\item \emph{Coupling with the backward diffusions.} If there exists $a'' \in \cM_{L,\eps}$ such that $a'' < a$ and $\hat{Z}_{a''}$ does not explode on $[\theta_a + 10 t_L,t^n_{j+1}L]$ then $\hat{Z}_{a'}(t) \le -\sqrt{a_L} + (\ln a_L)/a_L^{1/4}$ for all $t\in [\theta_a,\theta_a+ 5t_L]$, and furthermore for all $t\in [{\theta}_a,t^n_{j+1} L]$ we have
\begin{align*}
&-(3/2) \sqrt{a_L} \le \fint_{\theta_a}^t \hat{Z}_{a'}(s) ds  \le -(1/2) \sqrt{a_L}\;.
\end{align*}
\end{enumerate}
The analogous statements hold for the backward diffusions $\hat{Z}_a$ and $\hat{Z}_{a'}$.
\end{proposition}
\noindent The proof of this proposition can be found in Subsection \ref{Subsec:ProofTypicalPairZ}.

\subsection{The key event}\label{Subsec:Key}

Fix $k\ge 1$ and $\eps > 0$: we aim at controlling the $k$ first eigenvalues / eigenfunctions on an event of probability at least $1-\cO(\eps)$. Recall that our setup relies on the following two parameters: $\eps$, which is the mesh of the approximation grid for the eigenvalues and $n$ which controls the mesh of the approximation grid of $[0,\infty)$.\\

We define $\cE$ as the event on which the following holds:
\medskip

(a) \textbf{Squeezing of the $k$ first eigenvalues.} There exists a random subset
$$\cA= \{\alpha'_{k+1} < \alpha_k < \alpha'_k < \alpha_{k-1} < \ldots < \alpha'_2 < \alpha_1 < \alpha'_1 < \alpha_0\}\;,$$
\hspace{22pt} of $\cM_{L,\eps}$ such that:
$$ -\lambda_{k+1} < \alpha'_{k+1} < \alpha_k < - \lambda_k < \alpha'_k < \alpha_{k-1} < \ldots <\alpha'_2 < \alpha_1 < -\lambda_1 < \alpha'_1 < \alpha_0\;.$$

(b) \textbf{Typical diffusions I.} The estimates of Proposition \ref{Prop:TypicalZ} are satisfied.

\medskip

(c) \textbf{Typical diffusions II.} The estimates of Proposition \ref{Prop:TypicalPairZ} are satisfied.

\begin{proposition}\label{Prop:E}
Fix $\eps > 0$. There exists $C>0$ such that $\varliminf_{n\to\infty} \varliminf_{L\to\infty} \P(\cE) \ge 1-C\eps$.
\end{proposition}
\begin{proof}
The probability of (b) and (c) is already estimated in Propositions \ref{Prop:TypicalZ} and \ref{Prop:TypicalPairZ}. Regarding (a), we already know that $(4\sqrt{a_L}(\lambda_i+a_L))_{i=1,\ldots,k+1}$ converges in law to the $k+1$ first atoms of a Poisson point process of intensity $e^x dx$ on $\R$. Consequently, there exists a constant $c>0$ such that the probability that the spacing between any two consecutive elements of this $(k+1)$-uplet is larger than $3\eps$ is at least $1-c \eps$ uniformly over all $L$ large enough. On the event where this property holds true, we can squeeze two consecutive elements of $\cM_{L,\eps}$ in between two consecutive eigenvalues and (a) follows.
\end{proof}

In the next two subsections, we will work on the event $\cE$ and will establish the convergences stated in Theorems \ref{Th:Main} and \ref{Th:Shape}.

\subsection{Control of the first eigenfunction}\label{Subsec:first}

We aim at controlling the process $\chi_1$, obtained from the first eigenfunction $\varphi_1$ after applying the Riccati transform:
$$ \chi_1(t) = \frac{\varphi_1'(t)}{\varphi_1(t)}\;,\quad t\ge 0\;.$$
We will do that by using typical diffusions $Z_a$ whose parameters $a$ belong to the random subset $\cA$. 

Thanks to (a) of $\cE$, we have $- \lambda_2 < \alpha'_2 < \alpha_1 < -\lambda_1 < \alpha'_1$. Set $a=\alpha_1$. By monotonicity, $Z_{\alpha'_2} \le Z_{a} \le \chi_1$ until the first explosion time of $Z_{\alpha'_2}$, and $\chi_1 \le \hat{Z}_{a}$ ``until'' the first explosion time of $\hat{Z}_{a}$. Since $- \lambda_2 < a < -\lambda_1$ and in view of Corollary \ref{Cor:bc}, the diffusion $Z_{a}$ explodes exactly once and by (c)-(3) its explosion time falls in some interval $[t_j^n L + 2^{-2n}L,t_{j+1}^n L - 2^{-2n}L]$ with $t^n_{j+1} < \eps^{-2}$.

\medskip

Let $\zeta_a,\hat{\zeta}_a$ be the explosion times of $Z_a,\hat{Z}_a$. Let us first prove the following ordering of the stopping times:
\begin{align}
t_j^n L + 2^{-2n}L < \hat \zeta_a < \hat \theta_a < \theta_a < \zeta_a < t_{j+1}^n L - 2^{-2n}L\,.\label{ineqstoppingtimesfirst}
\end{align}

We already know that $t_j^n L + 2^{-2n}L < \zeta_a < t_{j+1}^n L - 2^{-2n}L$. By Lemma \ref{Lemma:Symmetry}, the explosion time of $\hat{Z}_{a}$ lies before the explosion time of $Z_{a}$, and by monotonicity, in between those two explosion times we have $Z_a \le \chi_1 \le \hat{Z}_a$. By (c), we know that $Z_a(t^n_j L) \in [(1/2)\sqrt{a_L},(3/2)\sqrt{a_L}]$ and $\hat{Z}_a(t^n_j L) \in [-(3/2) \sqrt{a_L},-(1/2) \sqrt{a_L}]$ so that necessarily $\hat{Z}_a(t^n_j L) < Z_a(t^n_j L)$. Therefore the explosion time of $\hat{Z}_a$ must lie in $(t^n_j L+ 2^{-2n}L, \zeta_a]$.

In order to see that the diffusion $\hat Z_a$ does not reach $\sqrt{a_L}$ too early, we use the diffusion $\hat{Z}_{\alpha'_2}$. The latter cannot explode on $[\theta_a + 10 t_L,t^n_{j+1} L]$. Indeed, if it exploded there then by (c)-(3) the explosion time of the diffusion $\hat{Z}_a$ would lie in $[\theta_a + 9 t_L,t^n_{j+1} L]$ and since $\zeta_a < \theta_a + 9 t_L$ by (c)-(1), this would contradict the inequality $\hat{\zeta}_a \le \zeta_a$.

By (c)-(4) applied with $a'' = \alpha'_2$ and $a'=a$, we have $\hat{\theta}_a \notin [\theta_a,\theta_a + 5t_L]$. By (c)-(1), we know that $|\theta_a-\zeta_a|$ and $|\hat{\theta}_a-\hat{\zeta}_a|$ are less than $t_L$: in order not to contradict the inequality $\hat{\zeta}_a \le \zeta_a$ we see that $\hat{\theta}_a$ cannot lie to the right of $\theta_a + 5t_L$, and therefore satisfies $\hat{\theta}_a < \theta_a$. It finishes the proof of the inequalities \eqref{ineqstoppingtimesfirst}.\\

Let us show that for all $t\in [0,\hat\theta_a]$
\begin{equation}\label{Eq:fintZa} \fint_t^{\hat{\theta}_a} Z_{a}(s) ds \ge \frac14 \sqrt{a_L}\;.\end{equation}
By (b)-Entrance and (b)-Oscillations applied to $Z_a$, and by (c)-(4) applied to $\hat{Z}_a$ and $\hat{Z}_{a''}$ (using that $Z_{a''}$ does not explode on $[t_j^n, \hat \theta_a-10 t_L]$), we have
\begin{align*}
\sqrt{a_L}-1 \le Z_a(t)\;,\quad &\forall t\in [0,(3/8)t_L]\;,\\
\fint_{(3/8)t_L}^t Z_a(s) ds \in [\sqrt{a_L} -10, \sqrt{a_L}+10]\;,\quad &\forall t\in [(3/8)t_L,\hat{\theta}_a]\;,\\
(1/2) \sqrt{a_L} \le \fint_t^{\hat{\theta}_a} {Z}_{a}(s) ds  \le (3/2) \sqrt{a_L}\;,\quad &\forall t\in [t^n_j L, \hat{\theta}_a]\;.
\end{align*}
If $t\in [t^n_jL,\hat{\theta}_a]$, then \eqref{Eq:fintZa} immediately follows. On the other hand, if $t\in [(3/8)t_L,t^n_j L]$ then
\begin{align*}
\int_t^{\hat{\theta}_a} Z_{a}(s) ds &= -\int_{(3/8)t_L}^t Z_a(s) + \int_{(3/8)t_L}^{t^n_j L} Z_a(s) + \int_{t^n_j L}^{\hat{\theta}_a} {Z}_{a}(s) ds\\
&\ge -(\sqrt{a_L}+10)(t-\frac38 t_L) +  (\sqrt{a_L}-10)(t^n_j L - \frac38 t_L) + \frac12 \sqrt{a_L}(\hat{\theta}_a - t^n_j L) \\
&\ge -20(t-\frac38 t_L) + (\sqrt{a_L} - 10)(t^n_jL - t) + \frac12 \sqrt{a_L}(\hat{\theta}_a - t^n_j L)\\
&\ge -20(t-\frac38 t_L) + \frac12 \sqrt{a_L}(\hat{\theta}_a-t)\;.
\end{align*}
Note that by \eqref{ineqstoppingtimesfirst}, we have $\hat{\theta}_a-t \ge \hat{\theta}_a-t^n_jL \ge 2^{-2n} L$. Note also that $t-\frac38 t_L \le t^n_j L \le \eps^{-2} L$. Therefore the last quantity is larger than $\frac14 \sqrt{a_L}(\hat{\theta}_a-t)$ as required. This proves \eqref{Eq:fintZa} in that case. Finally, if $t\in [0,(3/8)t_L]$ then the bound we just proved together with the inequality $\sqrt{a_L}-1 \le Z_a(t)$ that holds for all $t\in [0,(3/8)t_L]$ allows to conclude. We have therefore proven \eqref{Eq:fintZa}.

Note that $\varphi_1(t) = \varphi_1(\hat{\theta}_a) \exp(-\int_{t}^{\hat{\theta}_a} \chi_1(s) ds)$ for all $t\in [0,\hat\theta_a]$. Since $\chi_1$ remains above $Z_{a}$ on $[0,\hat\theta_a]$, we obtain:
\begin{equation}\label{Eq:EigenStart}
\frac{\varphi_1(t)}{\varphi_1(\hat\theta_a)} \le e^{-\frac14 \sqrt{a_L}(\hat\theta_a-t)}\;,\quad t\in  [0,\hat\theta_a]\;.
\end{equation}

Similarly, combining (b)-Oscillations at infinity, (b)-Oscillations and (c)-(4), we deduce that for all $t\in [\theta_a,\infty)$ we have
$$ \fint_{{\theta}_a}^t \hat{Z}_{a}(s) ds \le -\frac14 \sqrt{a_L}\;.$$
Since $\chi_1$ remains below $\hat{Z}_a$ on $[\theta_a,\infty)$, we get
\begin{equation}\label{Eq:EigenEnd}
\frac{\varphi_1(t)}{\varphi_1(\theta_a)} \le e^{-\frac14 \sqrt{a_L}(t-\theta_a)}\;,\quad t\in  [\theta_a,\infty)\;.
\end{equation}

It remains to control $\varphi_1$ on $[\hat\theta_a,\theta_a]$. Set $a=\alpha_1$ and $a'=\alpha_1'$. On this interval, we have $Z_a \le \chi_1 \le Z_{a'}$. Using (c)-(1) and (c)-(2), we deduce that for all $t\in [\upsilon_a - (1/16)t_L,\upsilon_a + (1/16)t_L]$
$$ -C \frac{\sqrt{a_L}}{\ln a_L} \le \chi_1(t) - \sqrt{a_L}\tanh(-\sqrt{a_L}(t-\upsilon_a)) \le 2C \frac{\sqrt{a_L}}{\ln a_L}\;.$$
We deduce that for all $t\in[\upsilon_a - (1/16)t_L,\upsilon_a + (1/16)t_L]$
\begin{equation}\label{Eq:EigenBulk}
\frac1{\cosh(\sqrt{a_L}(t-\upsilon_a))}(1-2C|t-\upsilon_a|\frac{\sqrt{a_L}}{\ln a_L}) \le \frac{\varphi_1(t)}{\varphi_1(\upsilon_a)} \le \frac1{\cosh(\sqrt{a_L}(t-\upsilon_a))}(1+2C|t-\upsilon_a|\frac{\sqrt{a_L}}{\ln a_L})\;.
\end{equation}
By (c)-(2), we also deduce that $\chi_1$ remains below $-(1/2) \sqrt{a_L}$ on the time interval $[\upsilon_a +(1/16)t_L,\theta_a -(1/16)t_L]$ which is of length $(1/4 + o(1))t_L \ge (1/5)t_L$ thanks to (c)-(1). Consequently, $|\varphi_1|$ is decreasing there and satisfies
$$ |\varphi_1(\theta_a-(1/16)t_L)| \le |\varphi_1(\upsilon_a + (1/16)t_L)| e^{-\sqrt{a_L} t_L / 10}\;.$$
Again by (c)-(2), we know that $\chi_1$ remains below $\sqrt a_L$ on $[\theta_a-(1/16)t_L,\theta_a]$ and therefore
$$ \sup_{t\in[\theta_a-(1/16)t_L,\theta_a]} |\varphi_1(t)| \le |\varphi_1(\upsilon_a + (1/16)t_L)|\;.$$
Putting everything together, we deduce that all the points where $|\varphi_1|$ reaches its maximum over $[\upsilon_a - (1/16)t_L,\theta_a]$ lie at distance at most $4C /(\sqrt a_L \ln a_L)$ from $\upsilon_a$.\\
Using the very same arguments but on the backward diffusions, we deduce that the same result holds over $[\hat\theta_a,\hat\upsilon_a + (1/16)t_L]$ and with $\upsilon_a$ replaced by $\hat{\upsilon}_a$.\\
If we show that $[\hat{\upsilon}_a - \frac1{\sqrt{a_L}},\hat{\upsilon}_a + \frac1{\sqrt{a_L}}] \subset [\upsilon_a - (1/16)t_L,\theta_a]$ and $[{\upsilon}_a - \frac1{\sqrt{a_L}},{\upsilon}_a + \frac1{\sqrt{a_L}}] \subset [\hat\theta_a,\hat\upsilon_a + (1/16)t_L]$, then we will deduce that $\upsilon_a$ and $\hat{\upsilon}_a$ lie at distance at most $4C /(\sqrt a_L \ln a_L)$ from each other. By symmetry, we only give the details on the first inclusion.\\
Since $\hat{Z}_a$ remains above $Z_a$ over $[\hat{\zeta}_a, \theta_a]$, and that $Z_a$ is bounded from below by $\sqrt{a_L} \tanh(-\sqrt{a_L}(t-\upsilon_a)) - C \sqrt{a_L} / \ln a_L$ on $[\iota_a,\upsilon_a]$ we deduce that $\hat{\theta}_a > \iota_a$ and $\hat\upsilon_a > \upsilon_a -  2C/(\sqrt{a_L} \ln a_L)$. As a consequence $\hat{\upsilon}_a - \frac1{\sqrt{a_L}} > \upsilon_a - (1/16)t_L$. On the other hand, applying again (c)-(4) with $a'' = \alpha'_2$ and $a'=a$, we deduce that $\hat{Z}_{a}(t) \le -\sqrt{a_L} + (\ln a_L)/a_L^{1/4}$ for all $t\in [\theta_a,\theta_a + 5t_L]$. Recall that $\hat{\theta}_a < \theta_a$. By (c)-(1) and (c)-(2), we see that $\hat{Z}_{a}(t) \ge -\sqrt{a_L} +1$ on $[\hat{\theta}_a,\hat{\upsilon}_a + \frac1{\sqrt{a_L}}]$. Consequently, we must have the inequality $\hat{\upsilon}_a + \frac1{\sqrt{a_L}} \le \theta_a$. The first inclusion follows.\\
We have therefore proven that $\upsilon_a$ and $\hat{\upsilon}_a$ lie at distance at most $4C /(\sqrt a_L \ln a_L)$ from each other\\

Consequently
\begin{equation}\label{Eq:EigenBulk2}
\sup_{t\in[\hat{\theta}_a, \upsilon_a - (1/16)t_L] \cup [\upsilon_a + (1/16)t_L,\theta_a]} |\varphi_1(t)| \le |\varphi_1(\upsilon_a - (1/16)t_L)| \vee |\varphi_1(\upsilon_a + (1/16)t_L)|\;.
\end{equation}

Putting together \eqref{Eq:EigenBulk}, \eqref{Eq:EigenBulk2}, \eqref{Eq:EigenStart} and \eqref{Eq:EigenEnd}, we deduce that all the points where $|\varphi_1|$ reaches its global maximum, in particular $U_1$, lie at distance at most $4C / (\sqrt{a_L} \ln a_L)$ from $\upsilon_a$. Integrating \eqref{Eq:EigenBulk} and \eqref{Eq:EigenBulk2} we get the estimate:
$$ m_1([\hat\theta_a/L,\theta_a/L]) = \varphi_1^2(U_1) \frac2{\sqrt{a_L}} (1+o(1))\;.$$
On the other hand, \eqref{Eq:EigenStart} and \eqref{Eq:EigenEnd} yield
$$ m_1([0,\hat\theta_a/L]) \le \varphi_1(\hat\theta_a)^2 \cO(1/\sqrt{a_L})\;,\quad m_1([\theta_a/L,\infty)) \le \varphi_1(\theta_a)^2 \cO(1/\sqrt{a_L})\;.$$
By \eqref{Eq:EigenBulk} and \eqref{Eq:EigenBulk2}, we deduce that $|\varphi_1(\hat\theta_a)|$ and $|\varphi_1(\theta_a)|$ are negligible compared to $|\varphi_1(U_1)|$. Since $m_1$ is a probability measure, this ensures that
$\varphi_1^2(U_1) \sim \sqrt{a_L}/2$, that $m_1$ is asymptotically as close as desired to $\delta_{U_1/L}$ and that $|\varphi_1|$, appropriately rescaled around $U_1$, converges to the inverse of a hyperbolic cosine. 

Regarding the behavior of the Brownian motion around $U_1$, using the identity
$$ \chi_1(t) = \chi_1(U_1) + \int_{U_1}^t (-\lambda_1+\frac{\beta}{4}s - \chi_1(s)^2)ds + B(t) - B(U_1)\;,$$
 and the fact that $\chi_1$ is close to a hyperbolic cosine, a simple computation yields the asserted convergence. This completes the proof of Theorems \ref{Th:Main} and \ref{Th:Shape} regarding the first eigenfunction, except for the limiting law of the localization center which will be proven in Subsection \ref{Subsec:Expo}.

\subsection{Control of the $i$-th eigenfunction}\label{Subsec:second}

We treat in detail the case $i=2$, since the general case follows from exactly the same arguments combined with a simple recursion. The diffusion $Z_{\alpha_2}$ explodes twice while the diffusion $Z_{\alpha'_2}$ explodes only once. There exist $j_1 < j_2$ such that the two explosion times of $Z_{\alpha_2}$ fall within $[t_{j_1}^n L,t_{j_1+1}^n L]$ and $[t_{j_2}^n L,t_{j_2+1}^n L]$, and $t_{j_1+1}^n, t^n_{j_2+1} < \eps^{-2}$. By (c)-(3), the explosion time of $Z_{\alpha'_2}$ falls within one of these two intervals. Without loss of generality, let us assume that it falls in the first interval.

\medskip

On $[t_{j_1}^n L,t_{j_1+1}^n L]$, we use the ordering $Z_{\alpha_2} \le \chi_2 \le Z_{\alpha'_2}$ that holds up to the first explosion time of $Z_{\alpha_2}$, together with the estimates (c)-(1) and (c)-(3) to deduce that
$$ \chi_2(t) \ge \sqrt{a_L} \tanh(-\sqrt{a_L}(t-\upsilon_1)) - C\frac{\sqrt{a_L}}{\ln a_L}\;,\quad \forall t\in [\iota_1,\theta_1]\;,$$
and
$$ \chi_2(t) \le \sqrt{a_L} \tanh(-\sqrt{a_L}(t-\upsilon_1)) + 2C\frac{\sqrt{a_L}}{\ln a_L}\;,\quad \forall t\in [\iota'_1,\theta'_1]\;.$$
Here $\iota_1, \theta_1$ and $\iota_1',\theta_1'$ are shorthands for $\iota_{\alpha_2}^{j_1},\theta_{\alpha_2}^{j_1}$ and $\iota_{\alpha_2'}^{j_1},\theta_{\alpha_2'}^{j_1}$. By monotonicity, we necessarily have $\theta_1 < \theta'_1$. Consequently, we get
$$ \sup_{t\in [\iota_1 \vee \iota'_1,\theta_1]} |\chi_2(t) - \sqrt{a_L} \tanh(-\sqrt{a_L}(t-\upsilon_1))| \le 2C\frac{\sqrt{a_L}}{\ln a_L}\;,$$
so that for all $t\in [\iota_1 \vee \iota'_1,\theta_1]$ we have
\begin{equation}\label{Eq:EigenBulk221}
\frac1{\cosh(\sqrt{a_L}(t-\upsilon_1))}(1-2C|t-\upsilon_1|\frac{\sqrt{a_L}}{\ln a_L}) \le \frac{\varphi_2(t)}{\varphi_2(\upsilon_1)} \le \frac1{\cosh(\sqrt{a_L}(t-\upsilon_1))}(1+2C|t-\upsilon_1|\frac{\sqrt{a_L}}{\ln a_L})\;.
\end{equation}
By (b)-Entrance, we deduce that
$$ \sup_{t\in (0,(3/8)t_L]} |\chi_2(t) - \sqrt{a_L}\coth(\sqrt{a_L}t)| \le 1\;,$$
By (b)-Oscillations, we obtain for all $t\in [(3/8)t_L,\theta_1]$
$$ \fint_{(3/8)t_L}^{t} \chi_2(s) ds \in [\sqrt{a_L}-10, \sqrt{a_L}+10]\;,$$
Therefore all the points where $|\varphi_2|$ reaches its maximum over $ [0,\theta_1]$ lie at a distance negligible compared to $L$ from $\theta_1$.\\
To control the eigenfunction after time $\theta_1$, the situation is slightly different from the case of the first eigenfunction. We use the fact that $Z_{\alpha_2}$ and $Z_{\alpha'_2}$ remain close to each other and explode within a time of order $(3/8)t_L$ by (c)-(3) and (b)-Explosion, to deduce that
\begin{equation}\label{Eq:EigenEnd2}
-2\sqrt{a_L} \le \chi_2(t) \le -\frac12 \sqrt{a_L}\;,\quad t\in [\theta_1,\tau_{-2\sqrt{a_L}}(\chi_2)]\;,
\end{equation}
and
$$ \sup_{t\in (\tau_{-2\sqrt{a_L}}(\chi_2),z_1]} |\chi_2(t) - \sqrt{a_L}\coth(\sqrt{a_L} (t-z_1))| \le 1\;.$$
where $z_1$ is the first explosion time of $\chi_2$. Regarding this second estimate, note that $z_1$ falls in between the two explosion times of $Z_{\alpha_2}$ and $Z_{\alpha'_2}$, and that these two times are at a distance negligible compared to $t_L$ from each other by (c)-(3): consequently the control on the Brownian motion required to establish the second estimate is granted on the event $\cE$.\\ We deduce from these bounds that
\begin{align*}
\frac{\varphi_2(t)}{\varphi_2(\theta_1)} &\le e^{-\frac12 \sqrt{a_L}(t-\theta_1)}\;,\quad t\in  [\theta_1,z_1]\;,\\
\frac{\varphi_2(t)}{\varphi_2(\theta_1)} &\ge e^{-2 \sqrt{a_L}(t-\theta_1)}\;,\quad t\in  [\theta_1,\tau_{-2\sqrt{a_L}}(\chi_2)]\;,
\end{align*}
and
$$ \varphi_2(t) = \varphi_2'(z_1) \frac{\sinh(\sqrt{a_L}(t-z_1))}{\sqrt{a_L}}(1+o(1))\;,\quad t\in[\tau_{-2\sqrt{a_L}}(\chi_2),z_1]\;.$$
All these arguments suffice to obtain the following (rough) bound:
\begin{equation}\label{Eq:m2z1}
m_2([0,z_1]) \le \big(\varphi_2'(z_1)\big)^2 e^{o(L) \sqrt{a_L}}\;,
\end{equation}
for all $L$ large enough.

\medskip

After time $z_1$, the process $\chi_2$ comes down from $+\infty$ in an almost deterministic way:
 $$ \sup_{t\in (z_1,z_1+(3/8)t_L]} |\chi_2(t) - \sqrt{a_L}\coth(\sqrt{a_L} (t-z_1))| \le 1\;.$$
Indeed, the proof of this estimate for the diffusions $Z_a$ relies on a control of the Brownian motion on an interval of length $t_L$ right after the explosion time: on the event $\cE$ we do have such a control since $z_1$ is very close to the explosion times of $Z_{\alpha_2}$ and $Z_{\alpha'_2}$. From this estimate, we deduce that
$$ \varphi_2(t) = \varphi_2'(z_1) \frac{\sinh(\sqrt{a_L}(t-z_1))}{\sqrt{a_L}}(1+o(1))\;,\quad t\in[z_1,z_1+(3/8)t_L]\;.$$

Let $\theta_2,\upsilon_2$ be shorthands for $\theta_{a_2}^{j_2},\upsilon_{a_2}^{j_2}$. On the time interval $[z_1,\infty)$, it suffices to apply the same arguments as for the first eigenfunction in order to show that
\begin{align}\label{Eq:EigenStart2}
\frac{\varphi_2(t)}{\varphi_2(\hat\theta_2)} &\le e^{-\frac14 \sqrt{a_L}(t-\hat\theta_2)}\;,\quad t\in  [z_1,\hat\theta_2]\;,\\
\frac{\varphi_2(t)}{\varphi_2(\theta_2)} &\le e^{-\frac14 \sqrt{a_L}(t-\theta_2)}\;,\quad t\in  [\theta_2,\infty)\;.
\end{align}
as well as, for all $t\in[\upsilon_2 - (1/16)t_L,\upsilon_2 + (1/16)t_L]$
\begin{equation}\label{Eq:EigenBulk22}
\frac1{\cosh(\sqrt{a_L}(t-\upsilon_2))}(1-2C|t-\upsilon_2|\frac{\sqrt{a_L}}{\ln a_L}) \le \frac{\varphi_2(t)}{\varphi_2(\upsilon_2)} \le \frac1{\cosh(\sqrt{a_L}(t-\upsilon_2))}(1+2C|t-\upsilon_2|\frac{\sqrt{a_L}}{\ln a_L})\;,
\end{equation}
and
\begin{equation}\label{Eq:EigenBulk222}
\sup_{t\in[\hat{\theta}_2, \upsilon_2 - (1/16)t_L] \cup [\upsilon_2 + (1/16)t_L,\theta_2]} |\varphi_2(t)| \le |\varphi_2(\upsilon_2 - (1/16)t_L)| \vee |\varphi_2(\upsilon_2 + (1/16)t_L)|\;.
\end{equation}
These estimates ensure that all the points where $|\varphi_2|$ reach its maximum over $[z_1,\infty)$ lie at a distance smaller than $4C/(\ln a_L \sqrt{a_L})$ from $\upsilon_2$.

By (c)-(3) (applied to $\hat{Z}_{\alpha_2}$), we know that $\hat{\theta}_2$ lies at distance at least $2^{-2n} L$ from $z_1$ so that the previous estimates ensure that
\begin{align*}
|\varphi_2(\upsilon_2)| &\ge |\varphi_2(\hat{\theta}_2)| \ge  |\varphi_2(z_1+\frac38 t_L)| \exp(\frac14 \sqrt{a_L}(\hat{\theta}_2-z_1-\frac38 t_L))\;,\\
|\varphi_2(z_1+\frac38 t_L)| &= |\varphi_2'(z_1)| a_L^{-1/8} (1+o(1))\;.
\end{align*}
We thus deduce that
\begin{equation}\label{Eq:Derivups}
|\varphi_2(\upsilon_2)| \ge |\varphi_2'(z_1)| \exp({\frac18 \sqrt{a_L} 2^{-2n} L})\;.
\end{equation}
As a consequence all the points where the global maximum of $\varphi_2$ is attained, in particular $U_2$, lie at a distance smaller than $4C/(\ln a_L \sqrt{a_L})$ from $\upsilon_2$. Consequently,
$$ m_2([\hat{\theta}_2/L,\theta_2/L]) = \varphi_2^2(U_2) \frac2{\sqrt{a_L}} (1+o(1))\;,$$
and
$$ m_2([z_1/L,\infty)\backslash[\hat{\theta}_2/L,\theta_2/L]) \ll  \varphi_2^2(U_2/L) \frac2{\sqrt{a_L}}\;.$$
Furthermore \eqref{Eq:EigenBulk22} gives the convergence towards the inverse of a hyperbolic cosine, and a simple computation gives the convergence of the rescaled Brownian motion near $U_2$ (denoted $b_{2,\beta}$ in Theorem \ref{Th:Shape}). Putting together \eqref{Eq:m2z1} and \eqref{Eq:Derivups}, we deduce that $m_2$ gives a negligible mass to $[0,z_1/L]$, and is (asymptotically in $L$) as close as desired to a Dirac mass at $U_2/L$.

\subsection{Convergence towards exponential r.v.}\label{Subsec:Expo}

Recall that $(\Lambda_i,I_i)_{i\ge 1}$ are the atoms of a Poisson point process on $\R\times\R_+$ of intensity $e^x e^{-t} dx\otimes dt$. We already know that $(4\sqrt{a_L}(\lambda_i+a_L))_{i\ge 1}$ converges in law to $(\Lambda_i)_{i\ge 1}$. The goal of this subsection is to prove that $(4\sqrt{a_L}(\lambda_i+a_L),U_i/L)_{i\ge 1}$ converges in law to $(\Lambda_i,I_i)_{i\ge 1}$.\\

Let $\nu$ be the law of $(\Lambda_i,I_i)_{1\le i \le k}$. Let $(\cI_i)_{i=1,\ldots,k}$ be a collection of $k$ disjoint closed intervals of $\R$ such that\footnote{Here we mean that any point in $\cI_1$ is smaller than any point in $\cI_2$ and so on.} $\cI_1 < \cI_2 <\ldots < \cI_k$ and $\cI_k$ is unbounded to the right. Let also $(\cU_i)_{i=1,\ldots,k-1}$ be a collection of disjoint closed intervals in $\R_+$ and set $\cU_k = \R_+$. If we show that as $L\to\infty$
\begin{equation}\label{Eq:CVExpo}
\P\Big((4\sqrt{a_L}(\lambda_i+a_L),U_i/L)_{i=1,\ldots,k} \in \prod_{i=1}^k \cI_i\times\cU_i\Big)\to \nu\Big(\prod_{i=1}^k \cI_i\times\cU_i\Big)\;,
\end{equation}
then (recall that $k$ is arbitrary) standard arguments yield the convergence of $(4\sqrt{a_L}(\lambda_i+a_L),U_i/L)_{i\ge 1}$ to $(\Lambda_i,I_i)_{i\ge 1}$ as stated in Theorem \ref{Th:Main}.

\medskip

Consider a ``microscopic'' product set of the form
$$ \cC = \Big(\prod_{i=1}^{k-1} ((b_i-1) \eps, b_i \eps] \times (t^n_{j_i},t^n_{j_i+1}] \Big) \times \big((b_k\eps,\infty) \times \R_+\big)\;,$$
for some distinct $j_i \in \{0,\ldots, 2^{n}-1\}$ and some $b_i \in \Z \cap [-(1/\eps^2),1/\eps^2]$ satisfying $b_1 < b_2 < \ldots < b_k$. Recall that $q_1 > \ldots > q_m$ denote the elements of $\cM_{L,\eps}$ in decreasing order. There exist $\ell_1 < \ell_2 < \ldots < \ell_k$ such that $q_{\ell_i} = a_L - b_i \eps/(4\sqrt{a_L})$.

\smallskip

Let $\cG$ be the event implicitly defined in Lemma \ref{Lemma:ApproxZZ} and recall the event $\cE$ from Subsection \ref{Subsec:Key}. Recall also the r.v.~$V_j(i)$ defined in Subsection \ref{Subsec:PPP}. On the event $\cG\cap \cE$, we claim that
$$ \big\{ (4\sqrt{a_L}(\lambda_i+a_L),U_i/L)_{i=1,\ldots,k}  \in \cC \big\}\;,$$
coincides with
$$ \big\{ V_{j_i}(\ell_i)-V_{j_i}(\ell_i-1)=1,\;\;\;\forall i\in\{1,\ldots,k-1\} \,;\quad V_j(\ell_k) = 0,\;\;\; \forall j\notin \{j_1,\ldots,j_{k-1}\}\big\}\;.$$
Indeed, on the event $\cG$, the latter event coincides with the event where:\begin{itemize}
\item $Z_{q_1}, \ldots, Z_{q_{\ell_1-1}}$ do not explode on $[0,\infty)$,
\item $Z_{q_{\ell_1}}, \ldots, Z_{q_{\ell_2-1}}$ explode once on $[0,\infty)$ and their explosion times lie in $(t^n_{j_1}L,t^n_{j_1+1} L]$,
\item $\ldots$
\item $Z_{q_{\ell_{k-2}}}, \ldots, Z_{q_{\ell_{k-1}-1}}$ explode $k-2$ times on $[0,\infty)$ and their explosion times lie in $(t^n_{j_1}L,t^n_{j_1+1} L]$,$\ldots$, $(t^n_{j_{k-2}}L,t^n_{j_{k-2}+1}L]$,
\item $Z_{q_{\ell_{k-1}}}, \ldots, Z_{q_{\ell_k}}$ explode $k-1$ times on $[0,\infty)$ and their explosion times lie in $(t^n_{j_1}L,t^n_{j_1+1}L]$, $\ldots$, $(t^n_{j_{k-1}}L,t^n_{j_{k-1}+1}L]$.
\end{itemize}
In turn, on $\cE$, this event coincides with the first event of the claim, thus concluding the proof of the claim.

\smallskip

A direct computation shows that
$$ \nu(\cC) = e^{-e^{b_k\eps}} \prod_{i=1}^{k-1} 2^{-n}(e^{b_i\eps} - e^{(b_i-1)\eps}) \;.$$

From the arguments in the proof of Lemma \ref{Lemma:CVQn}, we deduce that
\begin{align*}
&\lim_{L\to\infty} \P\big( V_{j_i}(\ell_i)-V_{j_i}(\ell_i-1)=1,\;\;\;\forall i\in\{1,\ldots,k-1\} \,;\quad V_j(\ell_k) = 0,\;\;\; \forall j\notin \{j_1,\ldots,j_{k-1}\}\big)\\
&= e^{-(1-(k-1)2^{-n})e^{b_k\eps}} \prod_{i=1}^{k-1} (e^{-2^{-n} e^{(b_i-1)\eps}} - e^{-2^{-n} e^{b_i\eps}})\\
&=\nu(\cC) e^{(k-1)2^{-n} e^{b_k\eps}} \prod_{i=1}^{k-1} \delta_i\;,
\end{align*}
where
$$ \delta_i = \frac{e^{-2^{-n} e^{(b_i-1)\eps}} - e^{-2^{-n} e^{b_i\eps}}}{2^{-n}(e^{b_i\eps} - e^{(b_i-1)\eps})}\;.$$
As $n\to\infty$ we have
$$ e^{(k-1)2^{-n} e^{b_k\eps}} \prod_{i=1}^{k-1} \delta_i \to 1\;,$$
uniformly over all $b_1\eps,\ldots,b_k\eps$ in a compact set.

\medskip

Then, one can approximate from above and below (for the inclusion of sets) any set $\prod_{i=1}^k \cI_i\times\cU_i$ as above by the union of $\cO(2^{nk}\eps^{-k})$ microscopic sets and use the previous convergence, together with the fact that $\P(\cG\cap\cE)$ is of order $1 - \cO(\eps)$ for all $L$ and $n$ large enough, to deduce \eqref{Eq:CVExpo}.

\section{Simple estimates on $Z_a$}\label{Sec:Techos}

In this section, we provide some simple estimates on the diffusion $Z_a$ and we prove Lemma \ref{Lemma:ApproxZZ}, and Proposition \ref{Prop:TypicalZ}.\\
%
%

All the estimates that we need concern the process $Z_a$, for some $a\in\cM_{L,\eps}$ and on the time interval $[0,C_0 L {\ln L}]$ for some large enough constant $C_0>0$ (recall that after time $C_0 L{\ln L}$ these processes are almost deterministic by Lemma \ref{Lemma:NoExploInfinity}). We therefore introduce $\tilde{\cM}_{L,\eps}$ as the smallest interval that contains all points
$$ a + \frac{\beta t}{4}\;,\quad t\in [0, C_0 L {\ln L}]\;,\quad a\in \cM_{L,\eps}\;.$$
Recall that $\beta =(L \sqrt{a_L})^{-1} (1+o(1))$. There exists $C>0$ such that $\sup_{a\in \tilde{\cM}_{L,\eps}} a \le C a_L$.\\
We also introduce $\bar{\cM}_{L,\eps}$ as the smallest interval that contains all points
$$ a + \frac{\beta t}{4}\;,\quad t\in [0, \eps^{-2} L ]\;,\quad a\in \cM_{L,\eps}\;.$$
The parameter $\eps$ being fixed, there exists $C>0$ such that for all $L$ large enough
$$ \sup_{a\in \tilde{\cM}_{L,\eps}} |a - a_L| \le C a_L^{-1/2}\;.$$

For $a\in\R$ we set $a(t) := a + \beta t/4$. First, we state a bound on the probability that $Z_a$ remains close to the bottom of the well of its time-inhomogeneous potential.
\begin{lemma}\label{Lemma:RBM}
Fix $a >0$. For any $0<t_0<t_1$, any $0<d<D<\sqrt{a(t_0)}$ and any $x\in[\sqrt{a(t_0)}-d,\sqrt{a(t_0)}+d]$, we have
$$ \P\big(\exists t \in [t_0,t_1], Z_a(t) \notin [\sqrt{a(t_0)}-D,\sqrt{a(t)}+D] \, |\, Z_a(t_0) =x\big) \le 4 e^{-\frac{(D-d)^2}{2(t_1-t_0)}}\;.$$
\end{lemma}
\begin{proof}
Consider the reflected Brownian motion $R(t),t\ge t_0$ starting from $d$:
$$ dR(t) = dB(t) + d\ell(t)\;, \quad \int_{t\ge t_0} R(t) d\ell(t) = 0\;,\quad R(t_0) = d\;.$$
If $Z_a(t_0)$ lies in $[\sqrt{a(t_0)}-d,\sqrt{a(t_0)}+d]$, then $R(t) - (Z_a(t)-\sqrt{a(t)}) \ge 0$ for all $t\ge t_0$. Indeed, the inequality is satisfied at time $t_0$, and if this quantity vanishes at some time $t\ge t_0$ then either $R(t) = 0$ in which case we have
$$ d R(t) - d(Z_a(t) - \sqrt{a(t)}) = d\ell(t) + \frac{a'(t)}{2\sqrt{a(t)}} > 0\;,$$
or $R(t) > 0$ in which case
$$ d R(t) - d(Z_a(t) - \sqrt{a(t)}) = Z_a(t)^2 - a(t) + \frac{a'(t)}{2\sqrt{a(t)}} > 0\;.$$
Standard estimates on the reflected Brownian motion then show that
$$ \P\big(\sup_{t\in [t_0,t_1]} R(t) > D\big) \le 2 e^{-\frac{(D-d)^2}{2(t_1-t_0)}}\;.$$
A similar argument allows to control the probability that $Z_a$ crosses $\sqrt{a(t_0)} - D$.
\end{proof}

\subsection{Entrance, exit and return to the bottom of the well}
We start with the deterministic behavior of the diffusion $Z_a$ when it comes down from infinity and explodes. We denote by $\tau_x(Z_a)$ the first hitting time of $x$ by the process $Z_a$. Sometimes, we will simply write $\tau_x$.

\begin{lemma}\label{Lemma:Entrance}
Let $Z_a$ be the diffusion starting at time $0$ from $+\infty$. For any $b>0$, there exists $C=C(b)>0$ such that for all $L$ large enough, for all $a\in \tilde{\cM}_{L,\eps}$, with a probability at least $1-a_L^{-b}$ we have
$$ \sup_{t\in (0,\frac{\ln a}{\sqrt{a}}]} |Z_a(t) - \sqrt{a}\coth(\sqrt{a} t)| \le  C \frac{\ln a_L}{a_L^{1/4}} \;.$$
Similarly, let $Z_a$ be the diffusion starting at time $0$ from $-\sqrt{a}+(\ln a)^2 / a^{1/4}$. For any $b>0$, there exists $C=C(b)>0$ such that for all $L$ large enough, for all $a\in \tilde{\cM}_{L,\eps}$, with a probability at least $1-a_L^{-b}$ we have
$$ \sup_{t\in (0,\tau_{-\infty}]} |Z_a(t) - \sqrt{a}\coth(\sqrt{a}(t-\tau_{-\infty}))| \le C \frac{\ln a_L}{a_L^{1/4}} \;.$$
Finally, if $a\in\bar{\cM}_{L,\eps}$ then $\sqrt{a}\coth(\sqrt{a} \cdot)$ can be replaced by $\sqrt{a_L}\coth(\sqrt{a_L} \cdot)$ in the above estimates.
\end{lemma}
\begin{proof}
First of all, note that the derivative of the function $x\mapsto x \coth(x)$ is bounded on $\R$. Consequently there exists a constant $K>0$ such that
$$ \sup_{t > 0} \frac1{t} |\sqrt{a}t \coth(\sqrt a t)- \sqrt{a_L}t\coth(\sqrt{a_L} t)| \le K |\sqrt{a_L} - \sqrt{a}|\;.$$
This last term is of order $a_L^{-1}$ uniformly over all $a\in\bar{\cM}_{L,\eps}$. Consequently, the last part of the statement is proved.\\
The idea of the proof of the first part of the statement is very simple: when the process $Z_a$ is close to $\pm \infty$, the SDE that it solves is essentially deterministic. Let us provide the details in the case where $Z_a$ starts from $+\infty$. Consider the process $R(t) = Z_a(t)-B(t)$ and note that it solves
$$ dR(t) = a\, dt  - R(t)^2 \Big[ (1 + \frac{B(t)}{R(t)})^2 - \frac{\beta t}{4 R(t)^2}\Big] dt\;.$$
Fix $\ell \ge x:= \sqrt a + (\ln a)^2/a^{1/4}$ and $M=c \ln a / a^{1/4}$. Consider the solutions $F_1,F_2$ of
$$ dF_i(t) = (a- C_i F_i(t)^2)dt\;,\quad F_i(0)=+\infty\;,$$
with
$$ C_1 = (1 - \frac{M}{\ell-M})^2 - \frac{\beta}{4(\ell-M)^2} \frac{\ln a}{\sqrt a}\;,\quad C_2 = (1 + \frac{M}{\ell-M})^2\;.$$
Let $\cA := \{\sup_{t\in [0,\ln a / \sqrt a]} |B_t| \le M\}$ and note that $\P(\cA^\cc) \le 2a^{-c^2/2}$. On the event $\cA$ and as long as $Z_a$ is above $\ell$, we have
$$ F_2(t) \le R(t) \le F_1(t)\;.$$
Consequently
$$ F_2(t) - M \le Z_a(t) \le F_1(t) + M\;,\quad \forall t\in[0, \tau_{\ell}(Z_a) \wedge \frac{\ln a}{\sqrt a}]\;.$$
Note that $F_i(t) = \sqrt{a/C_i} \coth(\sqrt{aC_i} t)$. A straightforward computation shows that, in the case where $\ell =  x$, the first hitting times of $\ell$ by $F_2-M$ and $F_1+M$ both have the following expansion
$$ \frac38 \frac{\ln a}{\sqrt a} +  \frac1{2\sqrt a} \ln \frac2{(\ln a)^2} (1+o(1))\;.$$
As a consequence on the event $\cA$, for all $\ell \ge x$ we have  $\tau_{\ell}(Z_a) < \ln a / \sqrt a$ and
$$ F_2(t) - M \le Z_a(t) \le F_1(t) + M\;,\quad \forall t\in[0, \tau_{\ell}(Z_a)]\;.$$
Note that this implies that 
$$ \tau_{2\ell+M}(F_2)\le \tau_{2\ell}(Z_a) \le \tau_{\ell}(Z_a) \le \tau_{\ell-M}(F_1)\;.$$
One can check that
$$F_2(t)-M \le F_2(t) \le  \sqrt{a} \coth(\sqrt{a} t) \le F_1(t) \le  F_1(t)+M\;,\quad \forall t\ge 0\;.$$
Hence, for all $\ell \in [x,\infty)$
\begin{align*}
\sup_{t\in [\tau_{2\ell},\tau_{\ell}]}\big|Z_a(t) - \sqrt{a} \coth(\sqrt{a} t)\big| &\le \sup_{t\in[\tau_{2\ell+M}(F_2),\tau_{\ell-M}(F_1)]} \big|(F_1(t)+M) - (F_2(t) - M)\big|\\
&\le \sup_{t\in[\tau_{2\ell+M}(F_2),\tau_{\ell-M}(F_1)]} \Big|\frac{F_1(t)}{F_2(t)}-1\Big| F_2(t) + 2M\;.
\end{align*}
From the explicit expressions of $F_1$ and $F_2$, we obtain:
\begin{equation}\label{Eq:F1F2} \sup_{t>0}\Big|\frac{F_1(t)}{F_2(t)} - 1\Big| \lesssim \frac{M}{\ell}\;,\end{equation}
uniformly over all choices of $\ell \in [x,\infty)$. Therefore,
$$ \sup_{t\in [\tau_{2\ell+M}(F_2),\tau_{\ell-M}(F_1)]}\Big|\frac{F_1(t)}{F_2(t)} - 1\Big|F_2(t) \lesssim M\;,$$
so that
$$ \sup_{t\in [\tau_{2\ell},\tau_{\ell}]}\big|Z_a(t) - \sqrt{a} \coth(\sqrt{a} t)\big| \lesssim M\;,$$
uniformly over all choices of $\ell \in [x,\infty)$. Patching together these estimates, we get
$$ \sup_{t\in [0,\tau_{x}]}\big|Z_a(t) - \sqrt{a} \coth(\sqrt{a} t)\big| \lesssim M\;,$$

To complete the proof of the lemma, it remains to control the diffusion on the interval $[\tau_{x}(Z_a),\ln a / \sqrt a]$ and on the event $\cA$. To that end, we take $\ell = \sqrt{a}/2$ and, by the arguments at the beginning of the proof, we find
$$ F_2(t)-M \le Z_a(t) \le F_1(t)+M\;,\qquad t\in[0,\tau_{\ell}(Z_a)\wedge \frac{\ln a}{\sqrt a}]\;.$$
Note that $F_2(t) - M > \sqrt{a}/2$ for all $t>0$. We deduce that $\tau_{\ell}(Z_a)> \ln a / \sqrt a$. This and the estimate \eqref{Eq:F1F2} (which is also valid with the present choice of $\ell$) yield
$$ \sup_{t\in [0,\ln a /\sqrt a]}\big|Z_a(t) - \sqrt{a} \coth(\sqrt{a} t)\big| \lesssim M\;.$$
\end{proof}

The following lemma shows that, whatever point $Z_a$ starts from, with large probability it comes back within a short time to a neighborhood of the bottom of the well of its time-inhomogeneous potential.
\begin{lemma}\label{Lemma:Stabil}
For any $c>0$ there exists a constant $C>0$ such that the following holds for all $L$ large enough:
\begin{itemize}
\item Uniformly over all $a\in\tilde{\cM}_{L,\eps}$ and all $y\in (-\infty,+\infty]$, if $Z_a$ starts from $y$ at time $0$ then with a probability at least $1- Ca_L^{-2}$, it lies in the interval $[\sqrt{a}-c, \sqrt{a} + c]$ at time $\frac{(\ln a)^6}{\sqrt a}$,
\item Uniformly over all $a\in\bar{\cM}_{L,\eps}$ and all $y\in (-\infty,+\infty]$, if $Z_a$ starts from $y$ at time $0$ then with a probability at least $1 - C \frac{\ln\ln a_L}{{\ln a_L}}$, it lies in the interval $[\sqrt{a}-c, \sqrt{a} + c]$ at time $2 t_L$.
\end{itemize}
\end{lemma}
The proof of this lemma requires fine estimates on the behavior of $Z_a$ when it crosses the barrier of potential and is therefore postponed to Section \ref{Sec:Fine}. With this result at hand, we can prove the following.
\begin{lemma}\label{Lemma:XZSqueeze}
Let $X_a$ start from $+\infty$ and let $Z_a$ start from some $y\in (-\infty,+\infty]$. For any $c>0$ there exists a constant $C>0$ such that for all $L$ large enough the following has probability at least $1-Ca_L^{-2}$ uniformly over all $a\in \tilde{\cM}_{L,\eps}$ and all $y\in (-\infty,+\infty]$:
\begin{equation}\label{Eq:XaLSqueeze}
\inf_{t\in [0,10 \frac{\ln a}{\sqrt a}]} X_a(t) \ge \sqrt{a} - c \;,\quad  \sup_{t\in [\frac{\ln a}{\sqrt a},10 \frac{\ln a}{\sqrt a}]} X_a(t)\le \sqrt{a} + c\;,
\end{equation}
and
\begin{equation}\label{Eq:ZaLSqueeze}
\sqrt{a} - c \le Z_{a}(t) \le \sqrt{a} + c \;, \quad\forall t\in \Big[\frac{(\ln a)^6}{\sqrt a},\frac{(\ln a)^6}{\sqrt a}+10 \frac{\ln a}{\sqrt a}\Big]\;.
\end{equation}
\end{lemma}
\begin{proof}
The estimate on $Z_a$ follows from Lemmas \ref{Lemma:Stabil} and \ref{Lemma:RBM}: the cost in probability is bounded by $(C+1)a_L^{-2}$ where $C$ is the constant from Lemma \ref{Lemma:Stabil}. The estimate on $X_a$ follows from the counterparts of Lemmas \ref{Lemma:Entrance}, see~\cite[Lemma 4.2]{DL17}, and \ref{Lemma:RBM}, see~\cite[Lemma 4.4]{DL17}: the cost in probability is of the same magnitude.
\end{proof}

\subsection{Oscillations}

In this subsection, we collect estimates needed for the proof of ``Oscillations'' and ``Oscillations at infinity'' from Proposition \ref{Prop:TypicalZ}.

At several occasions, we will use the following argument. Let $I_1,\ldots,I_k$ be $k$ disjoint intervals, let $f$ be some function and fix $[a,b] \subset \R$. We have
\begin{equation}\label{Eq:Convexity}\fint_{I_j} f(s) ds \in [a,b]\;,\quad \forall j\in \{1,\ldots,k\}\qquad \Longrightarrow\qquad\fint_{I_1\cup \ldots\cup I_k} f(s) ds \in [a,b]\;.\end{equation}
Indeed, setting $I := I_1\cup \ldots\cup I_k$ we have
$$ \fint_{I} f(s) ds = \sum_{j=1}^k \frac{|I_j|}{|I|} \fint_{I_j} f(s) ds\;,$$
so by convexity the result follows.

\begin{lemma}\label{Lemma:Osc}
Fix $c>0$. We have as $L\to\infty$ uniformly over all $a\in \bar{\cM}_{L,\eps}$ and all $x\in [\sqrt{a_L}-1,\sqrt{a_L}+1]$
$$ \P_x\Big(\fint_0^t Z_a(s) \in [\sqrt{a_L}-c,\sqrt{a_L}+c]\;,\quad \forall t\in [0, \tau_{-2\sqrt{a_L}}(Z_a) \wedge \eps^{-2} L]\Big) \rightarrow 1\;.$$
\end{lemma}
\begin{proof}
For $n\ge 1$ introduce $s_i := i2^{-n}$ and let $I:=\min(i\ge 1: s_i > \eps^{-2})$. If we show that there exists $C'>0$ such that for all $n$ and for all $L$ large enough
$$  \sup_{i\le I}\sup_{a\in \bar{\cM}_{L,\eps}} \P_x(\fint_{s_i L}^t Z_a(s)ds \in [\sqrt{a_L}-c,\sqrt{a_L}+c]\;,\quad \forall t\in [s_i L, \tau_{-2\sqrt{a_L}}(Z_a) \wedge s_{i+1}L]) < C' 2^{-2n}\;,$$
then, by \eqref{Eq:Convexity}, we deduce the statement of the lemma. We will restrict ourselves to estimating the lower bound, namely $\fint_{s_i L}^t Z_a(s) ds \ge \sqrt{a_L}-c$. The upper bound follows from similar (and actually simpler) arguments.\\
So let us fix $i\ge 1$ and consider the time-homogeneous diffusion $X^i$ whose parameter $a$ is taken to be $a^i:=a + (\beta/4)s_i L - \frac1{4\sqrt{a_L}}$. From McKean's result recalled in Section \ref{Subsec:TimeHomo}, the probability that this diffusion explodes twice or more on $(s_i L, s_{i+1} L]$ is of order $2^{-2n}$. Similarly, the probability that this diffusion explodes on $(s_i L, (s_i+2^{-2n})L]$ is of order $2^{-2n}$: we can therefore exclude these two events in the sequel.\\
Until the end of the proof, we say that an event holds ``with large probability'' if its probability goes to $1$ as $L\to\infty$, uniformly over all parameters $a \in \bar{\cM}_{L,\eps}$ and $i\le I$.\\
Recall that $t_L = (\ln a_L)/\sqrt{a_L}$. Take $c' = c/8$. By Lemma \ref{Lemma:XZSqueeze} with large probability
$$ Z_a(t) \in [\sqrt{a_L}-c',\sqrt{a_L}+c']\;,\quad \forall t\in [s_i L, s_i L +9 t_L]\;.$$
Furthermore, by~\cite[Lemmas 4.3, 4.4. and 4.7]{DL17} with large probability we have
$$ \fint_{s_i L+(3/8) t_L}^t X^i(s) ds \in [\sqrt{a_L}-c',\sqrt{a_L}+c']\;,\quad \forall t\in [s_i L+(3/8) t_L, \tau_{-2\sqrt{a_L}}(X^i)]\;,$$
as well as $\tau_{-\infty}(X^i) - \tau_{-2\sqrt{a_L}}(X^i) \le t_L$. By the same computation as in the proof of Lemma \ref{Lemma:AL}, we can show that $X^i$ passes below $Z_a$ by time $s_i L +5 t_L$. Patching together these estimates we deduce that
$$ \fint_{s_i L}^t Z_a(s) ds \ge \sqrt{a_L}-3c'\;,\quad \forall t\in  [s_i L, \tau_{-2\sqrt{a_L}}(X^i)]\;.$$
If $\tau_{-2\sqrt{a_L}}(X^i) > s_{i+1} L$ then we are done. Otherwise $\tau_{-2\sqrt{a_L}}(X^i) \le s_{i+1} L$. By Lemma \ref{Lemma:Stabil}, with large probability $Z_a$ lies in $[\sqrt{a_L} - c',\sqrt{a_L}+c']$ at time $\tau_{-2\sqrt{a_L}}(X^i)+2t_L$. If $Z_a$ hits $-2\sqrt{a_L}$ on the interval of time $[\tau_{-2\sqrt{a_L}}(X^i), \tau_{-2\sqrt{a_L}}(X^i)+2t_L]$ then, using the fact that $\tau_{-2\sqrt{a_L}}(X^i) \ge \tau_{-\infty}(X^i) - t_L \ge s_i L +(1/2) 2^{-2n} L$ it is easy to check that
$$ \fint_{s_i^n L}^t Z_a(s) ds \ge \sqrt{a_L}-4c'\;,\quad \forall t\in  [s_i L, \tau_{-2\sqrt{a_L}}(Z_a)]\;,$$
and we are done.\\
If $Z_a$ does not hit $-2\sqrt{a_L}$ on the interval of time $[\tau_{-2\sqrt{a_L}}(X^i), \tau_{-2\sqrt{a_L}}(X^i)+2t_L]$, then it is also easy to check that
$$ \fint_{s_i L}^t Z_a(s) ds \ge \sqrt{a_L}-4c'\;,\quad \forall t\in  [s_i L, \tau_{-2\sqrt{a_L}}(X^i) + 2 t_L]\;.$$
Furthermore, one can iterate the above arguments starting from time $\tau_{-\infty}(X^i)$ and complete the proof.
\end{proof}

We now prove a result specific to the backward diffusions.
\begin{lemma}\label{Lemma:OscInfty}
Fix $C_0 >0$. There exists $c>0$ such that for all $\eps > 0$ small enough and all $L$ large enough the probability of the following event is larger than $1-c\eps$. For all $a\in {\cM}_{L,\eps}$, we have
$$ \fint_{\eps^{-2} L}^{t} \hat{Z}_a(s) ds \le -\frac12\sqrt{a_L}\;,\quad \forall t\in [2\eps^{-2}L, C_0 L \ln L]\;.$$
\end{lemma}
Note that $\sqrt{a+\frac{\beta}{4}t}$ can be larger than $\sqrt{a_L}$ for $a\in \cM_{L,\eps}$ and $t\in [2\eps^{-2}L, C_0 L \ln L]$, therefore we need to be careful in this proof at the current value of the bottom of well.
\begin{proof}
Since there are $\cO(\eps^{-2})$ elements in ${\cM}_{L,\eps}$, it suffices to prove that uniformly over all $a \in \cM_{L,\eps}$, the probability of the estimate of the statement is larger than $1-\cO(\eps^3)$. So we now fix $a\in \cM_{L,\eps}$. Consider the sequence
$$ s_i := e^i \eps^{-2} L\;,\quad i\ge 0\;,$$
and let $i_1$ be the smallest integer such that $s_{i_1} \ge C_0 L \ln L$. Note that $i_1 \le 2 \ln\ln L$. Set $a(s_i) := a + \frac{\beta s_i}{4}$. Assume that for every $i\in\{0,\ldots i_1-1\}$ we have
\begin{equation}\label{Eq:Bdsi}
\fint_t^{s_{i+1}} \hat{Z}_a(s) ds \in [-\sqrt{a(s_i)} - 10, -\sqrt{a(s_i)} + 10]\;,\quad \forall t \in [s_i,s_{i+1}]\;.
\end{equation}
Note that $a(s_i) \ge a_L$ for all $i$. Then, for any $t\in [2\eps^{-2}L, C_0 L \ln L]$, setting $i$ such that $s_i \le t < s_{i+1}$ we get
\begin{align*}
\int_{\eps^{-2} L}^{t} \hat{Z}_a(s) ds &= \int_{\eps^{-2} L}^{s_i} \hat{Z}_a(s) ds + \int_{s_i}^{s_{i+1}} \hat{Z}_a(s) ds - \int_{t}^{s_{i+1}} \hat{Z}_a(s) ds\\
&\le (-\sqrt{a_L}+10)(s_i-\eps^{-2} L)  + (-\sqrt{a(s_i)} + 10) (s_{i+1}-s_i)  - (-\sqrt{a(s_i)}-10)(s_{i+1}-t)\\
&\le (-\sqrt{a_L}+10)(t-\eps^{-2} L)  + 20(s_{i+1}-t)\;.
\end{align*}
Note that $t-\eps^{-2}L \ge \frac12 s_i$ so that for all $L$ large enough
$$ \frac14(-\sqrt{a_L}+10) (t-\eps^{-2} L)  + 20(s_{i+1}-t) \le -\frac1{16} \sqrt{a_L}s_i  +  20(e-1)s_i < 0\;.$$
Hence
$$ \int_{\eps^{-2} L}^{t} \hat{Z}_a(s) ds \le -\frac12 \sqrt{a_L} (t-\eps^{-2} L)\;,$$
as required. Consequently it suffices to evaluate the probability of \eqref{Eq:Bdsi}. We will only prove the upper bound since the lower bound is proved analogously. Fix $i\in \{0,\ldots,i_1-1\}$. Consider the backward time-homogeneous diffusion $\hat{X}$ whose parameter $a$ is given by $a(s_i)$ and that starts from $-\infty$ at time $s_{i+1}$. Let us now define the event $\cA_i$ on which
\begin{align*}
\hat{Z}_a(t) \le -\sqrt{a(s_i)} + 1\;,\quad \forall t\in [s_{i+1} - 10 t_L,s_{i+1} ]\;,\\
\hat{X}(t) \le -\sqrt{a(s_i)} + 1\;,\quad \forall t\in [s_{i+1} - 10 t_L,s_{i+1}]\;,\\
-\sqrt{a(s_i)} - 1 \le \fint_{t}^{s_{i+1} - (3/8)t_L} \hat{X}(s) ds \le -\sqrt{a(s_i)} + 1\;,\quad \forall t\in [s_{i}, s_{i+1} - (3/8)t_L]\;,
\end{align*}
and $\hat{X}$ does not hit $2\sqrt{a(s_i)}$ on $[s_{i}, s_{i+1})$.\\

Let us prove that the probability of $\cap_{i=0}^{i_1-1} \cA_i$ goes to $1$ as $L\to\infty$, uniformly over all parameters $a\in \cM_{L,\eps}$. To that end, it suffice to show that $i_1 \P(\cA_i^\cc)$ goes to $0$, uniformly over all parameters $a\in \cM_{L,\eps}$ and over all $i\in\{0,\ldots,i_1-1\}$. Recall that $i_1 \le 2\ln\ln L$. The two first estimates follow from (the backward version of) Lemma \ref{Lemma:XZSqueeze}: the cost in probability being of order $a_L^{-2}$. Provided that the fourth estimate is proven, the third estimate is a consequence of~\cite[Lemma 4.7]{DL17} whose cost in probability is of order $e^{-b(\ln\ln a_L)^2}$. We turn to the fourth estimate. The probability that $\hat{X}$ explodes on $[s_{i}-1, s_{i+1})$ is equal to
$$ \P(\gamma_{a(s_i)} \le s_{i+1} - s_i +1) \le \P(\frac{\gamma_{a(s_i)}}{m(a(s_i))} \le \frac{(e-1) s_i + 1}{m(a(s_i))})\;.$$
By \eqref{Eq:mac}, we have for some constant $C>0$
$$  \frac{(e-1)s_i +1}{m(a(s_i))} \le \frac{es_i}{m(a(s_i))} \le \frac{e^{i+1} \eps^{-2}L}{cm(a) e^{\frac12 e^i \eps^{-2}}} \le C e^{\eps^{-1}} e^{-\frac14 e^i \eps^{-2}}\;.$$
Consequently by Proposition \ref{Prop:CVrate}, a simple computation shows that, for some $i\in\{0,\ldots,i_1-1\}$, the probability that $\hat{X}$ explodes on $[s_{i}-1, s_{i+1})$ is less than $\eps$ for all $L$ large enough and all $\eps$ small enough, uniformly over all $a\in\cM_{L,\eps}$. By the analogue of Lemma \ref{Lemma:Entrance} for time-homogeneous diffusions, see~\cite[Lemma 4.2]{DL17}, we know that if $\hat{X}$ hits $2\sqrt{a(s_i)}$ on $[s_{i}, s_{i+1})$ then with large probability it explodes to $+\infty$ on $[s_{i}-1, s_{i+1})$. This concludes the proof of the fourth estimate.\\

We now work on the event $\cA_i$. By estimating the difference between $\hat{X}$ and $\hat{Z}_a$, as in the proof of Lemma \ref{Lemma:AL}, it is easy to deduce that $\hat{X}$ lies above $\hat{Z}_a$ at time $s_{i+1}-5 t_L$, and by monotonicity these two diffusions remain in this order on the interval $[s_i,s_{i+1} - 5 t_L]$. We thus deduce that for any $t\in [s_i , s_{i+1} - 10 t_L]$ we have
\begin{align*}
\int_{t}^{s_{i+1}  - 5 t_L} \hat{Z}_a(s) ds &\le \int_{t}^{s_{i+1}  - 5 t_L} \hat{X}(s) ds = \int_{t}^{s_{i+1}  - (3/8) t_L} \hat{X}(s) ds - \int_{s_{i+1}  - 5 t_L}^{s_{i+1} - (3/8) t_L} \hat{X}(s) ds\\
&\le (s_{i+1}-(3/8) t_L - t) (-\sqrt{a(s_i)} +1) - (5-(3/8))t_L (-\sqrt{a(s_i)} - 1)\\
&\le (s_{i+1}-5t_L-t) (-\sqrt{a_L} + 3)\;.
\end{align*}
On the other hand for any $t\in [s_{i+1} - 10 t_L,s_{i+1} - 5 t_L]$, the bound on $\hat{Z}_a$ stated in event $\cA_i$ yields
$$ \fint_{t}^{s_{i+1}  - 5t_L} \hat{Z}_a(s) ds \le - \sqrt{a_L} +1 \;.$$
Consequently by \eqref{Eq:Convexity} for any $t\in [s_i , s_{i+1} - 5 t_L]$ we have
$$ \fint_{t}^{s_{i+1}  - 5t_L} \hat{Z}_a(s) ds \le - \sqrt{a_L} +3 \;.$$
Again, the bound on $\hat{Z}_a$ stated in event $\cA$ yields for any $t\in [s_{i+1} - 5 t_L,s_{i+1}]$
$$ \fint_{t}^{s_{i+1}} \hat{Z}_a(s) ds \le - \sqrt{a_L} +1 \;.$$
Thus \eqref{Eq:Convexity} ensures that the upper bound in \eqref{Eq:Bdsi} holds.
\end{proof}

\subsection{Proof of Lemma \ref{Lemma:ApproxZZ}}\label{Subsec:ApproxZZ}

In this proof, $C_\eps$ denotes a (large) constant that only depends on $\eps$, it may change from line to line. By Theorem \ref{Th:Explo}, for $a=a_L - r/(4\sqrt{a_L})$ and $j\in\{0,\ldots,2^n-1\}$ the probability that $Z_{a}$ explodes more than once in $(t^n_j L,t^n_{j+1}L]$ converges to $(1-\exp(-2^{-n}e^{r})-2^{-n}e^{r}\exp(-2^{-n}e^{r}))$ as $L\to\infty$. Consequently, the probability that there exists $a\in\cM_{L,\eps}$ and $j$ such that $Z_{a}$ explodes more than once in $(t^n_j L,t^n_{j+1}L]$ is bounded by $C_\eps 2^{-n}$ uniformly over all $L$ large enough. This quantity goes to $0$ as $n$ goes to $\infty$. This proves the first part of the lemma.\\

We now prove that there exists a constant $C_\eps$ such that, for any given $j \in \{0,\ldots,2^n-1\}$ with a probability larger than $1-C_\eps 2^{-2n}$ for all $L$ large enough, the diffusion $Z_{a}$ explodes on $(t^n_j L, t^n_{j+1} L]$ if and only if the diffusion $Z_{a}^j$ explodes on this same interval. We first treat the case $j=2^n-1$. The probability that there exists $a\in \cM_{L,\eps}$ such that $Z_a$ explodes on $[t^n_{2^n-1}L,\infty)$ is bounded by a quantity of order $\eps^{-2} e^{\eps^{-1}} 2^{-n}$ as $L\to\infty$: this quantity vanishes as $n\to\infty$. Consequently, we can restrict ourselves to $j\in\{0,\ldots,2^{n}-2\}$ in the sequel.

Set $a(t^n_j L) := a +(\beta/4) t^n_jL$ and $a_-(t^n_j L) := a(t^n_j L) - 1/(4\sqrt{a(t^n_j L)})$. We introduce the time-homogeneous diffusion $X^j$ that starts from $+\infty$ at time $t^n_j L$ and whose parameter $a$ equals $a_-(t^n_j L)$. Define $\tilde{\tau}_{-2\sqrt{a(t^n_j L)}}$ and $\tilde{\tau}_{-\infty}$ as the first times after $t^n_j L$ at which $Z_a$ hits $-2\sqrt{a(t^n_j L)}$ and $-\infty$. We introduce $\kappa_j := \ln (a(t^n_j L)) / \sqrt{a(t^n_j L)}$ and the events
\begin{align*}
\cD_1 &:=\Big\{\sqrt{a(t^n_j L)} - (1/2) \le  Z_{a}(t) \le Z_{a}^j(t) \le  \sqrt{a(t^n_j L)} + (1/2)\;,\quad \forall t\in [t^n_j L + \kappa_j,t^n_j L + 9 \kappa_j]\Big\}\\
&\qquad\cap \Big\{\sqrt{a(t^n_j L)} - 1 \le  Z_{a}(t)\;,\quad \forall t\in [t^n_j L, t^n_j L+ \kappa_j]\Big\} \;,\\
\cD_2 &:=\Big\{\fint_{t^n_j L +  9 \kappa_j}^{t} Z_a(s)ds \in \Big[\sqrt{a(t^n_j L)} - 1,\sqrt{a(t^n_j L)} + 1\Big]\;,\quad \forall t\in[t^n_j L + 9 \kappa_j,\tilde{\tau}_{-2\sqrt{a(t^n_j L)}}]\Big\}\;,\\
\cD_3 &:=\{t_j^n L + 9 \kappa_j \le \tilde{\tau}_{-2\sqrt{a(t^n_j L)}}\}\cap\{ \tilde{\tau}_{-\infty}-\tilde{\tau}_{-2\sqrt{a(t^n_j L)}} \le \kappa_j\}\\
&\qquad\cap\{ {\tau}_{-\infty}(Z_a^j)-{\tau}_{-\sqrt{a(t^n_j L)}-1}(Z_a^j) \le \kappa_j\}\;,
\end{align*}
Finally let $\cD_4$ be the event on which $Z_a$ explodes at most once on $(t^n_j L,t^{n+1}_j L]$ and does not explode on $[(t^{n+1}_j-2^{-2n})L,t^{n+1}_j L]$. We then set $\cD:=\cap_i \cD_i$.\\

Let us evaluate the probability of $\cD$. By Lemmas \ref{Lemma:Entrance} and \ref{Lemma:XZSqueeze} and by monotonicity, $\P(\cD_1^\cc \cup \cD_3^\cc)$ goes to $0$ as $L\to\infty$ uniformly over all $j$. By Theorem \ref{Th:Explo} the probability of $\cD_4^\cc$ is bounded by $C_\eps 2^{-2n}$ uniformly over all $L$ large enough and all $j$. Regarding $\cD_2$, it suffices to apply Lemma \ref{Lemma:Osc}. Consequently $2^n \sup_{a\in \cM_{L,\eps}} \sup_j \P(\cD^\cc) \to 0$ as $L\to\infty$ and $n\to\infty$.\\

%

From now on, we work on the event $\cD$. If $\tilde{\tau}_{-\infty} > t^{n}_{j+1} L$, none of the diffusions explode on $(t^n_j L, t^{n+1}_j L]$. Otherwise, the diffusion $Z_a$ explodes before time $t^n_{j+1} L$ and we aim at showing that $Z_a^j$ explodes too. Consider the process $D(t) := Z^j_{a}(t) - Z_{a}(t)$ that solves
$$ dD(t) = -(Z^j_{a} + Z_{a})(t) D(t) dt\;.$$
Note that $D(t) = D(t_0) \exp(-\int_{t_0}^t (Z^j_{a} + Z_{a})(s) ds)$ for all $t_0 \le t$. By $\cD_2\cap\cD_3$, we find
$$ D(\tilde{\tau}_{-2\sqrt{a(t^n_j L)}}) \le D(t^n_j L + 9 \kappa_j)\;.$$
Then a simple computation based on $\cD_1$ shows that the last term is smaller than $1$ for all $L$ large enough. Consequently $Z^j_a$ is below $-2\sqrt{a(t^n_j L)} + 1$ at time $\tilde{\tau}_{-2\sqrt{a(t^n_j L)}}$. By $\cD_3$, it explodes within a time smaller than $\kappa_j$. By $\cD_3\cap \cD_4$, $\tilde{\tau}_{-2\sqrt{a(t^n_j L)}}$ is smaller than $(t^n_{j+1} L - 2^{-2n})L$ and therefore $Z^a_j$ explodes before time $t^n_{j+1} L$. This concludes the proof.

\subsection{Proof of Proposition \ref{Prop:TypicalZ}}\label{Subsec:ProofTypicalZ}

We start with the forward diffusions. By monotonicity, for any $a\in\cM_{L,\eps}$ the number of explosions of $Z_a$ is bounded by the number of explosions of $Z_{a_<}$ where $a_< = \min \cM_{L,\eps}$. From Theorem \ref{Th:Explo}, we deduce that there exists $C,N_\eps > 0$ such that the probability that $Z_{a_<}$ explodes more than $N_\eps$ times or explodes after time ${\eps^{-2}}L$ is bounded by $C \eps$ uniformly over all $L$ large enough. Consequently, in the sequel we only have to deal with the $N_\eps$ first explosions of the diffusions, and the Long-time behavior is proved.\\
To prove the Entrance and Explosion estimates, it suffices to iterate (at most $N_\eps$ times) Lemma \ref{Lemma:Entrance}. Regarding the Oscillation estimates, it suffices to combine the Entrance estimate with Lemma \ref{Lemma:Osc}.\\
Concerning the backward diffusions, the situation is the same except for the Oscillations at infinity for which we apply Lemma \ref{Lemma:OscInfty}.

\section{Crossing the barrier of potential}\label{Sec:Fine}

This section is devoted to a fine description of the diffusion $Z_a$ when it crosses the barrier of potential, and to the proof of Lemma \ref{Lemma:Stabil} and Proposition \ref{Prop:TypicalPairZ}

\subsection{Escaping the well}\label{Subsec:Crossing}

In this subsection, we collect several precise estimates on the trajectory of $Z_a$ when it escapes the bottom of the well of its time-inhomogeneous potential: these estimates will be the core of the proof of Proposition \ref{Prop:TypicalPairZ}.\\
Since the diffusion escapes the well in a very short time, the time-inhomogeneity of its drift is negligible and therefore its behavior is almost the same as that of the time-homogeneous diffusion $X_a$. The estimates stated in this subsection are therefore very close to those collected in~\cite[Section 5]{DL17} on $X_a$. Consequently, the proofs will make references to estimates obtained therein.\\
In the sequel, we denote by $\P^{(a)}_x$ the law of $Z_a$ starting from $x$ (in the proofs below, we will sometimes only write $\P_x$), and by $\tau_x$ the first hitting time of $x$ by $Z_a$. We also set (recall that $t_L = \ln a_L / \sqrt{a_L}$) :
$$ T := \frac34 t_L\;,\quad \delta := \frac{(\ln a_L)^2}{a_L^{1/4}}\;.$$

\begin{lemma}\label{Lemma:Cross1}
For any $c>0$, there exists $C>0$ such that for all $L$ large enough, for all $a\in \bar\cM_{L,\eps}$ we have
$$ \P^{(a)}_{\sqrt{a_L}-\delta}[E(C) \,|\, \tau_{-\sqrt{a_L}+\delta} < \tau_{\sqrt{a_L}-\delta/2} \wedge T] \le a_L^{-c}\;,$$
where $E=E(C)$ is defined by
\begin{align*}
E &= \Big\{ \sup_{t\in [0,\tau_{-\sqrt{a_L}+\delta}]} |Z_a(t) - \sqrt{a_L} \tanh(-\sqrt{a_L}(t-\tau_0))| \ge C \frac{\sqrt{a_L}}{\ln a_L} \Big\}\\
&\cup \Big\{|\tau_0 - \frac{3}{8} t_L| \ge C \frac{\ln \ln a_L}{\sqrt{a_L}}\Big\} \cup \Big\{|\tau_{-\sqrt{a_L}+\delta} - \tau_0 - \frac{3}{8} t_L| \ge C \frac{\ln \ln a_L}{\sqrt{a_L}}\Big\}\;.
\end{align*}
\end{lemma}
\begin{proof}
Let us write $\tau_-$ and $\tau_+$ as shortcuts for $\tau_{-\sqrt{a_L}+\delta}$ and $\tau_{\sqrt{a_L}-\delta/2}$. Consider the diffusion
$$ dH(t) = (-a+H^2(t)) + dB(t)\;,$$
and let $\bP_x$ be its law when it starts from $x$ at time $0$. The Radon-Nikodym derivative of $\P_x$ w.r.t.~$\bP_x$ up to time $t$ is given by $\exp(G_t(H))$ where
$$ G_t(H) = \frac23 (H_0^3 - H_t^3)  - 2 a(H_0 - H_t) + \frac{\beta}{4} t H_t  + (2- \frac{\beta}{4}) \int_0^t H_s ds - \frac{\beta}{4} \int_0^t (a + \frac{\beta}{8} s - H^2_s)s ds\;.$$
Consequently
\begin{align*}
\frac{\P_{\sqrt{a_L}-\delta}(E ; \tau_- < \tau_+ \wedge T)}{\P_{\sqrt{a_L}-\delta}(\tau_- < \tau_+ \wedge T)} &= \frac{\bP_{\sqrt{a_L}-\delta}(E ; \tau_- < \tau_+ \wedge T ; e^{G_{\tau_-}(H)})}{\bP_{\sqrt{a_L}-\delta}(\tau_- < \tau_+ \wedge T ; e^{G_{\tau_-}(H)})}\\
&\le e^{5 \sqrt{a_L} T} \frac{\bP_{\sqrt{a_L}-\delta}(E ; \tau_- < \tau_+ \wedge T )}{\bP_{\sqrt{a_L}-\delta}(\tau_- < \tau_+ \wedge T )}\;,
\end{align*}
where the last bound follows from an elementary computation performed on $G_{\tau_-}(H)$. The proof of~\cite[Lemma 5.1]{DL17} shows that for any $r>0$ we have
$$ \frac{\bP_{\sqrt{a_L}-\delta}(E ; \tau_- < \tau_+ \wedge T )}{\bP_{\sqrt{a_L}-\delta}(\tau_- < \tau_+ \wedge T )} \lesssim a_L^{-r}\;,$$
for all $L$ large enough. This concludes the proof.
\end{proof}

\begin{lemma}\label{Lemma:Cross2}
For any $c>0$, for all $L$ large enough and for all $a\in \bar\cM_{L,\eps}$ we have
$$ \P^{(a)}_{\sqrt{a_L}-\delta}[\tau_{-\sqrt{a_L}+\delta} > T \,|\, \tau_{-\sqrt{a_L}+\delta} < \tau_{\sqrt{a_L}-\delta/2}] \le a_L^{-c}\;.$$
\end{lemma}
\begin{proof}
Applying the same Girsanov transform as in the previous proof, one can apply the arguments in the proof of~\cite[Lemma 5.2]{DL17}.
\end{proof}

\begin{lemma}\label{Lemma:Cross3}
Take $C>1$ and set $S = C\ln\ln a_L / \sqrt{a_L}$. There exists $C'>0$ such that for all $L$ large enough and for all $a\in \bar\cM_{L,\eps}$ we have
$$ \P^{(a)}_{-\sqrt{a_L}+\delta}[\tau_{-\sqrt{a_L}} > S \wedge \tau_{-\sqrt{a_L}+2\delta} \,|\, \tau_{-\sqrt{a_L}} < \tau_{\sqrt{a_L}-\delta/2}] \le C' (\ln a_L)^{2-2C}\;,$$
\end{lemma}
\begin{proof}
For simplicity, we set $\tau_- :=\tau_{-\sqrt{a_L}}$, $\tau_+ :=  \tau_{-\sqrt{a_L}+2\delta}$ and $\tau_{++} := \tau_{\sqrt{a_L}-\delta/2}$. Set
$$ I(a) = \exp(\frac23 ((-\sqrt{a_L}+\delta)^3 -(-\sqrt{a_L})^3) - 2 a((-\sqrt{a_L} + \delta) - (-\sqrt{a_L})))\;.$$
and note that $\ln I(a)$ coincides with the sum of the two first terms of $G_{\tau_-}$. For $S' = \sqrt L$, we are going to show that as $L\to\infty$
\begin{align}
\P_{-\sqrt{a_L}+\delta}(\tau_- < S\wedge \tau_+) &\gtrsim I(a) e^{-2\sqrt a S}\;,\label{Eq:tau-+1}\\
\P_{-\sqrt{a_L}+\delta}(S\wedge \tau_+ \le \tau_- < S'\wedge \tau_{++}) &\lesssim I(a) e^{-2\sqrt a S} (\ln a_L)^{2-2C}\;,\label{Eq:tau-+2}\\
\P_{-\sqrt{a_L}+\delta}(\tau_{-} \wedge \tau_{++} > S') &\lesssim I(a) e^{-2\sqrt a S} (\ln a_L)^{2-2C}\;.\label{Eq:tau-+3}
\end{align}
These three bounds suffice to deduce the statement of the lemma. Indeed, the term on the l.h.s.~of the bound of the statement equals
\begin{align*}
\frac{\P_{-\sqrt{a_L}+\delta}(S\wedge \tau_+ \le \tau_- < \tau_{++})}{\P_{-\sqrt{a_L}+\delta}(\tau_- < \tau_+)} &\le \frac{\P_{-\sqrt{a_L}+\delta}(S\wedge \tau_+ \le \tau_- < S'\wedge\tau_{++})+\P_{-\sqrt{a_L}+\delta}(S' < \tau_{-}<\tau_{++})}{\P_{-\sqrt{a_L}+\delta}(\tau_- < S\wedge\tau_+)}\\
&\lesssim (\ln a_L)^{2-2C}\;.
\end{align*}

We start with \eqref{Eq:tau-+1}. Using the same Girsanov transform as before, we obtain
$$ \P_{-\sqrt{a_L}+\delta}(\tau_- < S\wedge \tau_+) = \bP_{-\sqrt{a_L}+\delta}(\tau_- < S\wedge \tau_+ ; e^{G_{\tau_-}})\;.$$
A simple computation shows that on the event $\tau_- < S\wedge \tau_+$ we have
$$ e^{G_{\tau_-}} \ge I(a) e^{-2\sqrt a S} (1+o(1))\;,$$
where $o(1)$ is a deterministic quantity that goes to $0$ as $L\to\infty$. In addition, it was shown in the proof of~\cite[Lemma 5.3]{DL17} that $\bP_{-\sqrt{a_L}+\delta}(\tau_- < S\wedge \tau_+)$ goes to $1$ as $L\to\infty$. This concludes the proof of \eqref{Eq:tau-+1}.

Regarding \eqref{Eq:tau-+2}, using again the Girsanov transform we get
$$ \P_{-\sqrt{a_L}+\delta}(S\wedge \tau_+ \le \tau_- < S'\wedge \tau_{++}) = \bP_{-\sqrt{a_L}+\delta}(S\wedge \tau_+ \le \tau_- < S'\wedge \tau_{++} ; e^{G_{\tau_-}})\;.$$
A simple computation shows that on the event $S\wedge \tau_+ \le \tau_- < S'\wedge \tau_{++}$ we have
$$ e^{G_{\tau_-}} \le I(a) (1+o(1)) e^{2 \int_0^{\tau_-} H(s) ds}\;,$$
where $o(1)$ is a deterministic quantity that goes to $0$ as $L\to\infty$. Moreover, it was shown in the proof of~\cite[Lemma 5.3]{DL17} that
$$ \bP_{-\sqrt{a_L}+\delta}(S\wedge \tau_+ \le \tau_- < S'\wedge \tau_{++} ; e^{2 \int_0^{\tau_-} H(s) ds} ) \lesssim (\ln a_L)^{2-2C}\;,$$
consequently \eqref{Eq:tau-+2} follows.

Finally, we prove \eqref{Eq:tau-+3}. To that end, we consider the time-homogeneous diffusion
$$ dX_a(t) = (a - X_a(t)^2) dt + dB(t)\;,$$
and we denote by $\Q_x$ its law when it starts from $x$. The Radon-Nikodym derivative of $\P_x$ w.r.t.~$\Q_x$ is given by $\exp(U_t(X))$ where
\begin{equation}\label{Eq:UtX}
U_t(X) = \frac{\beta}{4} \Big(t X_a(t) - \int_0^t [X_a(s)+s(a-X_a(s)^2)] ds\Big) - \frac{\beta^2}{96} t^3\;.
\end{equation}
Note that on the event $\tau_{-} \wedge \tau_{++} > S'$, the r.v.~$\exp(U_{S'}(X))$ is bounded by $2$ almost surely for all $L$ large enough. Henceforth
$$\P_{-\sqrt{a_L}+\delta}(\tau_{-} \wedge \tau_{++} > S') \le 2 \Q_{-\sqrt{a}+\delta}(\tau_{-} \wedge \tau_{++} > S') \le \frac{2}{S'} \Q_{-\sqrt{a}+\delta}(\tau_{-} \wedge \tau_{++})\;.$$
Using the classical formula for the expectation of the exit time from an interval for a diffusion, see for instance~\cite[Th VII.3.6]{RevuzYor}, one can show that
$$ \Q_{-\sqrt{a_L}+\delta}(\tau_{-} \wedge \tau_{++}) \le t_L\;.$$
Hence
$$\P_{-\sqrt{a_L}+\delta}(\tau_{-} \wedge \tau_{++} > S') \le 2 t_L L^{-1/2} \ll I(a) e^{-2\sqrt a S} (\ln a_L)^{2-2C}\;.$$
\end{proof}

The following lemma shows that if the diffusion $Z$ starts from $\sqrt{a_L}-\delta$ and hits $-\sqrt{a_L}$ before $\sqrt{a_L}$, then it does not hit $\sqrt{a_L}-\delta/2$ with large probability. Intuitively: if the diffusion is conditioned to cross the barrier of potential, then it does it right away. At a technical level, this estimate is easy to establish for the time-homogeneous diffusion thanks to an estimate on its scale function, see~\cite[Sec 5, proof of Prop 3.3]{DL17}. Here the situation is slightly more involved since the drift is time-inhomogeneous.

\begin{lemma}\label{Lemma:ImmediateDescent}
There exists $c>0$ such that for all $a\in\bar{\cM}_{L,\eps}$ and all $L$ large enough, we have
$$ \P^{(a)}_{\sqrt{a_L} -\delta}(\tau_{-\sqrt{a_L}} < \tau_{\sqrt{a_L}-\delta/2} \,|\, \tau_{-\sqrt{a_L}} < \tau_{\sqrt{a_L}}) \ge 1 - e^{-c(\ln a_L)^4}\;.$$
\end{lemma}
\begin{proof}
We set $\tau_- := \tau_{-\sqrt{a_L}}$, $\tau_+:= \tau_{\sqrt{a_L}-\delta/2}$ and $\tau_{++} := \tau_{\sqrt{a_L}}$. We have
\begin{align*}
\P_{\sqrt{a_L} -\delta}(\tau_{-} < \tau_{+} \,|\, \tau_{-} < \tau_{++}) &= \frac{\P_{\sqrt{a_L} -\delta}(\tau_{-} < \tau_{+} )}{\P_{\sqrt{a_L} -\delta}(\tau_{-} < \tau_{++})}\\
&=1 - \frac{\P_{\sqrt{a_L} -\delta}(\tau_{+} < \tau_{-} < \tau_{++})}{\P_{\sqrt{a_L} -\delta}(\tau_{-} < \tau_{++})}\;.
\end{align*}
We then bound separately the two terms in the fraction. First
\begin{align*}
\P_{\sqrt{a_L} -\delta}(\tau_{+} < \tau_{-} < \tau_{++}) &\le \P_{\sqrt{a_L} -\delta/2}(\tau_{-} < \tau_{++})\\
&\le \bQ_{\sqrt{a_L} - \delta/2}(\tau_{-} < \tau_{++})\;.
\end{align*}
The second inequality comes from the trivial coupling under which $X_a \le Z_a$ until the first explosion time of $X_a$.\\
Second, taking $S=L^{1/4}$ and using the expression of the Radon-Nikodym derivative \eqref{Eq:UtX} (which we bound from below by $1/2$) we get
\begin{align*}
\P_{\sqrt{a_L} -\delta}(\tau_{-} < \tau_{++}) &\ge \P_{\sqrt{a_L} -\delta}(\tau_{-} < \tau_{++} \wedge S)\\
&\ge (1/2) \bQ_{\sqrt{a_L}-\delta}(\tau_{-} < \tau_{++} \wedge S)\;.
\end{align*}
We claim that
$$ \bQ_{\sqrt{a_L}-\delta}(\tau_{-} < \tau_{++} \wedge S) \sim \bQ_{\sqrt{a_L}-\delta}(\tau_{-} < \tau_{++})\;.$$
With this claim at hand, we deduce that
\begin{align*}
\P_{\sqrt{a_L} -\delta}(\tau_{-} < \tau_{+} \,|\, \tau_{-} < \tau_{++}) \ge 1 - 4 \frac{\bQ_{\sqrt{a_L} - \delta/2}(\tau_{-} < \tau_{++})}{\bQ_{\sqrt{a_L}-\delta}(\tau_{-} < \tau_{++})}\;,
\end{align*}
so that an estimate in~\cite[Section 5 - Proof of Proposition 3.3]{DL17} shows that this is of order $1-\exp(-c(\ln a)^4)$ for some $c>0$.\\
We are left with proving the claim. First of all, by~\cite[Prop. VII.3.2]{RevuzYor} and a computation on the scale function one can prove that for any $\kappa > 0$ and for all $L$ large enough
\begin{equation}\label{Eq:EvalExitProba}
\bQ_{\sqrt{a_L}-\delta}(\tau_{-} < \tau_{++}) \ge L^{-1-\kappa}\;.
\end{equation}
Second, we have
$$ \sup_{y\in [-\sqrt{a_L} + \delta,\sqrt{a_L}]} \bQ_y(\tau_{++} > T) < a^{-1}\;.$$
Indeed if one starts the diffusion $X_a$ at any point in $[-\sqrt{a_L} + \delta,\sqrt{a_L}]$ then using a comparison with a deterministic ODE, on an event of probability at least $1-a^{-1}/2$, we can show that $X_a$ passes above $\sqrt{a_L} -\delta$ by time $T/2$, and then using a comparison with an Ornstein-Uhlenbeck process, one can show that it hits $\sqrt{a_L}$ within an additional time $T/2$ with probability at least $1-a^{-1}/2$.\\
Third, for any $\lambda \in (0,1)$
$$ \sup_{y\in [-\sqrt{a_L},-\sqrt{a_L}+\delta]} \bQ_y[e^{\lambda \tau_{-\sqrt{a_L}} \,\wedge\, \tau_{-\sqrt{a_L} + \delta}}] \le 2\;.$$
Indeed, let $Y = X_a + \sqrt a$. We have
$$ d|Y|(t) = |Y|(t) (2\sqrt a - Y(t))dt + dW(t) + d\ell(t)\;,$$
where $\ell$ is the local time of $Y$ at $0$ and $W$ is Brownian motion. We thus deduce that the first exit time of $[-\sqrt{a_L},-\sqrt{a_L} + \delta]$ by $X_a$ starting from $-\sqrt{a_L} +x$ is stochastically smaller than the first hitting time of $\delta$ by a reflected Brownian motion starting at $x+\sqrt a - \sqrt{a_L}$. Hence standard estimates on reflected Brownian motion yield the asserted (crude) estimate.\\

Consequently for $\lambda \in (0,1)$, using the Markov property at time $\tau_{-\sqrt{a_L}} \wedge \tau_{\sqrt{a_L}}\wedge T$ we get
$$ \sup_{y\in [-\sqrt{a_L}+\delta,\sqrt{a_L}]} \bQ_y[e^{\lambda \tau_{-\sqrt{a_L}} \wedge \tau_{\sqrt{a_L}}}] \le e^{\lambda T} + e^{\lambda T} a^{-1}G\;.$$
where
$$ G:= \sup_{y\in [-\sqrt{a_L},\sqrt{a_L}]} \bQ_y[e^{\lambda \tau_{-\sqrt{a_L}} \wedge \tau_{\sqrt{a_L}}}]\;.$$
Then,
$$\sup_{y\in [-\sqrt{a_L}+\delta,\sqrt{a_L}]} \bQ_y[e^{\lambda \tau_{-\sqrt{a_L}} \wedge \tau_{\sqrt{a_L}}}] \le e^{\lambda T} + e^{\lambda T} a^{-1} G\;,$$
and
\begin{align*}
&\sup_{y\in [-\sqrt{a_L},-\sqrt{a_L}+\delta]} \bQ_y[e^{\lambda \tau_{-\sqrt{a_L}} \wedge \tau_{\sqrt{a_L}}}]\\
&\le \sup_{y\in [-\sqrt{a_L},-\sqrt{a_L}+\delta]} \bQ_y[e^{\lambda \tau_{-\sqrt{a_L}} \wedge \tau_{-\sqrt{a_L} + \delta}}](1+\bQ_{-\sqrt{a_L} + \delta}[e^{\lambda \tau_{-\sqrt{a_L}} \wedge \tau_{\sqrt{a_L}}}])\;.
\end{align*}
Consequently, $G \le 4e^{\lambda T}(1+ a^{-1} G)$ so that
$$ G \le 4e^{\lambda T} \sum_{n\ge 0} (4e^{\lambda T} a^{-1})^n\;.$$
For $L$ |arge enough we thus get $G\le 8e^{\lambda T}$ and
$$ \bQ_{\sqrt{a_L} - \delta}[e^{\lambda \tau_{-\sqrt{a_L}} \wedge \tau_{\sqrt{a_L}}}] \le e^{\lambda T} + e^{\lambda T} a^{-1} G \le 2 e^{\lambda T}\;.$$
Therefore, we find
$$ \bQ_{\sqrt{a_L}-\delta}(S < \tau_{-\sqrt{a_L}} < \tau_{\sqrt{a_L}}) \le \bQ_{\sqrt{a_L}-\delta}(S < \tau_{-\sqrt{a_L}} \wedge \tau_{\sqrt{a_L}}) \le 2 e^{\lambda T} e^{-\lambda S}\;,$$
which is negligible compared to $\bQ_{\sqrt{a_L}-\delta}(\tau_{-\sqrt{a_L}} < \tau_{\sqrt{a_L}})$ thanks to \eqref{Eq:EvalExitProba}.
\end{proof}

\begin{lemma}\label{Lemma:aa'}
Take $\kappa \in (0,1)$. There exists $C>0$ such that the following holds for all $L$ large enough and for all $a\in\bar{\cM}_{L,\eps}$ with a probability at least $1- \cO(1/\ln a_L)$. Set $a' = a +\kappa$ and assume that $Z_a(0) = \sqrt{a_L} - \delta$ and that $Z_{a'}(0) \in (\sqrt{a_L} - \delta, 10 \sqrt{a_L})$. Conditionally given $\tau_{-\sqrt{a_L}}(Z_a) < \tau_{\sqrt{a_L}}(Z_a)$, we have:
\begin{align*}
|Z_{a'}(t)-Z_a(t)| \le 1\;,\quad &t\in [\upsilon_a - (1/16)t_L,\upsilon_a + (1/16)t_L]\;,\\
Z_{a'}(t) \le -\sqrt{a_L} + Ca_L^{3/7}\;,\quad &t\in [\upsilon_a + (1/16)t_L,\theta_a - (1/16)t_L]\;,\\
Z_{a'}(t) \le \sqrt{a_L} - 1\;,\quad &t\in [\theta_a - (1/16)t_L,\theta_a]\;.
\end{align*}
\end{lemma}
\begin{proof}
Consider the process $R(t) = Z_{a'}(t) - Z_a(t)$ and note that
$$ dR(t) =\kappa dt - R(t) (Z_a(t) + Z_{a'}(t))dt\;.$$
Using the estimates on the behavior of $Z_a$ collected in Lemmas \ref{Lemma:Cross1}, \ref{Lemma:Cross2},\ref{Lemma:Cross3} and \ref{Lemma:ImmediateDescent}, it is a straightforward computation to deduce the above estimates: actually, the same computation was performed in the proof of~\cite[Lemma 5.4]{DL17}.
\end{proof}

\subsection{From the unstable equilibrium point}\label{Subsec:Crossing2}

In this subsection, we collect estimates that we will need up to time $C_0L \ln L$: consequently we consider $a\in\tilde{\cM}_{L,\eps}$. In the estimates below, we let the diffusion start from a point at distance of order $a^{-1/4}$ from $-\sqrt{a}$: whenever $a\in \bar{\cM}_{L,\eps}$, this is equivalent with starting from a point at distance of order $a_L^{-1/4}$ from $-\sqrt{a_L}$ so that these estimates can be patched with those obtained in the previous subsection.

\begin{lemma}\label{Lemma:OscUnstable}
Fix $x\in \R$. Uniformly over all $a\in\tilde{\cM}_{L,\eps}$ we have the convergence
$$\P^{(a)}_{-\sqrt a + \frac{x}{2a^{1/4}}}(\tau_{-\sqrt a -\delta} < \tau_{-\sqrt a+\delta}) \to \P(\cN(0,1) > x)\;,\quad L\to\infty\;.$$
Furthermore, for all $L$ large enough, for all $a\in\tilde{\cM}_{L,\eps}$ and for any $s\in [0,L^{1/4}]$ we have
\begin{align*}
\P^{(a)}_{-\sqrt a + \frac{x}{2a^{1/4}}}(\tau_{-\sqrt a -\delta} > s \,|\, \tau_{-\sqrt a -\delta} < \tau_{-\sqrt a+\delta}) &\lesssim \frac{\ln\ln a}{s \sqrt a}\;.
\end{align*}
Finally, there exists $c>0$ such that for all $L$ large enough, for all $a\in\tilde{\cM}_{L,\eps}$ and for any $s\in [0,L^{1/4}]$ we have
$$ \sup_{x\in [-\sqrt a -\delta,-\sqrt a + \delta]} \P^{(a)}_x(\tau_{-\sqrt a -\delta} \wedge \tau_{-\sqrt a+\delta} > s) \lesssim \frac{\ln\ln a}{s \sqrt a} \wedge e^{-c s / \delta^2}\;.$$
\end{lemma}
\begin{proof}
Let $\Q_y$ be the law of $X_a$ starting from $y$. Using the scale function associated to the diffusion $X_a$, see~\cite[Section 4]{DL17}, we obtain
$$ \Q_{-\sqrt a + \frac{x}{2a^{1/4}}}(\tau_{-\sqrt a -\delta} < \tau_{-\sqrt a +\delta}) = \frac{\int_{-\sqrt a + \frac{x}{2a^{1/4}}}^{-\sqrt a +\delta} e^{2V_a(u)} du}{\int_{-\sqrt a -\delta}^{-\sqrt a +\delta} e^{2V_a(u)} du}\;,$$
where $V_a(u) =u^3/3 - au$. Writing $V_a(u) = V_a(-\sqrt a) - (u+\sqrt a)^2 \sqrt a + (u+\sqrt a)^3 /3$ and noticing that the cubic terms are negligible, we find
$$ \Q_{-\sqrt a + \frac{x}{2a^{1/4}}}(\tau_{-\sqrt a -\delta} < \tau_{-\sqrt a +\delta}) \sim \frac{\int_{\frac{x}{2a^{1/4}}}^{\delta} e^{-2\sqrt a u^2} du}{\int_{-\delta}^{\delta} e^{-2\sqrt a u^2} du} \to \P(\cN(0,1) > x)\;.$$
By~\cite[Lemma 5.7]{DL17}, we know that there exists $C>0$ such that for all $y\in [-\sqrt a -\delta,-\sqrt a + \delta]$
\begin{equation}\label{Eq:ExpectQy}
\Q_{y}(\tau_{-\sqrt a -\delta} \wedge \tau_{-\sqrt a+\delta} ) \le C \frac{\ln\ln a}{\sqrt a}\;.
\end{equation}
Furthermore, given the expression \eqref{Eq:UtX} of the Radon-Nikodym derivative of $\P_x$ w.r.t.~$\Q_x$, we have
$$ \P_{-\sqrt a + \frac{x}{2a^{1/4}}}(\tau_{-\sqrt a -\delta} < L^{1/4} \wedge \tau_{-\sqrt a+\delta}) = (1+o(1)) \Q_{-\sqrt a + \frac{x}{2a^{1/4}}}(\tau_{-\sqrt a -\delta} < L^{1/4} \wedge \tau_{-\sqrt a+\delta})\;,$$
$$ \P_{-\sqrt a + \frac{x}{2a^{1/4}}}(s \le \tau_{-\sqrt a -\delta} < L^{1/4} \wedge \tau_{-\sqrt a+\delta}) = (1+o(1)) \Q_{-\sqrt a + \frac{x}{2a^{1/4}}}(s \le \tau_{-\sqrt a -\delta} < L^{1/4} \wedge \tau_{-\sqrt a+\delta})\;,$$
and (note that we only have to compute the Radon-Nikodym derivative up to time $L^{1/4}$)
\begin{align*}
\P_{-\sqrt a + \frac{x}{2a^{1/4}}}(L^{1/4} < \tau_{-\sqrt a -\delta} < \tau_{-\sqrt a+\delta}) &\le \P_{-\sqrt a + \frac{x}{2a^{1/4}}}(L^{1/4} < \tau_{-\sqrt a -\delta} \wedge \tau_{-\sqrt a+\delta})\\
& \le(1+o(1)) \Q_{-\sqrt a + \frac{x}{2a^{1/4}}}(L^{1/4} < \tau_{-\sqrt a -\delta} \wedge \tau_{-\sqrt a+\delta})\\
& \le \frac{2C}{L^{1/4}}  \frac{\ln\ln a}{\sqrt a}\;\end{align*}
The two first bounds of the statement then follow by combining all these estimates and by using the Markov inequality on \eqref{Eq:ExpectQy}.\\
Regarding the third bound, given the expression \eqref{Eq:UtX} of the Radon-Nikodym derivative of $\P_x$ w.r.t.~$\Q_x$, we have for any $s\in [0,L^{1/4}]$
\begin{align*}
\P_x(\tau_{-\sqrt a -\delta} \wedge \tau_{-\sqrt a+\delta} > s) &= (1+o(1))\Q_x(\tau_{-\sqrt a -\delta} \wedge \tau_{-\sqrt a+\delta} > s)\;
\end{align*}
so that using \eqref{Eq:ExpectQy}, we deduce that
$$ \P_x(\tau_{-\sqrt a -\delta} \wedge \tau_{-\sqrt a+\delta} > s) \lesssim \frac{\ln\ln a}{s\sqrt a}\;.$$
Furthermore, we also have for any $\lambda > 0$
$$\P_x(\tau_{-\sqrt a -\delta} \wedge \tau_{-\sqrt a+\delta} > s) \le (1+o(1)) e^{-\lambda s} \Q_x(\exp(\lambda \tau_{-\sqrt a -\delta} \wedge \tau_{-\sqrt a+\delta}))\;.$$
To conclude, it suffices to compute the exponential moment on the r.h.s. Let $Y = X_a + \sqrt a$. We have
$$ d|Y|(t) = |Y|(t) (2\sqrt a - Y(t))dt + dW(t) + d\ell(t)\;,$$
where $\ell$ is the local time of $Y$ at $0$ and $W$ is Brownian motion. Consequently, the first exit time of $[-\sqrt a-\delta,-\sqrt a + \delta]$ by $X_a$ is stochastically smaller than the first exit time of a reflected Brownian motion from $[0,\delta]$: standard estimate yield for all $\lambda \in [0,\pi^2(8\delta^2)^{-1}]$
$$ \sup_{x\in [-\sqrt a-\delta,-\sqrt a + \delta]} \Q_x(\exp(\lambda \tau_{-\sqrt a -\delta} \wedge \tau_{-\sqrt a+\delta})) \le \frac1{\cos(\delta \sqrt{2\lambda})}\;,$$
thus concluding the proof.
\end{proof}

We now show that when $Z$ starts from $-\sqrt a + \delta$, with large probability it gets back to $\sqrt a$ within a time $\ln a / \sqrt a$.
\begin{lemma}\label{Lemma:DownTop}
For any $C > 1$, for all $L$ large enough and for all $a\in \tilde{\cM}_{L,\eps}$ we have 
$$ \P^{(a)}_{-\sqrt{a}+\delta}(\tau_{\sqrt{a}} < \tau_{-\sqrt{a}} \wedge C\frac{\ln a}{\sqrt a}) \ge 1 - 2a^{-2} - a^{-\frac32 (C-1)}\;.$$
\end{lemma}
\begin{proof}
Let $\kappa := \ln a / \sqrt a$. We first note that by monotony for $a'>a$ we have
$$ \P^{(a')}_{\sqrt a - \delta}(\tau_{\sqrt{a}} < (C-1)\kappa \wedge \tau_{\sqrt{a} -2\delta}) \ge \P^{(a)}_{\sqrt a - \delta}(\tau_{\sqrt{a}} < (C-1)\kappa \wedge \tau_{\sqrt{a} -2\delta})\;.$$
Applying the strong Markov property at time $\tau_{\sqrt{a}-\delta}$ we get
$$ \P^{(a)}_{-\sqrt a + \delta}(\tau_{\sqrt a} < C\kappa \wedge \tau_{-\sqrt a}) \ge \P^{(a)}_{-\sqrt a + \delta}(\tau_{\sqrt a-\delta} < \kappa \wedge \tau_{-\sqrt{a}}) \P^{(a)}_{\sqrt a - \delta}(\tau_{\sqrt{a}} < (C-1)\kappa \wedge \tau_{\sqrt{a} -2\delta})\;.$$
We are going to estimate the two factors on the r.h.s.~independently.\\
Regarding the first factor, set $R(t) = Z_a(t) - B(t)$ and note that
$$ dR(t) = \Big(a + \frac{\beta t}{4} - (R(t)+B(t))^2\Big)dt\;.$$
Consider the event $\cA:=\{\sup_{t \le \kappa} |B(t)| < M\}$ with $M=2\ln a / a^{1/4}$: this event has probability at least $1-2a^{-2}$. On the event $\cA$ and as long as the process $Z_a$ remains in $[-\sqrt a, +\sqrt a]$ we have
$$ dR(t) \ge a (1-\frac{3M}{\sqrt a}) - R^2(t) (1+ \frac{4 M}{\sqrt a})^2\;.$$
Indeed, if $|Z_a| \in [\frac12 \sqrt a,\sqrt a]$ we have
$$ a + \frac{\beta t}{4} - (R(t)+B(t))^2 \ge a - R(t)^2(1+ \frac{M}{\frac{\sqrt a}{2} - M})^2 \ge a - R(t)^2 (1 + \frac{4 M}{\sqrt a})^2\;.$$
While if $|Z_a|\le \sqrt a$ then
$$ a + \frac{\beta t}{4} - (R(t)+B(t))^2 \ge a - M^2 -2\sqrt a M - R(t)^2 \ge a(1- \frac{3M}{\sqrt a}) - R(t)^2\;.$$
Hence on the event $\cA$, we have $Z_a (t) \ge F(t) - M$ as long as $Z_a$ has not hit $\pm \sqrt a$, where $F$ is the solution of 
$$ dF(t) = a (1-\frac{3M}{\sqrt a}) - F^2(t) (1+ \frac{4 M}{\sqrt a})^2\;,\quad F(0) = -\sqrt{a} + \delta\;,$$
Simple computations show that
$$ F- M \ge -\sqrt a\;,\quad F(\ln a / \sqrt a) > \sqrt a - \delta/2\;.$$
We thus deduce that on the event $\cA$, the process $Z_a$ hits $\sqrt a - \delta$ by time $\kappa$ without hitting $-\sqrt a$.\\

We turn to the second factor. Let $A(t) = Z_{a}(t) - \sqrt{a}$ and note that
$$ dA(t) = a dt + \frac{\beta t}{4} dt - (A(t)+\sqrt{a})^2 dt + dB(t) \ge -A(t) (2\sqrt{a} + A(t)) + dB(t)\;.$$
Let $U$ be the Ornstein-Uhlenbeck process
$$ dU(t) = -2U(t)(\sqrt{a} - \delta) + dB(t)\;,\quad U(0) = -\delta\;.$$
Let $\tau_-$ and $\tau_+$ be the first hitting times of $\sqrt{a} - 2\delta$ and $\sqrt{a}$ by $Z_{a}$. Note that these stopping times coincide with the first hitting times of $-2\delta$ and $0$ by $A$. Note that until time $\tau_-\wedge \tau_+$, we have
$$ -A(t) (2\sqrt{a} + A(t)) \ge -2A(t) (\sqrt{a}-\delta)\;,$$
and therefore $A(t) \ge U(t)$. If we denote by $\bP$ the law of $U$, then
\begin{align*}
\P_{\sqrt a - \delta}^{(a)}(\tau_+ < (C-1)\kappa \wedge \tau_-)&\ge \bP(\tau_0 < (C-1)\kappa \wedge \tau_{-2\delta})\\
&\ge \bP(\tau_0 < (C-1)\kappa) - \bP(\tau_0 > \tau_{-2\delta})
\end{align*}
Using standard estimates on the Ornstein-Uhlenbeck process, see for instance~\cite[II.7.2.0.2 and II.7.2.2.2]{Handbook}, we deduce that
$$ \P^{(a)}_{\sqrt a - \delta}(\tau_+ < (C-1)\kappa \wedge \tau_-) \ge 1 - \cO(\delta a^{1/4} e^{-2(\sqrt a-\delta)(C-1)\kappa}) \ge 1 - a^{-\frac32 (C-1)}\;.$$
\end{proof}

\subsection{Proof of Lemma \ref{Lemma:Stabil}}

Fix $c>0$. We start with the first part of the statement. Let $a\in \tilde{\cM}_{L,\eps}$. We distinguish several cases according to the value $y$. First, if $y=+\infty$ then Lemma \ref{Lemma:Entrance} shows that with a probability at least $1-a_L^{-3}$ the process lies in $[\sqrt a - c/2,\sqrt a + c/2]$ at time $(3/8) \ln a /\sqrt{a}$. Second if $y\in [\sqrt{a}-c/2,\sqrt{a}+c/2]$ then Lemma \ref{Lemma:RBM} with a probability at least $1-a_L^{-3}$ the process remains in the strip $[\sqrt{a}-c,\sqrt{a}+c]$ until time $(\ln a)^6/\sqrt a$. Third if $y=-\sqrt a + \delta$ then Lemma \ref{Lemma:DownTop} shows that the diffusion comes back to $\sqrt{a}$ before time $10 \ln a / \sqrt a$ with probability at least $1-3a^{-2}$. Fourth if $y = -\sqrt a-\delta$, then Lemma \ref{Lemma:Entrance} shows that with a probability at least $1-a_L^{-3}$ the process explodes and comes back to $[\sqrt a - c/2,\sqrt a + c/2]$ by time $(3/4)t_L$.\\
By monotonicity and using Lemma \ref{Lemma:RBM}, we thus deduce the first statement of the lemma for any $y \notin [-\sqrt a - \delta, -\sqrt a + \delta]$. It remains to treat the case where $y$ lies in $[-\sqrt a - \delta, -\sqrt a + \delta]$. By Lemma \ref{Lemma:OscUnstable}, the diffusion exits the interval $[-\sqrt{a} - \delta,-\sqrt{a} + \delta]$ by time $(\ln a)^6/(2\sqrt a)$ with a probability of order $1-a^{-c'(\ln a_L)^2}$, and then, one can apply the estimates already established for $y= -\sqrt{a} \pm \delta$.\\

Let us now adapt the above argument to prove the second part of the statement. Let $a\in \bar{\cM}_{L,\eps}$. The two first cases remain unchanged. In the third case where $y=-\sqrt a + \delta$, we apply Lemma \ref{Lemma:DownTop} to show that the diffusion comes back to $\sqrt{a}$ before time $(6/5) \ln a / \sqrt a$ with probability at least $1-Ca_L^{-3/10}$. Finally if $y$ lies in $[-\sqrt a - \delta, -\sqrt a + \delta]$, then by Lemma \ref{Lemma:OscUnstable}, the diffusion exits the interval $[-\sqrt{a} - \delta,-\sqrt{a} + \delta]$ by time $\ln a/(3\sqrt a)$ with a probability of order $1-\cO(\ln\ln a_L / \ln a_L)$, and then, one can apply the estimates already established for $y= -\sqrt{a} \pm \delta$.

\subsection{Proof of Proposition \ref{Prop:TypicalPairZ}}\label{Subsec:ProofTypicalPairZ}

We need a last lemma before we proceed with the proof of Proposition \ref{Prop:TypicalPairZ}. It turns out that the process $X_a$ introduced in Section \ref{Section:Explo} admits an explicit invariant measure, we refer to~\cite[Section 4.1]{DL17}. We will call stationary time-homogeneous diffusion a process $X_a$ that starts from the invariant measure.

\begin{lemma}\label{Lemma:4}
Fix $\eps > 0$. For all $L$ large enough and all $n$ large enough, for all $a \le a'\in\cM_{L,\eps}$, for all $j\in\{0,\ldots,2^n-1\}$ such that $t^n_{j+1} < \eps^{-2}$, with a probability at least $1-(\ln a_L)^{-1/2}$ the following holds. If $\theta^j_a < t^n_{j+1}L$ and if there exists $a'' \in \cM_{L,\eps}$ such that $a'' < a$ and $\hat{Z}_{a''}$ does not explode on $[\theta^j_a + 10 t_L,t^n_{j+1}L]$ then $\hat{Z}_{a'}(t) \le -\sqrt{a_L} + (\ln a_L)/a_L^{1/4}$ for all $t\in [\theta_a^j,\theta_a^j+ 5t_L]$, and furthermore for all $t\in [\theta_a^j,t^n_{j+1} L]$ we have
\begin{align*}
-(3/2) \sqrt{a_L} \le  \fint_{\theta_a^j}^t \hat{Z}_{a'}(s) ds  \le -(1/2) \sqrt{a_L}\;.
\end{align*}
\end{lemma}
\begin{proof}
The proof is an adaptation of~\cite[Lemma 3.6]{DL17}. We abbreviate $\theta_a^j$ in $\theta$. The r.v.~$\theta\wedge t^n_{j+1}L$ is a stopping time in the filtration $\cF_t,t\ge 0$ of the underlying Brownian motion $B$. By the strong Markov property, the process $(B(t+\theta\wedge t^n_{j+1}L)-B(\theta\wedge t^n_{j+1}L),t\ge 0)$ is a standard Brownian motion, independent from $\cF_{\theta\wedge t^n_{j+1}L}$. Hence, conditionally given $\theta\wedge t^n_{j+1}L$, the process $(\hat{Z}_{a'}(t),t\in [\theta\wedge t^n_{j+1}L,t^n_{j+1} L])$ has the law of the backward diffusion.\\
We introduce a stationary time-homogeneous diffusion $\hat{Y}$ driven by $\hat B$ and whose parameter $a$ is taken to be
$$ a^Y := \frac{a' + \frac{\beta}{4} t^n_j L + a'' + \frac{\beta}{4} t^n_{j+1} L}{2}\;.$$
Note that for $n$ large enough w.r.t.~$\eps$ we have
$$(a' + \frac{\beta}{4} t^n_j L) - (a'' + \frac{\beta}{4} t^n_{j+1} L) \ge \frac{\eps}{8 \sqrt{a_L}}\;.$$

We introduce the event
$$\mathcal{A} := \Big\{\hat{Y}(t) \leq -\sqrt a + \frac{(\ln a)^2}{a^{1/4}} \mbox{ for all }t\in[\theta\wedge (t^n_{j+1} L),(\theta + 11 t_L)\wedge (t^n_{j+1} L)]\Big\}\;.$$
As in the proof of~\cite[Lemma 3.6]{DL17}, we can check that $\P(\cA) \ge 1 - (\ln a_L)^{-1}$.\\
Then, we define $\cB$ as the event on which $\hat{Z}_{a'}$, $\hat{Z}_{a''}$ and $\hat{Y}$ lie in $[-\sqrt{a_L}-1,-\sqrt{a_L}+1]$ on $[t^n_{j+1}L-10 t_L,t^n_{j+1}L]$. By Lemma \ref{Lemma:XZSqueeze} and by~\cite[Lemma 4.1]{DL17}, the probability of $\cB$ is at least $1-(\ln a_L)^{-1}$. By the same computation as in the proof of Lemma \ref{Lemma:AL} one can check that on $\cB$ we have
\begin{equation}\label{Eq:a'Ya''}
\hat{Z}_{a'}(t) \le \hat{Y}(t) \le \hat{Z}_{a''}(t)\;,\quad \forall t\in[q'', t^n_{j+1}L - 10 t_L]\;,
\end{equation}
where $q'':= \sup\{t\le t^n_{j+1} L:  \hat{Z}_{a''}(t) = +\infty\}$, and
\begin{equation}\label{Eq:a'Y}
\hat{Z}_{a'}(t) \le \hat{Y}(t) \;,\quad \forall t\in[q, t^n_{j+1}L - 10 t_L]\;,
\end{equation}
where $q:= \sup\{t\le t^n_{j+1} L:  \hat{Y}(t) = +\infty\}$.\\
Define now the event $\cC$ on which
$$ \fint_{\theta\wedge t^n_{j+1} L}^{t} \hat{Y}(s) ds \le -\sqrt a + \frac{(\ln a_L)^2}{2a_L^{1/4}}\;,\quad \forall t\in [\theta\wedge t^n_{j+1} L,t^n_{j+1} L]\;.$$
It is shown in the proof of~\cite[Lemma 3.6]{DL17} that $\P(\cC) > 1-\exp(-c(\ln a_L)^2)$ for some $c>0$.\\

We now work on the event $\cA\cap \cB \cap \cC$. If $\hat{Z}_{a''}$ does not explode on $[\theta + 10 t_L,t^n_{j+1}L]$ and if $\theta < t^n_{j+1} L$ then we claim that $\hat Y$ does not explode on $[\theta,t^n_{j+1} L]$. Indeed, by $\cB$ and \eqref{Eq:a'Ya''}, $\hat{Y}$ explodes ``after'' $\hat{Z}_{a''}$ and since $\hat{Z}_{a''}$ does not explode on $[\theta + 10 t_L,t^n_{j+1}L]$, we deduce that $\hat{Y}$ would explode on $[\theta,\theta + 10 t_L]$. But this would raise a contradiction with $\cA$.\\
Consequently by \eqref{Eq:a'Y} we have:
$$\hat{Z}_{a'}(t) \le -\sqrt{a}+\frac{(\ln a_L)^2}{a_L^{1/4}}\;,\quad \forall t\in [\theta,\theta+10\, t_L]\;,$$
so that $\hat{Z}_{a'}$ does not explode on $[\theta,t^n_{j+1} L]$. Moreover the bound of event $\cC$ combined with the condition of event $\cB$ yields
$$ \sup_{t\in [{\theta},t^n_{j+1} L]} \fint_{\theta}^t \hat{Z}_{a'}(s) ds  \le -(1/2) \sqrt{a_L}\;.$$
The proof of the lower bound of the statement can be carried out using similar (and actually simpler) arguments.
\end{proof}

\begin{proof}[Proof of Proposition \ref{Prop:TypicalPairZ}]
An adaptation of Theorem \ref{Th:Explo} shows that, for any $a\in \cM_{L,\eps}$, if we write $a= a_L - r/(4\sqrt{a_L})$ and if we let $(t_i)_{i\ge 1}$ be the starting times (in the increasing order) of the successive excursions to $-\sqrt{a_L}$ of $Z_a$ then the point process $(t_i / L)_{i\ge 1}$ restricted to $[0,\eps^{-2} L]$ converges in law to a Poisson point process of intensity $2e^r e^{-t} dt$. The intensity is twice that of the Poisson point process that appears in Theorem \ref{Th:Explo}: this is a consequence of Lemma \ref{Lemma:OscUnstable} since this result shows that, with a probability going to $1/2$, over an excursion to $-\sqrt{a_L}$ the process $Z_a$ explodes.\\
We thus deduce that the probability that there exists an interval $(t^n_jL,t^n_{j+1}L]$ with $t^n_{j+1} < \eps^{-2}$ and some $a\in \cM_{L,\eps}$ such that the diffusion $Z_a$ makes at least two excursions to $-\sqrt{a_L}$ is of order $\cO(2^{-n})$ uniformly over all $L$ large enough. Therefore we can assume from now on that there are at most one excursion to $-\sqrt{a_L}$ on every interval $(t^n_jL,t^n_{j+1}L]$ with $t^n_{j+1} < \eps^{-2}$.\\

From now on, we fix an interval $(t^n_jL,t^n_{j+1}L]$ with $t^n_{j+1} < \eps^{-2}$. We will write ``with large probability'' to say that an event holds with a probability that goes to $1$ as $L\to\infty$ uniformly over all $j$ and all $a$. By Lemma \ref{Lemma:XZSqueeze} and the same computation as in the proof Lemma \ref{Lemma:AL}, we deduce that with large probability, for all $a\le a' \in \cM_{L,\eps}$, $Z_a(t^n_j L)$ lies in $[(1/2 )\sqrt{a_L}, (3/2) \sqrt{a_L}]$ and $Z_a(t^n_j) \le Z_{a'}(t^n_j)$.\\
Property (4) is a consequence of Lemma \ref{Lemma:4}. Regarding Properties (1), (2) and (3), we argue as follows. For any $a\in\cM_{L,\eps}$, after its first hitting time of $\sqrt{a_L}-\delta$ we decompose the trajectory of $Z_a$ into two types of bridges:\begin{itemize}
\item \emph{Type I}: Bridges that start from $\sqrt{a_L}-\delta$, hit $\sqrt{a_L}$ before $-\sqrt{a_L}$ and then come back to $\sqrt{a_L}-\delta$,
\item \emph{Type II}: Bridges that start from $\sqrt{a_L}-\delta$, hit $-\sqrt{a_L}$ before $\sqrt{a_L}$, and then come back to $\sqrt{a_L}-\delta$ (possibly after an explosion).
\end{itemize}
We already know that there are at most one bridge of Type II on $(t^n_jL,t^n_{j+1}L]$. By Lemma \ref{Lemma:ImmediateDescent} we deduce that, with large probability, if there is a bridge of Type II on $(t^n_jL,t^n_{j+1}L]$ then it does not hit $\sqrt{a_L}-\delta/2$ before $-\sqrt{a_L}$. The estimates stated in Lemma \ref{Lemma:Cross1}, \ref{Lemma:Cross2}, \ref{Lemma:Cross3} and \ref{Lemma:aa'} then yield Properties (1), (2) and (3) of the statement.\end{proof}

\bibliographystyle{Martin}
\bibliography{library}

\end{document}